\documentclass[11pt,reqno]{amsart}

\usepackage{geometry}
\usepackage{tikz-cd}
\usepackage{hyperref}
\usepackage{xcolor} %
\usepackage[abbrev]{amsrefs}
\usepackage[utf8]{inputenc}
\usepackage[cyr]{aeguill}
\usepackage[all,color]{xy}
\usepackage[french, english]{babel}
\usepackage[T1]{fontenc}
\usepackage{amscd}  
\usepackage{amssymb}
\usepackage{titlesec}
\usepackage{titletoc}
\usepackage{enumitem}
\usepackage{mathtools} %

\newcommand{\blue}[1]{{\color{blue}#1}}

\theoremstyle{plain}
\newtheorem{theo}{Theorem}[section]
\newtheorem{lemm}[theo]{Lemma}
\newtheorem{prop}[theo]{Proposition}
\newtheorem{coro}[theo]{Corollary}
\newtheorem*{theo*}{Theorem}
\newtheorem*{prop*}{Proposition}

\theoremstyle{definition}
\newtheorem{defi}[theo]{Definition}
\newtheorem*{defi*}{Definition}
\newtheorem{nota}[theo]{Notation}
\newtheorem*{Nota*}{Notation}

\theoremstyle{remark}
\newtheorem{ex}[theo]{Example}
\newtheorem*{ex*}{Example}
\newtheorem{rema}[theo]{Remark}
\newtheorem{warning}[theo]{Warning}

\makeatletter \@addtoreset{equation}{section} \@addtoreset{equation}{subsection} \makeatother
\titleformat{\section}
     {\center\Large\scshape}{\thesection.}{1em}{}
 \titleformat{\subsection}
     {\center\large\scshape}{\thesubsection.}{1em}{}[]
 \titleformat{\subsubsection}
     {\scshape}{\thesubsubsection.}{1em}{}

\newcommand{\Def}[1]{\textbf{\boldmath{#1}}} %
\newcommand{\dfn}{\coloneqq} %
\newcommand{\rescale}[1]{\newline\resizebox{\linewidth}{!}{\begin{minipage}{\linewidth}
#1
\end{minipage}}\vspace{\baselineskip}\\}

\newcommand{\be}{\beta}
\newcommand{\ep}{\epsilon}
\newcommand{\Ga}{\Gamma}
\newcommand{\la}{\lambda}
\newcommand{\Om}{\Omega}

\newcommand{\si}{\sigma}

\newcommand{\te}{\theta}

\newcommand{\F}{\mathbb{F}}
\newcommand{\LL}{\mathbf{L}} %
\newcommand{\N}{\mathbb{N}}
\newcommand{\Z}{\mathbb{Z}}

\newcommand{\op}{\oplus}
\newcommand{\ot}{\otimes}
\newcommand{\x}{\times}

\newcommand{\bu}{\bullet}

\newcommand{\ul}{\underline}

\def\scr{\mathcal}                    
\def\tcirc{\,\tilde\circ\,}

\def\diag{\shorthandoff{;:!?}
\xymatrix}
\def\0{\mathbf 0}
\def\_{\underline}
\def\t{\mathfrak t}

\def\Q{\scr Q}

\def\ı{\mathcal N}
\def\r{{\_r}}
\def\q{{\_q}}
\def\C{\scr C}
\def\∑{\Sigma} %
\newcommand{\Sy}{\Sigma} %

\newcommand{\A}{\mathcal{A}}
\renewcommand{\P}{\mathcal{P}}
\newcommand{\ZZ}{\mathcal{Z}}
\newcommand{\CC}{\mathcal{C}}
\newcommand{\UU}{\mathbb{U}}
\newcommand{\ag}[1]{#1^+} %

\newcommand{\inj}{\hookrightarrow}
\newcommand{\ral}{\xrightarrow} %
\newcommand{\rla}{\rightleftarrows}
\newcommand{\rra}{\rightrightarrows}
\newcommand{\surj}{\twoheadrightarrow}

\newcommand{\Alg}[1]{#1\text{-}\mathrm{Alg}} %
\newcommand{\As}{\operatorname{As}}       
\newcommand{\uAs}{\operatorname{As}_+}       
\newcommand{\cat}[1]{\mathbf{\mathcal{#1}}} %
\newcommand{\Com}{\mathrm{Com}}

\newcommand{\pCom}{\mathrm{p}\mbox{-}\mathrm{Com}}
\newcommand{\Lie}{\operatorname{Lie}}
\newcommand{\RLie}{\operatorname{RLie}}
\newcommand{\Modd}[1]{\mathrm{Mod}_{#1}} %
\newcommand{\lMod}[1]{#1\text{-}\mathrm{Mod}} %
\newcommand{\Set}{\mathrm{Set}}
\newcommand{\Vect}[1]{\mathrm{Vect}_{#1}} %

\newcommand{\ab}{\mathrm{ab}}
\newcommand{\aug}{\mathrm{aug}}
\newcommand{\HH}{\mathrm{HH}} %
\newcommand{\HQ}{\mathrm{HQ}} %
\newcommand{\id}{\mathrm{id}}
\newcommand{\pr}{\mathrm{pr}} %
\newcommand{\Tr}{\mathrm{Tr}}
\newcommand{\triv}{\mathrm{triv}}

\DeclareMathOperator{\Ab}{Ab}
\DeclareMathOperator{\Comp}{Comp}

\DeclareMathOperator{\Der}{Der}

\DeclareMathOperator{\Hom}{Hom}
\newcommand{\Prim}{\mathrm{Prim}}

\dottedcontents{section}[1.5em]{}{1.5em}{1pc}

\begin{document}

\title{Quillen (co)homology of divided power algebras over an operad}
\date{\today}

\author{Ioannis Dokas} 
\address{National and Kapodistrian University of Athens}
\email{iodokas@math.uoa.gr}

\author{Martin Frankland} 
\address{University of Regina}
\email{Martin.Frankland@uregina.ca}

\author{Sacha Ikonicoff}
\address{Université de Strasbourg}
\email{ikonicoff@unistra.fr}

\begin{abstract}
Barr--Beck cohomology, put into the framework of model categories by Quillen, provides a cohomology theory for any algebraic structure, for example Andr\'e--Quillen cohomology of commutative rings. Quillen cohomology has been studied notably for divided power algebras and restricted Lie algebras, both of which are instances of divided power algebras over an operad $\P$: the commutative and Lie operad respectively. In this paper, we investigate the Quillen cohomology of divided power algebras over an operad $\P$, identifying Beck modules, derivations, and K\"ahler differentials in that setup. We also compare the cohomology of divided power algebras over $\P$ with that of $\P$-algebras, and work out some examples.
\end{abstract}



\maketitle

\tableofcontents

\section{Introduction}\label{sec:Intro}

Barr--Beck cohomology is a cohomology theory for general algebraic structures, based on simplicial resolutions \cite{BarrB69}. 
It recovers (up to a shift in degree) group cohomology, Chevalley--Eilenberg cohomology of Lie algebras, and Hochschild cohomology of algebras over a field; see for instance \cite{Barr96}. 
One important example is Andr\'e--Quillen cohomology of commutative rings \cites{Andre67,Andre74}, put into the framework of model categories by Quillen \cites{Quillen67,Quillen70}, which has various applications in algebra and algebraic geometry. Quillen cohomology later found applications in topology, often in the guise of inputs to spectral sequences \cites{Miller84,Goerss90and} or obstructions in %
realization and classification problems 
\cites{GoerssH00,GoerssH04spa,BlancDG04,Frankland11,BauesB11,BlancJT12}.

\subsubsection*{Divided power algebras}
%\paragraph{Divided power algebras}

Divided power structures are algebraic structures which are fairly ubiquitous when working over a field of positive characteristic. N.~Jacobson introduced the concept of restricted Lie algebra to study modular Lie theory \cite{Jacobson62}. A restricted Lie algebra is a Lie algebra equipped with an additional operation called the $p$-map, which satisfies specific relations. The archetypal example of a restricted Lie algebra is an associative algebra in prime characteristic equipped with the Lie bracket given by the commutator,  %
and with the $p$-map given by the $p$-th power. 
Restricted Lie algebras appear notably %
in field theory \cite{Jacobson37}, linear algebraic groups \cite{Borel91}, and the cohomology of the Steenrod algebra \cite{May66res}.

G.~Hochschild, and later B.~Pareigis, defined cohomology groups of %
restricted Lie algebras \cites{Hochschild54,Pareigis68}. In \cite{Dokas04}, the first author developed the %
Quillen cohomology theory for restricted Lie algebras. In contrast to the cases of associative algebras and Lie algebras, this cohomology does not coincide with Hochschild cohomology for restricted Lie algebras. In \cite{Dokas15}, $2$-fold extensions of restricted Lie algebras are classified using Duskin's and Glenn's torsor cohomology \cites{Duskin75,Glenn82} and the work of Cegarra and Aznar \cite{CegarraA86}.

The notion of algebra with divided powers was introduced by H.~Cartan in \cite{Cartan56exp7} to study certain bar constructions of commutative algebras, %
and further 
developed by N.~Roby \cite{Roby65}. A divided power algebra %
is a commutative algebra endowed with additional operations $\gamma_{n}$ which satisfy specific relations. An algebra $A$ over $\mathbb{Q}$  is naturally equipped with a structure of divided powers algebra such that $\gamma_{n}(x)=\frac{x^{n}}{n!}$ for all $x\in A$. Without any restriction on the ground ring, H.~Cartan proves in \cite{Cartan56exp7}  that the homotopy of a simplicial commutative algebra comes naturally equipped with the structure of a divided power algebra. %
Divided power algebras are essential in the theory of crystalline cohomology for schemes introduced by A.~Grothendieck %
\cite{Grothendieck68} 
and developed by P.~Berthelot \cite{Berthelot74}.
The first author studied %
Quillen 
cohomology of %
divided power algebras \cites{Dokas09,Dokas23}.

\subsubsection*{Operads}
%\paragraph{Operads}

Operads are an algebraic device which represent types of algebras, such as associative algebras, commutative algebras, and Lie algebras. The theory of operads has proven to be a powerful tool in studying different categories of algebras in a unified way. For instance, in the seminal book \cite{LodayV12}, J.-L.~Loday and B.~Vallette make a detailed account of many known homological and homotopical theories for algebras over an operad. This includes %
Quillen (co)homology for algebras over an operad, studied in \cites{Livernet98,Fresse98,Milles11,GoerssH00}. In the quadratic case, this is studied through a %
version of Koszul duality due to V.~Ginzburg and M.~Kapranov \cite{GinzburgK94} (see also \cites{GetzlerJ94,Fresse04}). It also includes a %
deformation theory due to M.~Gerstenhaber \cite{Gerstenhaber64}; see \cites{Balavoine97,KontsevichS00,Keller05} for the operadic account. 
Homotopical resolutions of the Gerstenhaber operad led the way to D.~Tamarkin's solution of the Deligne conjecture \cite{Tamarkin98}; see also \cites{GerstenhaberV95hom,Hinich03}. 

The notion of divided power algebras over an operad was introduced by B.~Fresse in \cite{Fresse00}. While usual algebras over an operad $\P$ are algebras over a monad $S(\P)$ built from $\P$ using coinvariant operations, divided power algebras over $\P$ are algebras over a monad $\Gamma(\P)$ which is built using \emph{invariant} operations. This recovers the classical %
divided power algebras of Cartan, and Fresse shows that it also encompasses the %
restricted Lie algebras of Jacobson. Furthermore, Fresse generalises Cartan's result by showing that the homotopy of a simplicial algebra over $\P$ is always equipped with a structure of divided power algebra over $\P$. Divided power algebras were further studied by the third author, who obtained a convenient characterisation for these objects using monomial operations \cite{Ikonicoff20}.

The aim of this article is to study the %
Quillen 
cohomology of divided power algebras over any operad $\P$, using techniques due to Quillen, Barr and Beck. This generalises %
work of the first author \cites{Dokas04,Dokas09,Dokas23}.

\subsection*{Outline and main results} 

Sections~\ref{sec:Prelim} and \ref{sec:divpower} are background sections. In Section~\ref{sec:Prelim}, we recall the basic constructions leading to Quillen homology and cohomology. In Section~\ref{sec:divpower}, we recall a characterisation for divided power algebras obtained by the third author in \cite{Ikonicoff20} which we will use throughout the article. As a preliminary useful result, we obtain the following:
\begin{prop*}[see Proposition~\ref{prop:powersp}]
    The structure of a $\Gamma(\P)$-algebra is entirely determined by monomial operations whose monomial degrees are powers of the characteristic of the base field.
\end{prop*}
The rest of the article is divided into three parts.

\textbf{Part~I. General theory.} In Sections~\ref{sec:Amod} to \ref{sec:Quillencohomology}, we develop the ingredients that go into Quillen cohomology for $\Gamma(\P)$-algebras.

In Section~\ref{sec:Amod}, we introduce a new notion of module over a $\Gamma(\P)$-algebra, %
and of abelian $\Gamma(\P)$-algebra, which corresponds to a module over the trivial (or terminal) $\Gamma(\P)$-algebra. 
Section~\ref{sec:BeckModules} is devoted to the proof of the following result:
\begin{theo*}[see Theorem~\ref{thm:Amodule}]
        The data of a Beck module over the $\Gamma(\P)$-algebra $A$ is equivalent to the data of an $A$-module as in Definition~\ref{def:Amodule}.
\end{theo*}
One side of the equivalence is obtained by building a %
semidirect product for $A$-modules.

In Section~\ref{sec:Universalalgebra}, we build a ring $\UU_{\Gamma(\P)}(A)$ which represents the operations %
that define 
an $A$-module. We obtain the following result:
\begin{theo*}[See Theorem~\ref{thm:BeckMod}]
    The category of $A$-modules is equivalent to the category of left modules over $\UU_{\Gamma(\P)}(A)$. 
\end{theo*}
In Section~\ref{sec:derivations}, we identify Beck derivations in the category of $A$-modules, and we construct an $A$-module $\Omega_{\Gamma(\P)}(A)$ which represents these derivations.

Section~\ref{sec:Quillencohomology} concludes our general theory by identifying the left adjoint functor to the functor which sends an $A$ module $M$ to the $\Gamma(\P)$-algebra over $A$ obtained by the semidirect product $A\ltimes M$ of $M$ by $A$. More precisely, we obtain:

\begin{theo*}[See Theorem~\ref{thm:Abelianization}]
    The following two functors form an adjoint pair:            
    \[
        \diag@=3.6cm{\Alg{\Gamma(\P)}/A\ar@<3pt>[r]^-{\UU_{\Gamma(\P)}(A)\otimes_{\UU_{\Gamma(\P)}(-)}\Omega_{\Gamma(\P)}(-)}&\lMod{A}.\ar@<3pt>[l]^-{A\ltimes -}}
    \]    
\end{theo*}
This induces an analogue of Quillen's cotangent complex, and by deriving the simplicial extension of this left adjoint functor, we obtain the desired Quillen homology of $A$.

\textbf{Part~II. Examples.} In Sections~\ref{sec:ExCom} and \ref{sec:ExLie}, we apply %
our general theory 
to the examples of classical divided power algebras and of restricted Lie algebras. 
We recover some results obtained by the first author in \cites{Dokas04,Dokas09,Dokas15,Dokas23}.

\textbf{Part~III. Comparisons.} In Sections~\ref{sec:comp} to \ref{sec:Good}, we investigate different comparison maps induced by adjunctions, which allow us to compare our cohomology for $\Gamma(\P)$-algebras to more usual cohomology theories in certain categories of algebras.

We focus on two types of adjunctions involving %
divided power algebras. First, for any operad $\P$, there is an adjunction between the category of $\P$-algebras and the category of $\Gamma(\P)$-algebras. Second, any morphism of operads $f \colon \P \to \Q$ induces an adjunction between the categories of $\Gamma(\P)$- and $\Gamma(\Q)$-algebras, similar to the extension/restriction of scalars adjunction for modules over a ring. In Section~\ref{sec:comp}, we describe the comparison maps induced by both types of adjunctions in a general setting. 

In Section~\ref{sec:compdiv}, we study the first type of adjunction for the operad $\Com$ of commutative algebras, which produces comparison maps between the usual Quillen cohomology of a commutative algebra, and the Quillen cohomology %
of a certain associated divided power algebra. Similarly, in Section~\ref{sec:complie}, we produce a comparison map between the usual Quillen cohomology of a Lie algebra, and the Quillen cohomology %
of a certain associated restricted Lie algebra.

In Section~\ref{sec:compasslie}, we study an adjunction between associative algebras and restricted Lie algebras which is induced by an injection of the operad of Lie algebras into the operad of associative algebras. We obtain the following:
\begin{theo*}[See Theorem~\ref{thm:AssLie}]
Let $L$ be a restricted Lie algebra and $M$ a $u(L)$-bimodule. Then there is an isomorphism
\[
\HQ_{\As}^{*}(u(L); M) \cong \HQ_{\RLie}^{*}(L; {}_{u(L)}M).
\]
\end{theo*}
Here, $u(L)$ denotes an analogue of the universal enveloping algebra for restricted Lie algebras, $\HQ_{\As}^{*}$ is the usual Quillen cohomology for associative algebras, and $\HQ_{\RLie}^{*}$ is the Quillen cohomology for restricted Lie algebras. 
This result is analogous to a classical theorem of Cartan and Eilenberg, which gives an isomorphism between Chevalley--Eilenberg cohomology of a Lie algebra~$L$ and Hochschild cohomology of its universal enveloping algebra $U(L)$ \cite{CartanE56}*{\S XIII, Theorem~5.1}.

Finally, in Section~\ref{sec:Good}, we come back to the case where the base field is of characteristic $0$, and we study the case of a good triple of operads $(\C,\A,\P)$. In this case, the notions of $\P$-algebras and $\Gamma(\P)$-algebras coincide. Good triples of operads generalise %
the mutual behaviours of the operads of commutative, of associative, and of Lie algebras. In particular, we see the operad $\P$ as a suboperad $\A$ of a certain type. In this setting, we can generalise the result of the previous section, and we obtain:
\begin{theo*}[See Theorem~\ref{thm:good}]
Let $(\CC,\A,\P)$ be a good triple of operads. Let $P$ be an $\P$-algebra and $M$ a Beck $U(P)$-module. Then we have the following isomorphism 
\[
\HQ_{\Alg{\A}}^{*}(U(P); M) \simeq \HQ_{\Alg{\P}}^{*}(P; {}_{P}M).
\]
\end{theo*}
Here again, $U(P)$ denotes a certain notion of universal enveloping $\A$-algebra, $\HQ_{\Alg{\A}}^{*}$ is the Quillen cohomology for $\A$-algebras, and $\HQ_{\Alg{\P}}^{*}$ is the Quillen cohomology for $\P$-algebras. 
If we consider the good triple $(\Com, \As, \Lie)$, then Theorem~\ref{thm:good} recovers the Cartan--Eilenberg result in characteristic zero. %

While this last section focusses on the characteristic zero setting, there is %
evidence that %
good triples of operads can be generalised in positive characteristic, and we expect that our result still holds under the correct assumptions. Formulating these assumptions is---to our knowledge---an open problem for future consideration.

\subsection*{Conventions and notations}

All the operads considered in this article will be algebraic, symmetric operads. We assume the reader has a good familiarity with this %
type of operad. For more details, we refer to \cite{LodayV12}.

For this whole article, we fix a base field $\F$ and an operad $\P$ in $\F$-vector spaces that is reduced, i.e., satisfying $\P(0) = 0$. A $\Gamma(\P)$-algebra will be generically denoted $A$.

\begin{nota}
We will denote by
\begin{itemize}
\item $\Vect{\F}$ the category of $\F$-vector spaces.
\item $\Alg{\P}$ the category of $\P$-algebras in $\Vect{\F}$, with forgetful functor $U^{\P}_{\F} \colon \Alg{\P} \to \Vect{\F}$ and its left adjoint $F^{\P}_{\F} \colon \Vect{\F} \to \Alg{\P}$.
\item $\Alg{\Ga(\P)}$ the category of $\Ga(\P)$-algebras in $\Vect{\F}$, with forgetful functor $U^{\Ga(\P)}_{\P} \colon \Alg{\Ga(\P)} \to \Alg{\P}$ and its left adjoint functor, $F^{\Ga(\P)}_{\P} \colon \Alg{\P} \to \Alg{\Ga(\P)}$.
\end{itemize}
\end{nota}

\subsection*{Acknowledgements}

We thank Muriel Livernet for helpful discussions.  
We also thank the anonymous referee for their useful suggestions. 

Frankland acknowledges the support of the Natural Sciences and Engineering Research Council of Canada (NSERC), grant RGPIN-2019-06082. 
Ikonicoff acknowledges the support of the Fields Institute.

\section{Preliminaries on Quillen cohomology}\label{sec:Prelim}

Throughout this section, let $\cat{C}$ be an algebraic category, i.e., a category of models for a finite product sketch (or equivalently, for a many-sorted Lawvere theory). %
More details about %
algebraic categories 
are given in \cite{Frankland15}*{\S 2}. For the purposes of this paper, the categories of $\P$-algebras and $\Ga(\P)$-algebras are examples of algebraic categories, in fact, the one-sorted kind: sets equipped with certain operations satisfying certain equations. Let us recall some terminology, also found in \cite{Barr96} and \cite{Frankland15}*{\S 1.3}.

\begin{defi}\label{def:BeckModule}
For an object $X$ of $\cat{C}$, a \Def{Beck module} over $X$ is an abelian group object in the slice category $\cat{C}/X$. The category of Beck modules over $X$ is denoted $(\cat{C}/X)_{\ab}$. 
\end{defi}

\begin{defi}
The \Def{abelianization} over $X$ is the left adjoint functor $Ab_X \colon \cat{C}/X \to (\cat{C}/X)_{\ab}$ to the forgetful functor $U_X \colon (\cat{C}/X)_{\ab} \to \cat{C}/X$.
\end{defi}

\begin{defi}
For a map $f \colon X \to Y$ in $\cat{C}$, %
the pullback functor $f^* \colon \cat{C}/Y \to \cat{C}/X$ preserves limits and thus induces a functor $f^* \colon (\cat{C}/Y)_{\ab} \to (\cat{C}/X)_{\ab}$, also called \Def{pullback}. The \Def{pushforward} along $f$ is the left adjoint $f_! \colon (\cat{C}/X)_{\ab} \to (\cat{C}/Y)_{\ab}$ of $f^*$. %
\end{defi}

Note that abelianizations and pushforwards exist when $\cat{C}$ is an algebraic category \cite{Frankland15}*{Proposition~3.26}.

\begin{lemm}\label{lem:PushAbel}
Given a map $f \colon X \to Y$ in $\cat{C}$ viewed as an object of the slice category $\cat{C}/Y$, its abelianization is given by
\[
Ab_Y(X \ral{f} Y) \cong f_! (Ab_X X)
\]
where $Ab_X X$ is shorthand for $Ab_X(X \ral{\id} X)$.
\end{lemm}

\begin{defi}\label{def:Derivation}
Given a Beck module $\pr \colon E \surj X$ over $X$, a (Beck) \Def{derivation} from $X$ to $E$ is a section of $\pr$. The set of derivations is denoted
\rescale{\begin{align*}
	\Der(X,E) \dfn \Hom_{\cat{C}/X}(X \ral{\id} X, E \ral{\pr} X) \cong \Hom_{(\cat{C}/X)_{\ab}}(Ab_X X, E \ral{\pr} X).
\end{align*}}
Note that $\Der(X,E)$ is canonically an abelian group.

The module of \Def{K\"ahler differentials} of $X$ is $\Om_{\cat{C}}(X) \dfn Ab_X X$, which represents derivations. %
\end{defi}

\begin{defi}
The \Def{cotangent complex} $\LL_X$ of $X$ is the derived abelianization of $X$, i.e., the simplicial module over $X$ given by $\LL_X \dfn Ab_X(C_{\bu} \to X)$, where $C_{\bu} \ral{\sim} X$ is a cofibrant replacement of $X$ in $s\cat{C}$, the category of simplicial objects in $\cat{C}$.

Here $s\cat{C}$ is endowed with the standard Quillen model structure \cite{Quillen67}*{\S II.4} \cite{GoerssJ09}*{\S II.5}. %
\end{defi}

\begin{defi}
The $n^{\text{th}}$ \Def{Quillen homology} module of $X$ is the $n^{\text{th}}$ (simplicially) derived functor of abelianization, given by $\HQ_n(X) \dfn \pi_n(\LL_X)$. Note that $\HQ_n(X)$ is a Beck module over $X$. 
\end{defi}

\begin{defi}
The $n^{\text{th}}$ \Def{Quillen cohomology} group of $X$ with coefficients in a module $M$ is the $n^{\text{th}}$ (simplicially) 
derived functor of derivations, given by $\HQ^n(X;M) \dfn \pi^n \Hom(\LL_X,M)$. Note that $\HQ^n(X;M)$ is an abelian group.
\end{defi}

\begin{ex}
Assuming $\cat{C}$ is a one-sorted algebraic category, consider the underlying set functor $U = U^{\cat{C}}_{\Set} \colon \cat{C} \to \Set$ and its left adjoint $F = F^{\cat{C}}_{\Set} \colon \Set \to \cat{C}$. Iterating the free-of-forget comonad $FU$ yields the standard augmented simplicial object
\[
C_{\bu} \dfn (FU)^{\bu+1}(X) \to X,
\]
which is a cofibrant replacement of $X$ in $s\cat{C}$. This was the original approach used by Barr and Beck \cite{BarrB69}, and by M.~Andr\'e in his work on Andr\'e--Quillen cohomology %
\cites{Andre67,Andre74}. 
\end{ex}

\section{Recollections on divided power algebras}\label{sec:divpower}

In this section, we review the general notion of a divided power algebra over an operad $\P$, also called a $\Gamma(\P)$-algebra, due to Fresse \cite{Fresse00}. Throughout this article, we use the characterisation of divided power algebras in terms of monomial operations due to the third author \cite{Ikonicoff20}. We rely heavily on the notation introduced in \cite{Ikonicoff20} for those operations, and operations on partitions of integers.

Recall that an operad $\P$ is (in particular)
a sequence of $\F$-vector spaces $\{\P(n)\}_{n\in\N}$ such that %
$\P(n)$ %
has a right action of $\Sy_n$, the symmetric group %
on $n$ letters. 
An operad is called \Def{reduced} when $\P(0)=0$, and Fresse showed the following:
\begin{prop}[\cite{Fresse00}*{\S 1.1.18}]
    For a reduced operad $\P$, the endofunctor $\Gamma(\P)$ in vector spaces defined on objects by:
    \[
        \Gamma(\P,V)=\bigoplus_{n>0}(\P(n)\otimes V^{\otimes n})^{\∑_n}
    \]
    is equipped with a monad structure.
\end{prop}
Here, $(\P(n)\otimes V^{\otimes n})^{\∑_n}$ stands for the module of invariants under the diagonal action of $\∑_n$, where the action of $\∑_n$ on $V^{\otimes n}$ is the usual permutation of factors.
\begin{defi}
    A \Def{divided power $\P$-algebra}, or \Def{$\Gamma(\P)$-algebra}, is an algebra over the resulting monad $\Gamma(\P)$.
\end{defi}
    \begin{theo}[\cite{Ikonicoff20}]\label{thm:inv}
    A divided power algebra over a reduced operad $\P$ is a vector space $A$ endowed with a family of operations $\beta_{x,\r} \colon A^{\times s}\to A$, given for all $\r=(r_1,\dots,r_s)$ such that $r_1+\dots+r_s=n$ and $x\in\P(n)^{\∑_\r}$, which satisfy the relations:
  \begin{enumerate}[label=($\beta$\arabic*)]%,itemsep=5pt
    \item\label{relperm}$\beta_{x,\r}((a_i)_i)=\beta_{\rho^*\cdot x,\r^\rho}((a_{\rho^{-1}(i)})_{i})$ for all $\rho\in\∑_s$, where $\rho^*$ denotes the block permutation with blocks of size $(r_i)$ associated to $\rho$.
    \item\label{rel0} $\beta_{x,(0,r_1,r_2,\dots,r_s)}(a_0,a_1,\dots,a_s)=\beta_{x,(r_1,r_2,\dots,r_s)}(a_1,\dots,a_s)$.
    \item\label{rellambda} $\beta_{x,\r}(\lambda a_1,a_2,\dots,a_s)=\lambda^{r_1}\beta_{x,\r}(a_1,\dots,a_s)\quad \text{for all } \lambda\in \F$.
    \item\label{relrepet}
    If $r_1+\dots+r_s=n$ and $q_1+\dots+q_{s'}=s$, then
%    \begin{multline*}
	\[
      \beta_{x,\r}(\underbrace{a_1,\dots,a_1}_{q_1},\underbrace{a_2,\dots,a_2}_{q_2},\dots,\underbrace{a_{s},\dots,a_{s}}_{q_{s'}}) =  \beta_{\big(\sum_{\sigma\in \∑_{\q\rhd\r}/\∑_\r}\sigma\cdot x\big),\ \q\rhd\r}(a_1,a_2,\dots,a_s).
	\]
%    \end{multline*}
    \item\label{relsomme} $\beta_{x,\r}(a_0+a_1,\dots,a_s)=\sum_{l+m=r_1}\beta_{x,\r\circ_1(l,m)}(a_0,a_1,\dots,a_s)$.
    \item\label{rellin} $\beta_{\lambda x+y,\r}=\lambda\beta_{x,\r}+\beta_{y,\r}$ , for all $x,y\in\P(n)^{\∑_\r}$,
    \item \label{relunit}$\beta_{1_\P,(1)}(a)=a \quad \text{for all } a\in A$.
    \item \label{relcomp} Let $r_1+\dots+r_s=n$, $x\in \P(n)^{\∑_\r}$ and for all $i\in[s]$, let $q_{i,1}+\dots+q_{i,k_i}=m_i$, $x_i\in\P(m_i)^{\∑_{\q_i}}$ and $(b_{ij})_{1\le j\le k_i}\in A^{\times k_i}$. Denote by $b=(b_{ij})_{i\in[s],j\in[k_i]}$ and for all $i\in[s]$, $b_i=(b_{i,j})_{j\in[k_i]}$. Then:
    \[
      \beta_{x,\r}(\beta_{x_1,\q_1}(b_{1}),\dots,\beta_{x_s,\q_s}(b_{s}))=\beta_{\sum_{\tau}\tau\cdot x \big(x_1^{\times r_1},\dots,x_s^{\times r_s}\big),\r\diamond(\q_i)_{i\in[s]}}(b),
    \]
    where $\r\diamond(\q_i)_{i\in[s]}$ is defined in \cite{Ikonicoff20}, where $\beta_{\cdot,\r\diamond(\q_i)_{i\in[s]}}$ is defined in \cite{Ikonicoff20} and where $\tau$ ranges over $\∑_{\r\diamond(q_i)_{i\in[s]}}/(\prod_{i=1}^{s}\∑_{r_i}\wr \∑_{\q_i})$ in the sum.
  \end{enumerate}
  \end{theo}

We can refine this characterisation if we specify the characteristic of the base field:
\begin{prop}\label{prop:powersp}
    Let $A$ be a $\Gamma(\P)$-algebra over a field $\F$ of characteristic $p$. Then, the operations $\beta_{x,\r}$ are generated by the family of operations $\beta_{x,\r}$ such that all the $r_i$s are powers of $p$.
\end{prop}
\begin{proof}
    Let $n\in\N$, $r_1+\dots+r_s=n$, and $x\in\P(n)^{\∑_\r}$. For each $i\in[s]$, denote $r_i=\sum_{j=0}^{k_{r_i}}r_{i,j}p^j$ the mod $p$ expansion of $r_i$. For all $i\in[s]$ and $j\in\{0,\dots,k_i\}$, denote by $R_{ij}=r_{i,j}p^j$. We get a partition $\underline R=(R_{1,0},\dots,R_{1,k_{r_1}},R_{2,0},\dots,R_{s,k_s})$ of $[n]$ into $k_{r_1}+\dots+k_{r_s}+s$ integers. Note that since $\underline R$ is finer than $\r$, $x$ is fixed by the action of $\∑_{\underline R}$. 
    Since $x\in\P(n)^{\r}$, we obtain 
    \[
    \sum_{\sigma\in\∑_{\r}/\∑_{\underline R}}x=\left(\prod_{i=1}^s\binom{r_i}{r_{i,0}p^0,\dots,r_{i,k_{r_i}}p^{k_{r_i}}}\right)x.
    \]
    Using Lucas's Theorem, the product $\prod_{i=1}^k\binom{r_i}{r_{i,0}p^0,\dots,r_{i,k_{r_i}}p^{r_{k_i}}}$ is equal to $1$ modulo $p$. Using relation~\ref{relrepet}, we then get: 
    \[
        \beta_{x,\r}(a_1,\dots,a_s)=\beta_{x,\underline R}(\underbrace{a_1,\dots,a_1}_{k_{r_1}+1},\dots,\underbrace{a_s,\dots,a_s}_{k_{r_s}}).
	\]
    Now, consider the partition 
    \[
    		\underline{Q} \dfn (\underbrace{p^0,\dots,p^0}_{r_{1,0}},\dots,\underbrace{p^{k_{r_1}},\dots,p^{k_{r_1}}}_{r_{1,{k_{r_1}}}},\dots,\underbrace{p^{k_{r_s}},\dots,p^{k_{r_s}}}_{r_{s,{k_{r_s}}}}).
    	\]
    Since $\underline Q$ is finer than $\underline R$, $x$ is stable under the action of $\∑_{\underline Q}$. Using Lucas's Theorem again, we get:
    \begin{align*}
        \sum_{\sigma\in\∑_{\underline R}/\∑_{\underline Q}}\sigma x &= \left(\prod_{i=1}^s\prod_{j=1}^{k_{r_i}}\binom{r_{i,j}p^{j}}{p^j,\dots,p_j}\right)x
        =\left(\prod_{i=1}^s\prod_{j=1}^{k_{r_i}}r_{i,j}!\right)x.
    \end{align*}
    Since $r_{i,j}<p$ for all $i,j$, $r_{i,j}!$ is invertible modulo p. So, we can write:
    \[
        x=\left(\prod_{i=1}^s\prod_{j=1}^{k_{r_i}}\left(r_{i,j}!\right)^{-1}\right)\sum_{\sigma\in\∑_{\underline R}/\∑_{\underline Q}}\sigma x,
    \]
    and so, using relation~\ref{relcomp},
%    \begin{multline*}
	\[
        \beta_{x,\underline R}(\underbrace{a_1,\dots,a_1}_{k_{r_1}+1},\dots,\underbrace{a_s,\dots,a_s}_{k_{r_s}}) = \\
        \left(\prod_{i=1}^s\prod_{j=1}^{k_{r_i}}\left(r_{i,j}!\right)^{-1}\right)\beta_{x,\underline Q}(\underbrace{a_1,\dots,a_1}_{r_{1,0}+\dots+r_{1,k_{r_1}}},\dots,\underbrace{a_s,\dots,a_s}_{r_{s,0}+\dots+r_{s,k_{r_s}}}).
	\]
%    \end{multline*}
    Here, the integers in $\underline Q$ are either $0$ or powers of $p$. Using relation~\ref{rel0} allows us to remove the zeros, and we obtain the result.
\end{proof}

\section{Abelian \texorpdfstring{$\Gamma(\P)$}{Gamma(P)}-algebras, \texorpdfstring{$A$}{A}-modules}\label{sec:Amod}

In this section, we develop the notion of a module over a divided power algebra, given by a certain abelian (or square-zero) extension, analogous to the notion of module over an operadic algebra \cite{LodayV12}*{\S 12.3}. 
In the setting of usual (non-unital) $\F$-algebras, since any vector space can be equipped with the trivial algebra structure, the category of abelian (square-zero) $\F$-algebras coincide with the category of $\F$-vector spaces. An $A$-module then becomes a vector space equipped with an $A$-action. However, we will see that in the category of divided power algebras, abelian $\Gamma(\P)$-algebras correspond to vector spaces equipped with additional internal operations. %
This section contains our candidates for the definitions of abelian $\Gamma(\P)$-algebras and of $A$-modules. %
We will show in Section~\ref{sec:BeckModules} that these are indeed the right definitions (see Theorem~\ref{thm:Amodule}, Corollary~\ref{coro:abelianGammaPalg}).
\begin{defi}\label{def:abeliangammapalg}
        An \Def{abelian} %
        $\Gamma(\P)$-algebra is a $\Gamma(\P)$-algebra $M$ such that all the operations $\beta_{x,\r}$ are trivial as soon as $\r$ contains two non-zero integers. Abelian $\Gamma(\P)$-algebras with $\Gamma(\P)$-algebra morphisms form a category $(\Alg{\Gamma(\P)})_{\mathrm{Ab}}$.
    \end{defi}
    \begin{prop}
        Equivalently, an abelian $\Gamma(\P)$-algebra is a vector space $M$ equipped, for all $n\in\N$ and $x\in\P(n)^{\∑_n}$, with an operation (a set map) $\beta_{x} \colon M\to M$ satisfying:
        \begin{enumerate}[label=(Ab$\beta$\arabic*)]
        \setcounter{enumi}{2}
            \item \label{relAblambda}$\beta_{x}(\lambda m)=\lambda^n\beta_x(m)$,
            \item\label{relAbrepet}Suppose there is $\r\in\Comp_s(n)$, with $s>1$ and $\r$ contains at least two non-zero integers, and $y\in\P(n)$ such that $x=\sum_{\sigma\in\∑_n/\∑_\r}y$. Then $\beta_x=0$,
            \item \label{relAbsomme}$\beta_x(m_1+m_2)=\beta_x(m_1)+\beta_x(m_2)$,
            \item \label{relAblin}$\beta_{\lambda x+ y}(m)=\lambda\beta_x(m_1)+\beta_y(m_2)$,
            \item \label{relAbunit}$\beta_{1_\P}(m)=m$,
            \item \label{relAbcomp}$\beta_x(\beta_y(m))=\beta_{\sum_\tau \tau\cdot x(y,\dots,y)}(m)$, where $y\in\P(l)$ and $\tau$ ranges over $\∑_{nl}/\∑_n\wr\∑_l$.      
        \end{enumerate}
    \end{prop}
    \begin{proof}
        Let $M$ be a $\Gamma(\P)$ algebra such that all the operations $\beta_{x,\r}$ are trivial as soon as $\r$ contains two non-zero integers. For all $x\in\P(n)^{\∑_n}$, denote by $\beta_{x} \dfn \beta_{x,(n)}$. We will prove that relations \ref{relAblambda} to \ref{relAbcomp} are satisfied:
         \begin{enumerate}[label=(Ab$\beta$\arabic*)]
        \setcounter{enumi}{2}
            \item is a direct consequence of \ref{rellambda},
            \item is deduced from \ref{relrepet} in the following way: if there is $\r\in\Comp_s(n)$ with $s>1$ and $\r$ contains at least two non-zero integers, and $y\in\P(n)$ such that $x=\sum_{\sigma\in\∑_n/\∑_\r}y$, then
            \begin{equation*}
               \beta_x(m)=\beta_{x,(n)}(m)=\beta_{y,\r}(m,\dots,m)=0.
            \end{equation*}
            \item is deduced from \ref{relsomme} and \ref{rel0} in the following way:
            \begin{align*}
                \beta_x(m+m') &= \beta_{x,(n)}(m+m') = \sum_{i+j=n}\beta_{x,(i,j)}(m,m') \\
                &= \beta_{x,(n,0)}(m,m')+\beta_{x,(0,n)}(m,m')\\
                &= \beta_x(m)+\beta_x(m'). 
            \end{align*}
            \item is a %direct 
            consequence of \ref{rellin},
            \item is a %direct 
            consequence of \ref{relunit},
            \item is a %direct 
            consequence of \ref{relcomp}.
        \end{enumerate}
        Let now $M$ be a vector space equipped, for all $x\in \P(n)^{\∑_n}$, with an operation $\beta_x \colon M\to M$, satisfying the relations \ref{relAblambda} to \ref{relAbcomp}. For all $x\in \P(n)^{\∑_n}$ and $\r\in\Comp_s(n)$, define an $s$-ary operation $\beta_{x,\r}$, such that:
         \[
             \beta_{x,\r}(m_1,\dots,m_s)=\begin{cases}
                 \beta_x(m_i)&\mbox{if }r_i=n,\\
                 0&\mbox{if $\r$ contains at least two non-zero integers.}
             \end{cases}
         \]
         We want to show that $\beta_{x,\r}$ satisfy the relations \ref{relperm} to \ref{relcomp}. Relations \ref{relperm} and \ref{rel0}, are a direct consequence of the definition. Relations \ref{rellambda}, \ref{relsomme}, \ref{rellin}, \ref{relunit} are direct consequences respectively of relations \ref{relAblambda}, \ref{relAbsomme}, \ref{relAblin}, and \ref{relAbunit}.

         To prove relation~\ref{relrepet}, note that the left term of the equality,
         \[
              \beta_{x,\r}(\underbrace{a_1,\dots,a_1}_{q_1},\underbrace{a_2,\dots,a_2}_{q_2},\dots,\underbrace{a_{s},\dots,a_{s}}_{q_{s'}}),
         \]
         is equal to 0 except if $r_i=n$ and $q_j=m$ for certain $i$ and $j$, in which case, it is equal to $\beta_{x}(a_i)$. On the other hand, the right term of the equality, 
         \[
             \beta_{\big(\sum_{\sigma\in \∑_{\q\rhd\r}/\∑_\r}\sigma\cdot x\big),\ \q\rhd\r}(a_1,a_2,\dots,a_s),
         \]
         is 0 unless $\q\rhd\r$ contains a single non-zero integers. But this happens if and only if both $\r$ and $\q$ contain a single non-zero integer, namely, if there is an $i$ and a $j$ such that $r_i=n$ and $q_j=m$, and if so, $\∑_{\q\rhd\r}/\∑_\r$ is trivial, so this term is also equal to $\beta_x(a_i)$.

         To prove relation~\ref{relcomp}, note that the left term in the equality, 
         \[
             \beta_{x,\r}(\beta_{x_1,\q_1}(b_{1}),\dots,\beta_{x_s,\q_s}(b_{s})),
         \]
         Is 0 unless $r_i=n$, $q_{i,j}=m_i$, and in this case, this term is equal to $\beta_x(\beta_{x_i}(b_{i,j}))$. On the other hand, the right term in the equality, 
         \[
             \beta_{\sum_{\tau}\tau\cdot x\big(x_1^{\times r_1},\dots,x_s^{\times r_s}\big)\big),\r\diamond(\q_i)_{i\in[s]}}(b),
         \]
         Is equal to 0 unless $\r\diamond(\q_i)_{i\in[s]}$ has a single non-zero integer. But again, this happens if and only if there exists an $i$ such that $r_i=n$, a $j$ such that $q_{i,j}=m_i$, and in this case, $\r\diamond(\q_i)_{i\in[s]}=(nm_i)$,$\∑_{\r\diamond(q_i)_{i\in[s]}}/(\prod_{i=1}^{s}\∑_{r_i}\wr \∑_{\q_i})=\∑_{nm_i}/\∑_n\wr\∑_{m_i}$, and so, this term is equal to
         \[
             \beta_{\sum_{\tau}\tau x(x_i,\dots,x_i)}(b_{i,j}),
         \]
         where $\tau$ ranges over $\∑_{nm_i}/\∑_n\wr\∑_{m_i}$ in the sum.

         We have left to prove that 
         \[
             \beta_x(\beta_{x_i}(b_{i,j}))=\beta_{\sum_{\tau}\tau x(x_i,\dots,x_i)}(b_{i,j}),
         \]
         where $\tau$ ranges over $\∑_{nm_i}/\∑_n\wr\∑_{m_i}$ in the sum, but that is exactly relation~\ref{relAbcomp}.
    \end{proof}
    \begin{prop}\label{prop:Abpowersp}
        Let $M$ be an abelian $\Gamma(\P)$-algebra over a field $\F$ of characteristic $p$. 
        Then $\beta_{x}(m)=0$ for all $x\in\P(n)^{\∑_n}$, for all $n\neq p^{i}$, $i\in\N$.
    \end{prop}
    \begin{proof}
        This is similar to the proof of Proposition~\ref{prop:powersp}. Denote by $n_0p^0+\dots+n_kp^k$ the mod $p$ expansion of $n$. Denote by $\underline Q$ the partition $(\underbrace{p^0,\dots,p^0}_{n_0},\dots,\underbrace{p^k,\dots,p^k}_{n_k})$. Denote by $y=(\prod_{j=0}^kn_j!)^{-1}$. Using Lucas's Theorem, 
        one has:
        \[
            \sum_{\sigma\in\∑_{n}/\∑_{\underline Q}}\sigma\cdot y=\left(\prod_{j=0}^k\binom{n_jp^j}{p^j,\dots,p^j}\right)y=\left(\prod_{j=0}^kn_j!\right)y=x.
        \]
        Either $n$ is a power of $p$, or $\underline Q$ has at least two non-zero integers, which, using relation~\ref{relAbrepet}, implies that $\beta_x=0$.
    \end{proof}
    \begin{coro}
        Let $M$ be an abelian $\Gamma(\P)$-algebra over a field $\F$ of characteristic $p$. Then, the operations $\beta_{x}$ are generated by the family of operations $\beta_x$ such that $x\in \P(p^i)^{\∑_{p^i}}$.
    \end{coro}
    \begin{defi}\label{def:Amodule}
        A \Def{module} over the $\Gamma(\P)$-algebra $A$ is an abelian $\Gamma(\P)$-algebra $M$ equipped with a divided power action of $A$ on $M$ represented by operations $\beta_{x,\r} \colon A^{\times s-1}\times M\to M$ for all $x\in \P(n)$ and $\r\in\Comp_s(n)$ such that $r_s\neq 0$, satisfying:
        \begin{enumerate}[label=($\beta$AM\arabic*)]
        \setcounter{enumi}{5}
    \item\label{relAMlin} $\beta_{\lambda x+y,\r}=\lambda\beta_{x,\r}+\beta_{y,\r}$ , for all $x,y\in\P(n)^{\∑_\r}$,
    \item \label{relAMcompM}If $s=1$, then $\beta_{x,(n)}(m)=\beta_{x}(m)$.
    \item \label{relAMcomp} Let $r_1+\dots+r_s=n$, $x\in \P(n)^{\∑_\r}$ and for all $i\in[s]$, let $q_{i,1}+\dots+q_{i,u_i}=k_i$, $x_i\in\P(k_i)^{\∑_{\q_i}}$. Let $(b_{ij})_{1\le j\le u_i}\in A^{\times u_i}$ for $i\in[s-1]$ and $b_{s,1},\dots,b_{s,u_s-1}\in A$. For all $i\in[s-1]$, denote by $b_i=(b_{i,j})_{j\in[u_i]}$. Then:
    \begin{multline*}
        \beta_{x,\r}(\beta_{x_1,\q_1}(b_1),\dots,\beta_{x_{s-1},\q_{s-1}}(b_{s-1}),\beta_{x_s,\q_s}(b_{s,1},\dots,b_{s,u_s-1},m))\\=\beta_{\sum_{\tau}\tau\cdot x(x_1^{r_1},\dots,x_s^{\times r_s}),\r\diamond(\q_1,\dots,\q_s)}(b_{1,1},\dots,b_{1,u_1},\dots,b_{s,u_s-1},m),
    \end{multline*}
    where $\r\diamond(\q_i)_{i\in[s]}$ is defined in \cite{Ikonicoff20} and where $\tau$ ranges over $\∑_{\r\diamond\q}/\prod_{i=1}^s\∑_{r_i}\wr\∑_{\q_i})$ in the sum.
        \end{enumerate}
        For $s>1$, the two sets of relations: 
        \begin{enumerate}[label=($\beta$A\arabic*)]
        \item\label{relAperm}$\beta_{x,\r}((a_i)_i,m)=\beta_{\rho^*\cdot x,\r^\rho}((a_{\rho^{-1}(i)})_{i},m)$ for all $\rho\in\∑_{s-1}$, where $\rho^*$ denotes the block permutation with blocks of size $(r_i)$ associated to $\rho$.
	    \item\label{relA0} %\resizebox{\linewidth}{!}{
	    $\beta_{x,(0,r_1,r_2,\dots,r_s)}(a_0,a_1,\dots,a_{s-1},m)=\beta_{x,(r_1,r_2,\dots,r_s)}(a_1,\dots,a_{s-1},m)$
	    %.}
	    \item\label{relAlambda} $\beta_{x,\r}(\lambda a_1,a_2,\dots,a_{s-1},m)=\lambda^{r_1}\beta_{x,\r}(a_1,\dots,a_{s-1},m)$ for all $\lambda\in \F$.
	    \item\label{relArepet}
    If $\r\in \Comp_{s}(n)$ and $\q\in \Comp_{s'}(s-1)$, then
%    \begin{multline*}
	\[
      \beta_{x,\r}(\underbrace{a_1,\dots,a_1}_{q_1},\underbrace{a_2,\dots,a_2}_{q_2},\dots,\underbrace{a_{s-1},\dots,a_{s-1}}_{q_{s'}},m) = \\
       \beta_{\big(\sum_{\sigma\in \∑_{\q'\rhd\r}/\∑_\r}\sigma\cdot x\big),\ \q'\rhd\r}(a_1,a_2,\dots,a_{s-1},m),
	\]
%    \end{multline*}
    where $\q'=(\q,0)$,
	    \item\label{relAsomme} %\resizebox{\linewidth}{!}{
	    $\beta_{x,\r}(a_0+a_1,\dots,a_{s-1},m)=\sum_{l+l'=r_1}\beta_{x,\r\circ_1(l,l')}(a_0,a_1,\dots,a_{s-1},m)$,
	    %}
        \end{enumerate}
        and:
        \begin{enumerate}[label=($\beta$M\arabic*)]
    \setcounter{enumi}{2}
    \item\label{relMlambda} $\beta_{x,\r}( a_1,a_2,\dots,a_{s-1},\lambda m)=\lambda^{r_1}\beta_{x,\r}(a_1,\dots,a_{s-1},m) \quad \text{for all } \lambda\in \F$.
    \item\label{relMrepet}
    Suppose there is $\q\in\Comp_{s'}(r_s)$ with $s'>0$, where $\q$ contains at least two non-zero integers, and there is a $y\in\P(n)^{\∑_{1}^{\times n-r_s}\times \∑_\q}$ such that $x=\sum_{\sigma\in\∑_{n-r_s}\times\∑_{r_s}/\∑_{1}^{\times n-r_s}\times \∑_\q}y$, then $\beta_{x,\r}(a_1,\dots,a_{s-1},m)=0$,
    \item\label{relMsomme} 
%    \begin{multline*} 
	\[
		\beta_{x,\r}(a_1,\dots,a_{s-1},m_1+m_2) = 
		\beta_{x,\r}(a_1,\dots,a_{s-1},m_1)+\beta_{x,\r}(a_1,\dots,a_{s-1},m_2),
	\] 
%    \end{multline*}
        \end{enumerate}
        And the two additional relations:
        \begin{enumerate}[label=($\beta$AM4)]
            \item \label{relAM4} Let $s,s'>0$, $\q\in\Comp_{s'}(n)$, and $\r\in \Comp_s(\q)$. Let $x\in \P(n)^{\∑_{\q\circ_i\r}}$. Then,
            \rescale{\begin{align*}
                \sum_{k=1}^{q_i}\beta_{\sum_{\rho\in \∑_1^{\times (q_1+\cdots+q_{i-1)}}\times \overline E_{\r,k}\times \∑_1^{\times(q_{i+1}+\cdots+q_{s'})}}\rho\cdot \sigma_i^* x,\q\circ_i(q_i-k,k)}(a_1,\dots,a_{s'},m)=0,
            \end{align*}}
            where $\sigma_i^*$ is the block permutation for blocks of size $q_1,q_2,\dots,q_{s'}$ associated to the transposition of $i$ and $s'$ in $\∑_{s'}$, and where:
            \rescale{\begin{align*}
                \overline E_{\r,k}=\left\{\rho'\in Sh(\r): \left|\{\rho'(r_1+\cdots+r_j):j\in[s]\}\cap\{q_{i}-k+1,\dots,q_{i}\}\right|\ge 2\right\},
            \end{align*}}
            where $Sh(\r)$ is the set of $(r_1,\dots,r_s)$-shuffles.
        \end{enumerate}
        \begin{enumerate}[label=($\beta$AM9)]
            \item \label{relAMspec} Let $r_1+\dots+r_s=n$, $x\in \P(n)^{\∑_\r}$, let $q_{1}+\dots+q_{u}=k$, $y\in\P(k)^{\∑_{\q}}$. Let $(a_i)_{i\in[s+u]}\in A^{\times s+u}$ and $m\in M$. Denote by $z=x\left(1_\P^{\times n-r_s},y^{\times r_s}\right)\in\P(n+r_s(k-1))$ Then,
            \begin{multline*}
                \sum_{t+t'=r_sq_{u},t'>0}\beta_{\sum_\tau \tau z,\r\diamond_s(\q)\circ_{s+u}(t,t')}(a_1,\dots,a_{s+u},m)=\\
                \sum_{\lambda+\lambda'=r_s,\lambda'>0}\sum_{l+l'=q_{u},l'>0}\beta_{\sum_{\sigma'}\sigma' z,R_{\lambda\lambda'}^{ll'}}(a_1,\dots,a_{s+u},m),
            \end{multline*}
            where, for all $l,l',\lambda,\lambda'$, $R_{\lambda\lambda'}^{ll'}=((R_{ij})_{i\in[s],j\in[u_s]},R_{s,u+1})$ is the partition of $[n+r_s(k-1)]$ into $s+u+1$ parts such that $R_{i}=(\r\diamond_s(\q)_{i}$ for all $i\in[s+u-1]$,  and such that:
%	\begin{multline*}
	\[
        R_{s+u} = \left\{ n-r_s+k-q_u+\alpha k+\gamma,\alpha\in[\lambda], \gamma\in[q_u] \right\} \cup  
        \left\{ \lambda k+(k-q_u)+\alpha k+\gamma, \alpha\in[\lambda'],\gamma\in[l] \right\}
	\]        
%    \end{multline*}
    \[
        R_{s+u+1} = \left\{ n-r_s+\lambda k+(k-l')+\alpha k+\gamma,\alpha\in[\lambda'],\gamma\in[l'] \right\},
    \]
where $\tau$ ranges over $\∑_{\r\diamond_s\q}/\left(\prod_{i=1}^{s-1}\∑_{r_i}\right)\times\∑_{r_s}\wr\∑_{\q}$ in the sum, and where $\sigma'$ ranges over $\∑_{R_{\lambda\lambda'}^{ll'}}/\left(\prod_{i=1}^{s-1}\∑_{r_i}\right)\times\∑_{\lambda}\wr\∑_\q\times\∑_{\lambda'}\wr\∑_{\q\circ_u(l,l')}$ in the sum.
        \end{enumerate}
        Modules over a $\Gamma(\P)$-algebra $A$ form a category $\lMod{A}$.
    \end{defi}
    \begin{nota}\label{nota:perm}
        We will use the notation:
        \[
            \beta_{x,\r}(a_1,\dots,a_{i-1},m,a_{i+1},\dots,a_s) \dfn \beta_{\sigma_i^*,\r^{\sigma_i}}(a_1,\dots,a_{i-1},a_{i+1},\dots,a_s,m),
        \]
        where $\sigma_i\in\∑_s$ sends $i$ to the $s$-th spot and keeps all other indices in order, that is:
        \[
            \sigma_i(j)=\begin{cases}
                j,&\mbox{if }j<i,\\
                s,&\mbox{if }j=i,\\
                j-1,&\mbox{if }j>i,
            \end{cases}
        \]
        and where $\sigma_i^*\in\∑_n$ is the block permutation with blocks of size $(r_i)_{i\in[s]}$ associated to $\sigma_i$.
    \end{nota}
    Once again, we can %
    refine this characterisation if we specify the characteristic of the base field:
    \begin{prop}\label{prop:AMpowersp}
        Let $M$ be an $A$-module over a field $\F$ of characteristic $p$. Then the operations $\beta_{x,\r}(a_1,\dots,a_{s-1},m)$ are generated by the operations $\beta_{x,\r}(a_1,\dots,a_{s-1},m)$ where all the $r_i$s are powers of $p$, and $\beta_{x,\r}(a_1,\dots,a_{s-1},m)=0$ when $r_s$ is not a power of $p$.
    \end{prop}
    \begin{proof}
        The proof, very similar to that of Propositions~\ref{prop:powersp} and \ref{prop:Abpowersp}, is omitted. 
    \end{proof}

\section{Beck modules over \texorpdfstring{$A$}{A}}\label{sec:BeckModules}

This section is devoted to %
a description 
of Beck modules over a $\Gamma(\P)$-algebra. 
Recall from Definition~\ref{def:BeckModule} that a Beck module over an object $A$ is 
an abelian group object over $A$; 
see \cite{Beck67}, \cite{Quillen70}*{\S 2--4}, \cite{Barr96}*{\S 6}, or \cite{Frankland10qui}*{\S A} for more background. 
The main result of this section, Theorem~\ref{thm:Amodule}, is that the notion of Beck module over a $\Gamma(\P)$-algebra $A$ corresponds to the notion of $A$-module from Definition~\ref{def:Amodule}. As a particular case, %
we obtain Corollary~\ref{coro:abelianGammaPalg}, which shows that the notion of abelian $\Gamma(\P)$-algebra from Definition~\ref{def:abeliangammapalg} corresponds to the notion of an abelian group object in the category of $\Gamma(\P)$-algebras.

    \begin{rema}
        If $\diag{B\ar[r]^{p}&A}$ and $\diag{C\ar[r]^{p'}&A}$ are $\Gamma(\P)$-algebras over $A$, we can endow $B\times_A C$ with a structure of $\Gamma(\P)$-algebra over $A$. The map $B \times_A C \to A$ is given by $p \times_A p'$, and the evaluation map $\P\tcirc (B\times_AC)\to B\times_AC$ is given by universal property of the pullback, using the two maps $\P\tcirc(B\times_A C)\to \P\tcirc B\to B$ and $\P\tcirc(B\times_A C)\to\P\tcirc C\to C$. This means that for $s$ elements $(b_1,c_1),\dots,(b_s,c_s,)\in B\times C$ such that $p(b_i)=p'(c_i)$,
    \[
        \beta_{x,\r}\left((b_1,c_1),\dots,(b_s,c_s)\right)=(\beta_{x,\r}(b_1,\dots,b_s),\beta_{x,\r}(c_1,\dots,c_s)).
    \]
    \end{rema}

\begin{theo}\label{thm:Amodule}
        The data of a Beck module over the $\Gamma(\P)$-algebra $A$ is equivalent to the data of an $A$-module as in Definition~\ref{def:Amodule}.
    \end{theo}
    \begin{proof}
        We will define a pair of functors:
        \[
            \ker  \colon  \left(\Alg{\Gamma(\P)}/A\right)_{\ab} \to \lMod{A},
        \]
        and
        \[
            A\ltimes -  \colon  \lMod{A} \to \left(\Alg{\Gamma(\P)}/A\right)_{\ab},
        \]
        and show that these are inverse to each other.
        \paragraph{Let us prove that the data of a Beck module over $A$ yields an $A$-module.} Let $\diag{B\ar[r]^{pr}& A}$ be a $\Gamma(\P)$-algebra over $A$. Denote by $M \dfn \ker(p)$. The data of a Beck module structure on $B$ is equivalent to the data of a map $\diag{A\ar[r]^z& B}$, and of a multiplication map $\diag{B\times_A B\ar[r]^-\mu& B}$ over $A$, satisfying the following conditions:
        \begin{enumerate}[label=(\arabic*)]
            \item $z$ is a map of $\Gamma(\P)$-algebras over $A$,\\
            \item $\mu$ is a map of $\Gamma(\P)$-algebras over $A$,
            \item $\mu$ is commutative,
            \item $z$ is a unit map for the multiplication $\mu$.
        \end{enumerate}
        Condition (1) implies that $z$ splits $pr$, so we have a split exact sequence of vector spaces:
        \[
            \diag{\0\ar[r]& M\ar[r]& B\ar[r]^{pr}& A\ar@/^/[l]^{z}\ar[r]&\0},
        \]
        which provides us with a linear isomorphism $B\cong A\oplus M$.

        The fact that $pr$ is a map of $\Gamma(\P)$-algebras implies that
        \begin{align*}
          pr(\beta_{x,\r}((a_1,m_1),\dots,(a_s,m_s))) &= \beta_{x,\r}(pr(a_1,m_1),\dots,pr(a_s,m_s)) \\
          &= \beta_{x,\r}(a_1,\dots,a_s), 
        \end{align*}
        so there exists a set map $\nabla_{x,\r} \colon A^{\times s}\times M^{\times s}\to M$ such that:
        \rescale{\begin{align*}
            \beta_{x,\r}((a_1,m_1),\dots,(a_s,m_s))=\left(\beta_{x,\r}(a_1,\dots,a_s),\nabla_{x,\r}(a_1,\dots,a_s,m_1,\dots,m_s)\right).
        \end{align*}}
        Note that, as a vector space, $B\times_A B\cong A\oplus M\oplus M$. One can rewrite condition (4) as the equality $\mu(a,m,0)=\mu(a,0,m)=(a,m)$. Classical computations on Beck modules for abelian groups show that, in fact, $\mu(a,m,m')=(a,m+m')$,
        from which the commutativity of $\mu$ is in fact necessary.

        Inspecting the $\Gamma(\P)$-algebra structure on $B\times_A B$ (see previous section), this implies that:
        \rescale{\begin{align*}
             &\mu\left(\beta_{x,\r}\left((a_1,m_1,m'_1),\dots,(a_s,m_s,m'_s)\right)\right)=\\
             &\bigg(\beta_{x,\r}(a_1,\dots,a_s),\nabla_{x,\r}(a_1,\dots,a_s,m_1,\dots,m_s)+\nabla_{x,\r}(a_1\dots,a_s,m'_1,\dots,m'_s)\bigg).
         \end{align*}} 
        However, since $\mu$ is a $\Gamma(\P)$-algebra map, one also has:
%        \begin{multline*}
		\[
            \mu(\beta_{x,\r}\left((a_1,m_1,m'_1),\dots,(a_s,m_s,m'_s)\right)) = 
            \left(\beta_{x,\r}(a_1,\dots,a_s),\nabla_{x,\r}(a_1,\dots,a_s,m_1+m'_1,\dots,m_s+m'_s)\right)
		\]
%        \end{multline*}
        %
        %
        %
        %
        %
        %
        And hence,
%        \begin{multline*}
		\[
            \nabla_{x,\r}(a_1,\dots,a_s,m_1+m'_1,\dots,m_s+m'_s) = 
            \nabla_{x,\r}(a_1,\dots,a_s,m_1,\dots,m_s)+\nabla_{x,\r}(a_1\dots,a_s,m'_1,\dots,m'_s).
		\]
%        \end{multline*}
        This implies that
        \begin{equation}
            \label{E1}\tag{E1}\nabla_{x,\r}(a_1,\dots,a_s,m_1,\dots,m_s)=\sum_{i=0}^s\nabla_{x,\r}(a_1,\dots,a_s,0,\dots,0,m_i,0,\dots,0),
        \end{equation}
        and that $\nabla_{x,\r}(a_1,\dots,a_s,0,\dots,0)=0$, which is also a consequence of condition (1). From relation~\ref{rellambda}, we also obtain:
        \begin{equation*}
            \beta_{x,\r}((a_1,m_1),\dots,(a_{i-1},m_{i-1}),0,(a_{i+1},m_{i+1}),\dots,(a_s,m_s))=0,
        \end{equation*}
        for all $i$ such that $r_i>0$, which also implies:
        \begin{equation}
            \label{E3}\tag{E3}\nabla_{x,\r}(a_1,\dots,a_{i-1},0,a_{i+1},\dots,a_s,m_1,\dots,m_{i-1},0,m_{i+1},\dots,m_s)=0.
        \end{equation}
        Using the decomposition $(a,m)=(a,0)+(0,m)$ and the relation~\ref{relsomme}, one then gets:
%
%
            %
%
%        \begin{multline*}
		\[
            \beta_{x,\r}\left((a_1,m_1),\ldots,(a_s,m_s)\right) = 
            \sum_{(l_i+l'_i=r_i)_{i}}\beta_{x,(l_1,l'_1,l_2,\ldots,l_s,l'_s)}\left( (a_1,0),(0,m_1),(a_2,0)\ldots,(a_s,0),(0,m_s) \right).
		\]
%        \end{multline*}
        Using relations~\ref{rellambda} and \ref{rel0}, we then get:
         \rescale{\begin{align*}
            &\beta_{x,\r} \left( (a_1,m_1),\ldots,(a_s,m_s) \right) = \left( \beta_{x,\r}(a_1,\ldots,a_s),0 \right) + \\
            &\sum_{(l_i+l'_i=r_i,l'_i>0)_{i}}\bigg(0,\nabla_{x,(l_1,l'_1,l_2,\ldots,l_s,l'_s)}(a_1,0,a_2,0,\ldots,a_s,0,0,m_1,0,m_2,\ldots,0,m_s)\bigg),
        \end{align*}}
        which, using~\eqref{E1} is equal to:
        \rescale{\begin{align*}
             &\left( \beta_{x,\r}(a_1,\dots,a_s), 0 \right) + \\ 
             &\sum_{(l_i+l'_i=r_i,l'_i>0)_{i}}\sum_{j=1}^{s}\bigg(0,\nabla_{x,(l_1,l'_1,l_2,\dots,l_s,l'_s)}(a_1,0,a_2,0,\dots,a_s,0,0,0,\dots,0,\underbrace{m_j}_{(2s+2j)^{\mbox{th}}},0,\dots,0)\bigg).
        \end{align*}}
        Using~\eqref{E3}, this is equal to:
        \rescale{\begin{align*}
             &\left( \beta_{x,\r}(a_1,\dots,a_s), 0 \right) + \\
             &\sum_{j=1}^s\sum_{l+l'=r_j,l'>0}(0,\nabla_{x,\r\circ_j(l,l')}(a_1,\dots,a_{j},0,a_{j+1},\dots,a_s,0,\dots,0,\underbrace{m_j}_{(s+j+2)^{\mbox{th}}},0,\dots,0)).
        \end{align*}}
        For all $j$, denote by $\sigma_j\in\∑_s$ the permutation that moves $j$ to the last spot and keeps all the other numbers in order, and denote by $\sigma_j^{*}\in\∑_n$ the block permutation of the blocks of size $(r_i)_i$ associated to $\sigma_j$. Using \ref{relperm}, one can see that 
        \begin{multline*}
            \nabla_{x,\r\circ_j(l,l')}(a_1,\dots,a_{j},0,a_{j+1},\dots,a_s,0,\dots,0,\underbrace{m_j}_{(s+j+2)^{\mbox{th}}},0,\dots,0)=\\\nabla_{\sigma_j^*x,\r^{\sigma_j}\circ_s(l,l')}(a_1,\dots,a_{j-1},a_{j+1},\dots,a_s,a_j,0,\dots,0,\underbrace{m_j}_{(2s+2)^{\mbox{th}}}).
        \end{multline*}
        For all $y\in\P(k)$, $b_1,\dots,b_{s-1}\in A$, $m\in M$ and $q_1+\dots+q_s=k$ such that $q_s>0$, we set:
        \begin{equation}
           \tag{$*$}\label{*} \beta_{y,\q}(b_1,\dots,b_{s-1},m) \dfn \nabla_{y,\q}(b_1,\dots,b_{s-1},0,\dots,0,\underbrace{m}_{(2q)^{\mbox{th}}}).
        \end{equation}
        From what precedes, we finally get the explicit $\Gamma(\P)$-algebra structure on $A\oplus M$ by:
        \begin{multline}
            \beta_{x,\r}\left((a_1,m_1),\dots,(a_s,m_s)\right) = \left( \beta_{x,\r}(a_1,\dots,a_s),0 \right) + \notag \\
            \sum_{j=1}^s\sum_{l+l'=r_j,l'>0}\left(0,\beta_{\sigma_j^*x,(\r^{\sigma_j})\circ_s(l,l')}(a_1,\dots,a_{j-1},a_{j+1},\dots,a_s,a_j,m_j)\right).\label{**}\tag{$**$}
        \end{multline}
        We claim that $M$ is an abelian $\Gamma(\P)$-algebra as defined in \ref{def:abeliangammapalg} and that the assignment~\eqref{*} endows $M$ with an $A$-module structure as defined in~\ref{def:Amodule}.

        \paragraph*{Let us show that $M$ is an abelian algebra:}
        Suppose $\r$ has at least two non-zero integers. For the sake of clarity we will assume that $s=2$. Then we want to prove that $\beta_{x,(r_1,r_2)}(m_1,m_2)=0$. In $B=A\oplus M$, this can be rewritten $\beta_{x,(r_1,r_2)}\left((0,m_1),(0,m_2)\right)$, which, using~\eqref{E1} and~\eqref{E3}, is equal to $0$.
        \paragraph*{Let us show that $M$ is an $A$-module:}
        Using the assignment~\eqref{*}, note that
        \[
            \beta_{y,\q}\left((b_1,0),\dots,(b_{s-1},0),(0,m)\right)=\left(0,\beta_{y,\q}(b_1,\dots,b_{s-1},m)\right).
        \]
        Then, the relations we have to verify for this assignment to equip $M$ with an $A$-module structure all are implied by the relations \ref{relperm} to \ref{relcomp} in $B$. 

        Only relations~\ref{relAM4} and~\ref{relAMspec} are not %completely 
        straightforward.

        Let us show that relation~\ref{relAM4} is satisfied. Using relation\ref{relperm}, it suffices to show~\ref{relAM4} for $i=s'$ Let $s,s'>0$, $\q\in\Comp_{s'}(n)$, and $\r\in \Comp_s(q_{s'})$. Let $x\in \P(n)^{\∑_{\q\circ_{s'}\r}}$. We want to show:
\[
    \sum_{k=1}^{q_{s'}}\beta_{\sum_{\rho\in \∑_1^{\times (n-q_{s'})}\times \overline E_{\r,k}}\rho x,\q\circ_{s'}(q_{s'}-k,k)}(a_1,\dots,a_{s'},m)=0.
\]
        On one hand, using the assignment~\eqref{**}, we have:
        \begin{multline*}
            \beta_{x,\q\circ_{s'} \r}((a_1,0),\dots,(a_{s'-1},0),\underbrace{(a_{s'},m),\dots,(a_{s'},m)}_{s}) = 
            \left( \beta_{x,\q\circ_{s'} \r}(a_1,\dots,a_{s'-1},\underbrace{a_{s'},\dots, a_{s'}}_{s}),0 \right) + \\
            \sum_{j=1}^s\sum_{l+l'=r_j,l'>0}\left(0,\beta_{\tau_j^*x,\q\circ_{s'}\left(\r^{\tau_j}\right)\circ_s(l,l')}(a_1,\dots,a_{s'-1},\underbrace{a_{s'},\dots,a_{s'}}_{s},m)\right),
        \end{multline*}
                where $\tau^*_j\in\∑_1^{\times(n-q_{s'})}\times\∑_{q_{s'}}$ denotes the block permutation for blocks of size $(r_1,\dots,r_s)$ associated with the transposition of $j$ and $s$ in $\∑_s$.
        Using relations~\ref{relrepet} on $A$,~\ref{relArepet} on $M$, and re-indexing $k=l'$, $l=r_j-k$ we then get:
        \begin{multline}\label{betaAM4.1}
            \beta_{x,\q\circ_{s'} \r}((a_1,0),\dots,(a_{s'-1},0),\underbrace{(a_{s'},m),\dots,(a_{s'},m)}_{s}) = 
            \left(\beta_{\sum_{\rho\in \∑_{\q}/\∑_{\q\circ_{s'}\r}}\sigma x,\q\circ_{s'} \r}(a_1,\dots,a_{s'-1},a_{s'}),0\right) + \\
            \sum_{j=1}^s\sum_{k=1}^{r_j}\left(0,\beta_{\sum_{\∑_{(\q\circ_{s'}\r)\circ_{s'+j-1}(r_j-k,k))}}\rho \tau_j^*x,\q\circ_{s'}(q_{s'}-k,k)}(a_1,\dots,a_{s'},m)\right).
        \end{multline}
        On the other hand, using relation \ref{relrepet} in $B$, we get:
%\begin{multline*}
	\[
         \beta_{x,\q\circ_{s'} \r}((a_1,0),\dots,(a_{s'-1},0),\underbrace{(a_{s'},m),\dots,(a_{s'},m)}_{s}) = 
         \beta_{\sum_{\sigma\in \∑_{\q}/\∑_{\q\circ_{s'}\r}}\sigma x, \q}\left((a_1,0),\dots,(a_{s'-1},0),(a_{s'},m)\right).
	\]
%\end{multline*}
Then, the assignment \eqref{**} gives:
\begin{multline}\label{betaAM4.2}
    \beta_{x,\q\circ_{s'} \r}((a_1,0),\dots,(a_{s'-1},0),\underbrace{(a_{s'},m),\dots,(a_{s'},m)}_{s}) = \\
    (\beta_{\sum_{\sigma\in \∑_{\q}/\∑_{\q\circ_{s'}\r}}\sigma x, \q}(a_1,\dots,a_{s'}),0) +  
    \sum_{k=1}^{q_j}(0,\beta_{\sum_{\sigma\in \∑_{\q}/\∑_{\q\circ_{s'}\r}}\sigma x, \q\circ_{s'}(q_{s'}-k,k)}(a_1,\dots,a_{s'},m)).
\end{multline}
Subtracting equation~\ref{betaAM4.2} from equation~\ref{betaAM4.1}, and projecting onto $M$, we get:
\begin{multline*}
    \left(\sum_{k=1}^{q_j}\beta_{\sum_{\sigma\in \∑_{\q}/\∑_{\q\circ_{s'}\r}}\sigma x, \q\circ_{s'}(q_{s'}-k,k)}(a_1,\dots,a_{s'},m)\right)-\\\left(\sum_{j=1}^s\sum_{k=1}^{r_j}\beta_{\sum_{\∑_{(\q\circ_{s'}\r)\circ_{s'+j-1}(r_j-k,k))}}\rho \tau_j^*x,\q\circ_{s'}(q_{s'}-k,k)}(a_1,\dots,a_{s'},m)\right)=0.
\end{multline*}
Observe that a set of representative for $\∑_{\q}/\∑_{\q\circ_{s'}\r}$ is given by $\∑_1^{\times n-q_{s'}}\times Sh(\r)$, and, if $k\le r_j$ a set of representative for 
\[
	\∑_{\q\circ_{s'} (q_{s'}-k,k)}/\∑_{(\q\circ_{s'}\r)\circ_{s'+j-1}(r_j-k,k))}
\]
is given by:
\[
    \∑_1^{\times n-q_{s'}}\times Sh(r_1,\dots,r_{j-1},r_{j+1},\dots,r_{s},r_j-k)\times\∑_1^{\times k}.
\]
Then, observe that, for all $j\in [s]$,
\begin{multline*}
     \{\rho\tau_j^*:\rho\in\∑_1^{\times n-q_{s'}}\times Sh(r_1,\dots,r_{j-1},r_{j+1},\dots,r_{s},r_j-k)\times\∑_1^{\times k}\}=\\\∑_1^{\times n-q_{s'}}\times\{\rho'\in Sh(\r):\rho'(r_1+\dots+r_j-k+1)=q_{s'}-k+1\},
\end{multline*}
since both sets are the sets of permutations of $\∑_n$ which shuffle the last blocks of size $r_1,\dots, r_s$ and send $n-q_{s'}+r_1+\dots+r_{j}-k+l$ on $n-k+l$ for all $l\in[k]$. Using relation \ref{relAMlin}, we then get:
\[
    \sum_{k=1}^{q_{s'}}\beta_{\sum_{\rho\in \∑_{1}^{\times n-q_{s'}}\times \overline{F}_{\r,k}}}\rho x,\q\circ_{s'}(q_{s'}-k,k)(a_1,\dots,a_{s'},m)=0,
\]
where 
\rescale{\begin{align*}
	\overline{F}_{\r,k} = Sh(\r)\setminus\left(\coprod_{j\in[s]}\{\rho'\in Sh(\r):\rho'(r_1+\dots+r_j-k+1)=q_{s'}-k+1\}\right).
\end{align*}}
Note that, if $k>r_j$, then $\{\rho'\in Sh(\r):\rho'(r_1+\dots+r_j-k+1)=q_{s'}-k+1\}=\emptyset$.

It remains to show that $\overline{F}_{\r,k}=\overline{E}_{\r,k}$. But, for all $\rho'\in Sh(\r)$, the condition for $\rho'$ not to satisfy $\rho'(r_1+\dots+r_j-k+1)=q_{s'}-k+1$ for any $j\in [s]$ is equivalent to asking that $(\rho')^{-1}(\{q_{s'}-k+1,\dots,q_{s'}\})$ intersects at least two of the blocks $(r_1,\dots,r_s)$, and this happens if and only if $\rho'$ satisfies 
\[
    \left|\{\rho'(r_1+\cdots+r_j):j\in[s]\}\cap\{q_{i}-k+1,\dots,q_{i}\}\right| \geq 2.
\]
Thus, we have shown that relation~\ref{relAM4} is satisfied.

Let us now show that relation~\ref{relAMspec} is satisfied. Let $r_1+\dots+r_s=n$, $x\in \P(n)^{\∑_\r}$, let $q_{1}+\dots+q_{u}=k$, $y\in\P(k)^{\∑_{\q}}$. Let $(a_i)_{i\in[s+u]}\in A^{\times s+u}$ and $m\in M$. Denote by $z=x\left(1_\P^{\times n-r_s},y^{\times r_s}\right)\in\P(n+r_s(k-1))$. 
        On the one hand, using \eqref{**}, observe that:
        \begin{align*}
            &\beta_{x,\r} \left( (a_1,0),\dots,(a_{s-1},0),\beta_{y,\q}((a_s,0),\dots,(a_{s+u-1},0),(a_{s+u},m)) \right) \\
            &= \beta_{x,\r} \Bigg( (a_1,0),\dots,(a_{s-1},0),\sum_{l+l'=q_s,l'>0}(\beta_{y,\q}(a_s,\dots,a_{s+u-1}), 
            \beta_{y,\q\circ_u(l,l')}(a_s,\dots,a_{s+u},m)) \Bigg) \\
            &= \left( \beta_{x,\r}(a_1,\dots,a_{s-1},\beta_{y,\q}(a_s,\dots,a_{s+u})),0 \right) + \\ 
            \quad &\sum_{\lambda,\lambda'} \Bigg( 0,\beta_{x,\r\circ_s(\lambda,\lambda')} \Bigg( a_1,\dots,a_{s-1},\beta_{y,\q}(a_s,\dots,a_{s+u}), 
            \sum_{l+l'=q_s,l'>0}\beta_{y,\q\circ_u(l,l')}(a_s,\dots,a_{s+u},m) \Bigg) \Bigg),
        \end{align*}
        where $(\lambda,\lambda')$ runs over the pairs of non-negative integers satisfying $\lambda+\lambda'=r_s,\lambda'>0$. Using relation~\ref{relcomp} on $A$ gives us:
        \rescale{\begin{align*}
            (\beta_{x,\r}(a_1,\dots,a_{s-1},\beta_{y,\q}(a_s,\dots,a_{s+u})),0)= \left(\beta_{\sum_\tau \tau z,\r\diamond_s(\q)}(a_1,\dots,a_{s+u}),0\right).
        \end{align*}}
        Using relation~\ref{relMsomme} and \ref{relAMcomp} on $M$, we get
        \rescale{\begin{align*}
        	&\sum_{\lambda,\lambda'}\left(0,\beta_{x,\r\circ_s(\lambda,\lambda')}\left(a_1,\dots,a_{s-1},\beta_{y,\q}(a_s,\dots,a_{s+u}),\sum_{l+l'=q_s,l'>0}\beta_{y,\q\circ_u(l,l')}(a_s,\dots,a_{s+u},m)\right)\right)\\
             &= \sum_{\lambda,\lambda'}\sum_{l+l'=q_s,l'>0}\left(0,\beta_{\sum_\sigma\sigma z,(\r\circ_s(\lambda,\lambda')\diamond((1)^{\times s-1},\q,\q\circ_u(l,l')))}(a_1,\dots,a_{s+u},a_s,\dots,a_{s+u},m)\right),
        \end{align*}}
        where $\sigma$ ranges over $\∑_{(\r\circ_s(\lambda,\lambda')\diamond_s(\q\circ_u(l,l')))}/\left(\prod_{i=1}^{s-1}\∑_{r_i}\right)\times \∑_{\lambda}\wr\∑_\q\times\∑_{\lambda'}\wr\∑_{\q\circ_u(l,l')}$ in the sum.
        Observe the repetition of input. Using relation~\ref{relArepet}, and noting that:
        \begin{multline*}
            \sum_{\sigma''\in\∑_{R_{\lambda\lambda'}^{ll'}}/\∑_{(\r\circ_s(\lambda,\lambda')\diamond((1)^{\times s-1},\q,\q\circ_u(l,l')))}} \sum_{\sigma\in\∑_{(\r\circ_s(\lambda,\lambda')\diamond_s(\q\circ_u(l,l')))}/\left(\prod_{i=1}^{s-1}\∑_{r_i}\right)\times \∑_{(\lambda,\lambda')}\wr\∑_{\q\circ_u(l,l')}} \sigma''\sigma z \\
            = \sum_{\sigma'\in \∑_{R_{\lambda\lambda'}^{ll'}}/\left(\prod_{i=1}^{s-1}\∑_{r_i}\right)\times \∑_{\lambda}\wr\∑_\q\times\∑_{\lambda'}\wr\∑_{\q\circ_u(l,l')}} \sigma'z,
        \end{multline*}
        we get
        \begin{multline*}
            \sum_{\lambda,\lambda'} \Bigg( 0,\beta_{x,\r\circ_s(\lambda,\lambda')} \Bigg( a_1,\dots,a_{s-1},\beta_{y,\q}(a_s,\dots,a_{s+u}), 
            \sum_{l+l'=q_s,l'>0}\beta_{y,\q\circ_u(l,l')}(a_s,\dots,a_{s+u},m) \Bigg) \Bigg) \\
            = \sum_{\lambda+\lambda'=r_s,\lambda'>0}\sum_{l+l'=q_{u},l'>0}\beta_{\sum_{\sigma'}\sigma' z,R_{\lambda\lambda'}^{ll'}}(a_1,\dots,a_{s+u},m).
        \end{multline*}               
        Finally, we obtain the equality:
        \begin{multline}\label{part1spec}
            \beta_{x,\r} \left( (a_1,0),\dots,(a_{s-1},0),\beta_{y,\q}((a_s,0),\dots,(a_{s+u-1},0),(a_{s+u},m)) \right) = \\
            \left(\beta_{\sum_\tau \tau z,\r\diamond_s(\q)}(a_1,\dots,a_{s+u}),0\right) + 
            \sum_{\lambda+\lambda'=r_s,\lambda'>0}\sum_{l+l'=q_{u},l'>0}\beta_{\sum_{\sigma'}\sigma' z,R_{\lambda\lambda'}^{ll'}}(a_1,\dots,a_{s+u},m).
        \end{multline}
        On the other hand, using relation~\ref{relcomp},
        \begin{multline*}
            \beta_{x,\r} \left( (a_1,0),\dots,(a_{s-1},0),\beta_{y,\q}((a_s,0),\dots,(a_{s+u-1},0),(a_{s+u},m)) \right)\\
            = \beta_{\sum_\tau \tau z,\r\diamond_s(\q)}((a_1,0),\dots,(a_{s+u-1},0),(a_{s+u},m)),
        \end{multline*}
        where $\tau$ ranges over $\∑_{\r\diamond_s\q}/\left(\prod_{i=1}^{s-1}\∑_{r_i}\right)\times\∑_{r_s}\wr\∑_{\q}$ in the sum. Using \eqref{**}, we then get:
        \begin{multline}\label{part2spec}
            \beta_{x,\r} \left( (a_1,0),\dots,(a_{s-1},0),\beta_{y,\q}((a_s,0),\dots,(a_{s+u-1},0),(a_{s+u},m)) \right) = \\
            \left(\beta_{\sum_\tau \tau z,\r\diamond_s(\q)}(a_1,\dots,a_{s+u}),0\right) + 
            \sum_{t+t'=r_sq_{u},t'>0} \left( 0,\beta_{\sum_\tau \tau z,\r\diamond_s(\q)\circ_{s+u}(t,t')}(a_1,\dots,a_{s+u},m) \right).
        \end{multline}
        Comparing the equalities~\eqref{part1spec} and \eqref{part2spec}, and projecting onto $M$, yields the relation~\ref{relAMspec}. 

        Assigning to each Beck module $\diag{B\ar[r]^{p}&A}$, the abelian $\Gamma(\P)$-algebra $\ker(p)$ with the above $A$ module structure provides a functor:
        \[
            \ker  \colon  \left(\Alg{\Gamma(\P)}/A\right)_{\ab} \to \lMod{A}.
        \]
        \paragraph{Let us now prove that the data of an $A$-module yields a Beck module over $A$.} Let $M$ be an $A$-module. Consider the vector space $B \dfn A\oplus M$, equipped with:
        \begin{enumerate}[label=(\arabic*)]
            \item \label{assignment1}For all $\r\in\Comp_s(n)$, $x\in \P(n)^{\∑_\r}$, an operation $\beta_{x,\r}$ defined as in \eqref{**},
            \item A map $pr \colon B\to A$ given by the projection orthogonally to $M$,
            \item A map $z \colon A\to B$ given by the natural injection,
            \item A map $\mu \colon B\times_AB\to A$ given by $\mu(a,m,m')=(a,m+m')$.
        \end{enumerate}
        We claim that:
        \begin{enumerate}[label=(\Alph*)]
            \item \label{condA}The assignment~\ref{assignment1} endows $B$ with a $\Gamma(\P)$-algebra structure,
            \item \label{condB}$pr$ is a $\Gamma(\P)$-algebra map,
            \item \label{condC}$z$ and $\mu$ are maps of $\Gamma(\P)$-algebras over $A$, and finally,
            \item \label{condD}$z$ and $\mu$ endow $B$ with the structure of a Beck module in $A$.
        \end{enumerate}
        Assuming~\ref{condA}, conditions~\ref{condB}, \ref{condC} and \ref{condD} are obvious. To prove~\ref{condA} we have to check that the operations defined in~\ref{assignment1} satisfy the conditions~\ref{relperm} to \ref{relcomp}. Using Notation~\ref{nota:perm}, we can rewrite~\eqref{**} as:
        \begin{multline*}
            \beta_{x,\r} \left( (a_1,m_1),\dots,(a_s,m_s) \right) = \left( \beta_{x,\r}(a_1,\dots,a_s), 0 \right) + \\
            \sum_{j=1}^s\sum_{l+l'=r_j,l'>0}\left(0,\beta_{x,\r\circ_j(l,l')}(a_1,\dots,a_{j},m_j,a_{j+1},\dots,a_s)\right).
        \end{multline*}
        Then, relation~\ref{relperm} is directly deduced from \ref{relperm} on $A$ and \ref{relAperm}. Relation \ref{rel0} is deduced using \ref{rel0} on $A$ and using the fact that there is no positive $l,l'$ with $l+l'=0$ such that $l'>0$. Relation \ref{rellambda} is a direct consequence of \ref{rellambda} on $A$, \ref{relAlambda} and \ref{relMlambda} on $M$.

        Let us prove relation~\ref{relrepet}. For clarity we can assume $s'=1$. Then $q_1=s$, and we need to prove that:
        \[
            \beta_{x,\r}\left((a,m),\dots,(a,m)\right)=\beta_{\sum_{\sigma\in\∑_n/\∑_\r}\sigma\cdot x,(n)}\left((a,m)\right).
        \]
        Using \eqref{**}, one has:
%        \begin{multline*}
	\[
            \beta_{\sum_{\sigma\in\∑_n/\∑_\r}\sigma\cdot x,(n)}\left((a,m)\right) = \left( \beta_{\sum_{\sigma\in\∑_n/\∑_\r}\sigma x,(n)}(a),0\right) + 
            \sum_{l+l'=n,l'>0}\left(0,\beta_{\sum_{\sigma\in\∑_n/\∑_\r}\sigma x,(l,l')}(a,m)\right),
	\]
%        \end{multline*}
        which, re-indexing $k=l'$, is equal to 
        \[
        		\left( \beta_{\sum_{\sigma\in\∑_n/\∑_\r}\sigma x,(n)}(a),0\right)+\sum_{k=1}^n\left(0,\beta_{\sum_{\sigma\in\∑_n/\∑_\r}\sigma x,(n-k,k)}(a,m)\right).
        \]
        Using relation \ref{relAM4}, this becomes: 
        \[
        		\left( \beta_{\sum_{\sigma\in\∑_n/\∑_\r}\sigma x,(n)}(a),0 \right) + \sum_{k=1}^n \left( 0,\beta_{\sum_{\rho}\rho x,(n-k,k)}(a,m) \right),
        	\]
        	where $\rho$ runs over the set of shuffles $Sh(\r)$ satisfying:
       \[
           \left|\{\rho'(r_1+\cdots+r_j):j\in[s]\}\cap\{q_{i}-k+1,\dots,q_{i}\}\right|=1.
       \] 
       From what precedes, this set is equal to $\coprod_{j\in[s]}\{\rho'\in Sh(\r):\rho'(r_1+\dots+r_j-k+1)=n-k+1\}$, which is equal to $\coprod_{j\in[s]}\{\rho\sigma_j^*:\rho\in Sh(r_1,\dots,r_{j-1},r_{j+1},\dots,r_{s},r_j-k)\times\∑_1^{\times k}\}$.

       The set $Sh(r_1,\dots,r_{j-1},r_{j+1},\dots,r_{s},r_j-k)\times\∑_1^{\times k}$ is empty if $k>r_j$, and otherwise, it is a set of representative for $\∑_{n-k}\times\∑_k/\∑_{\r^{\sigma_j}\circ_s(r_j-k,k)}$. So, the above sum is equal to:
       \begin{multline*}
            \beta_{x,\r}\left((a,m),\dots,(a,m)\right) = \\
            \left( \beta_{\sum_{\sigma\in\∑_n/\∑_\r}\sigma x,(n)}(a),0\right) + \sum_{j=1}^s\sum_{k=1}^n\left(0,\beta_{\sum_{\rho\in \∑_{n-k}\times\∑_k/\∑_{\r^{\sigma_j}\circ_s(r_j-k,k)}}\rho \sigma_j^*x,(n-k,k)}(a,m)\right).
       \end{multline*}
       Using relation \ref{relrepet} on $A$ and relation \ref{relArepet} on $M$, we get:
%\begin{multline*}
	\[
            \beta_{x,\r}\left((a,m),\dots,(a,m)\right) = 
            \left( \beta_{ x,\r}(a,\dots,a),0\right) + \sum_{j=1}^s\sum_{k=1}^n\left(0,\beta_{\sigma_j^*x,\r^{\sigma}\circ_s(r_j-k,s)}(a,\dots,a,m)\right),
	\]
%       \end{multline*}
       and, using the assignment \eqref{*}, we finally obtain:
       \[
            \beta_{x,\r}\left((a,m),\dots,(a,m)\right)=\beta_{\sum_{\sigma\in\∑_n/\∑_\r}\sigma\cdot x,(n)}\left((a,m)\right).
       \]
        Let us prove relation~\ref{relsomme}. For clarity, we will assume $s=2$. On the one hand, we have:
        \begin{multline*}
            \beta_{x,\r} \left( (a,m)+(b,m'),(a_2,m_2) \right) = \left( \beta_{x,\r}(a+b,a_2),0 \right) + \\
            \sum_{l+l'=r_1,l'>0}\left(0,\beta_{x,\r\circ_1(l,l')}(a+b,m+m',a_2)\right) + 
            \sum_{l+l'=r_2,l'>0}\left(0,\beta_{x,\r\circ_2(l,l')}(a+b,a_2,m_2)\right),
        \end{multline*}
         which, using relations~\ref{relsomme} on $A$, \ref{relMsomme} on $M$ and \ref{relAsomme} on $M$, is equal to:
        \rescale{\begin{align*}
             &\sum_{k_1+k_2=r_1} \left( \beta_{x,\r\circ_1(k_1,k_2)}(a,b,a_2),0 \right)\\
             +& \sum_{l+l'=r_1,l'>0} \sum_{k_1+k_2=l} \left( 0,\beta_{x,\r\circ_1(k_1,k_2,l')}(a,b,m,a_2)+\beta_{x,\r\circ_1(k_1,k_2,l')}(a,b,m',a_2) \right)\\
             +& \sum_{k_1+k_2=r_1} \sum_{l+l'=r_2,l'>0} \left( 0,\beta_{x,(\r\circ_2(l,l'))\circ_1(k_1,k_2)}(a,b,a_2,m_2) \right).
         \end{align*}}
         Note that the middle summand can be rewritten, changing the indices:
         %\resizebox{\linewidth}{!}{\begin{minipage}{\linewidth}
%\begin{align*}
	\[
		\sum_{k_1+k_2+k_3=r_1,k_3>0} \left( 0,\beta_{x,\r\circ_1(k_1,k_2,k_3)}(a,b,m,a_2)+\beta_{x,\r\circ_1(k_1,k_2,k_3)}(a,b,m',a_2) \right).
	\]
%\end{align*}
%\end{minipage}}\vspace{\baselineskip}
             On the other hand, note that the assignment~\eqref{**} gives:
         \begin{align*}
              &\sum_{k_1+k_2=r_1} \beta_{x,\r\circ_1(k_1,k_2)}\big((a,m),(b,m'),(a_2,m_2)\big) \\
              = &\sum_{k_1+k_2=r_1}(\beta_{x,\r\circ_1(k_1,k_2)}(a,b,a_2),0) 
              + \sum_{k_1+k_2=r_1}\sum_{l+l'=k_1,l'>0}\beta_{x,\r\circ_1(l,l',k_2)}(a,m,b,a_2),0) \\
              &+ \sum_{k_1+k_2=r_1}\sum_{l+l'=k_2,l'>0}\beta_{x,\r\circ_1(k_1,l,l')}(a,b,m',a_2),0) \\
              &+ \sum_{k_1+k_2=r_1}\sum_{l+l'=r_2,l'>0}\left(0,\beta_{x,(\r\circ_2(l,l'))\circ_1(k_1,k_2)}(a,b,a_2,m_2)\right).
         \end{align*}
         The sum of the second and third summands can be rewritten:
         \rescale{\begin{align*}
         	           \sum_{k_1+k_2+k_3=r_1,k_3>0}\left(0,\beta_{x,\r\circ_1(k_1,k_2,k_3)}(a,b,m,a_2) + \beta_{x,\r\circ_1(k_1,k_2,k_3)}(a,b,m',a_2)\right).
         \end{align*}}
	   This proves that relation~\ref{relsomme} holds.

       Relation~\ref{rellin} is a %
       consequence of relation~\ref{rellin} on $A$ and \ref{relAMlin} on $M$. Relation~\ref{relunit} is a %
       consequence of relation~\ref{relunit} on $A$, \ref{relAMcompM} and \ref{relAbunit} on $M$.

       Let us now prove relation~\ref{relcomp}. For simplicity, we assume that $s=1$. 
       We want to prove that:
%       \begin{multline*}
	\[
            \beta_{x,(n)}\left(\beta_{y,\q}((a_1,m_1),\dots,(a_{u},m_{u}))\right) = 
            \beta_{\sum_{\tau}\tau\cdot x\left(y^{\times n}\right),(n)\diamond(\q)} \left( (a_1,m_1),\dots,(a_{u},m_{u}) \right),
	\]
%        \end{multline*}
        where $\tau$ ranges over $\∑_{(n)\diamond(\q)}/\∑_{(n)}\wr\∑_{\q}$ in the sum. Here, $x\in\P(n)^{\∑_n}$, $\q\in\Comp_{u}(k)$, and $y\in\P(k)^{\∑_{\q}}$. 
        On the one hand, using the assignment \eqref{**},
\begin{multline*}
     \beta_{x,(n)} \left( \beta_{y,\q}((a_1,m_1),\dots,(a_{u},m_{u})) \right) =
      \left( \beta_{x,(n)}\left(\beta_{y,\q}(a_1,\dots,a_{u})\right),0 \right) + \\
%    \rescale{\begin{align*}
     \sum_{\lambda+\lambda'=n,\lambda'>0} \left( 0,\beta_{x,(\lambda+\lambda')} \left( \beta_{y,\q}(a_1,\dots,a_{u}),\sum_{j=1}^u\sum_{l+l'=q_j,l'>0}\beta_{y,\q\circ_j(l,l')}(a_1,\dots,a_j,m_j,a_{j+1},\dots,a_u) \right) \right).
%    \end{align*}}
\end{multline*}
Using relation~\ref{relcomp} on $A$ and \ref{relMsomme} on $M$, we then get:
\begin{multline}\label{spec3}
	\beta_{x,(n)}\left(\beta_{y,\q}((a_1,m_1),\dots,(a_{u},m_{u}))\right) = 
	\left(\beta_{\sum_{\tau}\tau\cdot x\left(y^{\times n}\right),(n)\diamond(\q)}(a_1,\dots,a_u),0\right) + \\
	\sum_{j=1}^u\sum_{\lambda+\lambda'=n,\lambda'>0}\sum_{l+l'=q_j,l'>0}\bigg(0,\beta_{x,(\lambda+\lambda')}\bigg(\beta_{y,\q}(a_1,\dots,a_{u}), \beta_{y,\q\circ_j(l,l')}(a_1,\dots,a_j,m_j,a_{j+1},\dots,a_u)\bigg)\bigg)
\end{multline}
On the other hand, using the assignment \eqref{**},
\begin{multline}\label{spec4}
    \beta_{\sum_{\tau}\tau\cdot x\left(y^{\times n}\right),(n)\diamond(\q)}((a_1,m_1),\dots,(a_{u},m_{u})) = 
    \left(\beta_{\sum_{\tau}\tau\cdot x\left(y^{\times n}\right),(n)\diamond(\q)}(a_1,\dots,a_u),0\right) + \\
    \sum_{j=1}^{u}\sum_{t+t'=nq_j,t'>0}\left(0,\beta_{\sum_{\tau}\tau\cdot x\left(y^{\times n}\right),\left((n)\diamond(\q)\right)\circ_j(t,t')}(a_1,\dots,a_{j},m_j,a_{j+1},\dots,a_u)\right).
\end{multline}
Comparing the expressions~\eqref{spec3} and~\eqref{spec4}, we then have left to prove that in $M$, for all $j\in[u]$,
\begin{multline*}
    \sum_{\lambda+\lambda'=n,\lambda'>0}\sum_{l+l'=q_j,l'>0}\beta_{x,(\lambda+\lambda')}\bigg(\beta_{y,\q}(a_1,\dots,a_{u}), 
    \beta_{y,\q\circ_j(l,l')}(a_1,\dots,a_j,m_j,a_{j+1},\dots,a_u)\bigg) = \\ 
    \sum_{t+t'=nq_j,t'>0}\beta_{\sum_{\tau}\tau\cdot x\left(y^{\times n}\right),\left((n)\diamond(\q)\right)\circ_j(t,t')}(a_1,\dots,a_{j},m_j,a_{j+1},\dots,a_u).
\end{multline*}
We can, without loss of generality, suppose that $j=u$. Then, we are left to prove that:
\rescale{\begin{align*}
    &\sum_{\lambda+\lambda'=n,\lambda'>0}\sum_{l+l'=q_u,l'>0}\beta_{x,(\lambda+\lambda')}\left(\beta_{y,\q}(a_1,\dots,a_{u}),\beta_{y,\q\circ_u(l,l')}(a_1,\dots,a_u,m_u)\right)\\&=\sum_{t+t'=nq_u,t'>0}\beta_{\sum_{\tau}\tau\cdot x\left(y^{\times n}\right),\left((n)\diamond(\q)\right)\circ_j(t,t')}(a_1,\dots,a_{u},m_u).
\end{align*}}
Using relation~\ref{relAMcomp} on the left hand side of the equality gives us:
\rescale{\begin{align*}
    &\sum_{\lambda+\lambda'=n,\lambda'>0}\sum_{l+l'=q_u,l'>0}\beta_{x,(\lambda+\lambda')}\left(\beta_{y,\q}(a_1,\dots,a_{u}),\beta_{y,\q\circ_u(l,l')}(a_1,\dots,a_u,m_u)\right)=\\
    &\sum_{\lambda+\lambda'=n,\lambda'>0}\sum_{l+l'=q_u,l'>0}\beta_{\sum_{\sigma}x(y^{\times n}),(\lambda,\lambda')\diamond(\q,\q\circ_u(l,l'))}(a_1,\dots,a_u,a_1,\dots,a_u,m_u),
    \end{align*}}
    where $\sigma$ ranges over $\∑_{(\lambda,\lambda')\diamond(\q,\q\circ_u(l,l'))}/\∑_{\lambda}\wr\∑_\q\times\∑_{\lambda'}\wr\∑_{\q\circ_u(l,l')}$ in the sum.

    Denote by $R_{\lambda,\lambda'}^{l,l'}$ the partition of $kn$ into $u+1$ parts such that, for all $i\in[u-1]$, $(R_{\lambda,\lambda'}^{l,l'})_i=\left((n)\diamond(\q)\right)_i$, and such that:
%    \begin{multline*}
	\[
        (R_{\lambda,\lambda'}^{l,l'})_u = \left\{ k-q_u+\alpha k+\gamma,\alpha\in[\lambda], \gamma\in[q_u] \right\} \cup 
        \left\{ \lambda k+(k-q_u)+\alpha k+\gamma, \alpha\in[\lambda'],\gamma\in[l] \right\}
	\]
%    \end{multline*}
    \[
        (R_{\lambda,\lambda'}^{l,l'})_{u+1} = \left\{ \lambda k+(k-l')+\alpha k+\gamma,\alpha\in[\lambda'],\gamma\in[l'] \right\}.
    \]
    Then, using relation~\ref{relArepet}, and noting that:
%    \rescale{
	\begin{align*}
            &\sum_{\sigma''\in\∑_{R_{\lambda\lambda'}^{ll'}}/\∑_{(\lambda,\lambda')\diamond(\q,\q\circ_u(l,l'))}}\sigma \sum_{\∑_{(\lambda,\lambda')\diamond(\q,\q\circ_u(l,l'))}/\∑_{\lambda}\wr\∑_\q\times\∑_{\lambda'}\wr\∑_{\q\circ_u(l,l')}}\sigma x(y^{\times n}) \\
            &= \sum_{\sigma'\in \∑_{R_{\lambda\lambda'}^{ll'}}/\∑_{\lambda}\wr\∑_\q\times\∑_{\lambda'}\wr\∑_{\q\circ_u(l,l')}}\sigma'x(y^{\times n}), 
	\end{align*}
        %}
    we then get:
%    \rescale{
	\begin{align*}
        &\sum_{\lambda+\lambda'=n,\lambda'>0} \sum_{l+l'=q_u,l'>0}\beta_{x,(\lambda+\lambda')}\left(\beta_{y,\q}(a_1,\dots,a_{u}),\beta_{y,\q\circ_u(l,l')}(a_1,\dots,a_u,m_u)\right) \\
        &= \sum_{\lambda+\lambda'=n,\lambda'>0}\sum_{l+l'=q_u,l'>0}\beta_{\sum_{\sigma'}\sigma' x(y^{\times n}),R_{\lambda\lambda'}^{ll'}}(a_1,\dots,a_u,m_u).
	\end{align*}
	%}
    We have left to prove the equality:
    \begin{multline*}
        \sum_{\lambda+\lambda'=n,\lambda'>0} \sum_{l+l'=q_u,l'>0} \beta_{\sum_{\sigma'}\sigma' x(y^{\times n}),R_{\lambda\lambda'}^{ll'}}(a_1,\dots,a_u,m_u)\\
        = \sum_{t+t'=nq_u,t'>0} \beta_{\sum_{\tau}\tau\cdot x\left(y^{\times n}\right),\left((n)\diamond(\q)\right)\circ_j(t,t')}(a_1,\dots,a_{u},m_u).
    \end{multline*}
    But this is a particular case of the relation~\ref{relAMspec}.

    Assigning to each $A$-module $M$ the Beck module $A\ltimes M$ obtained by equipping the vector space $A\oplus M$ with the divided power $\P$-algebra structure given by \eqref{**} provides a functor:
    \[
            A\ltimes -  \colon  \lMod{A} \to \left(\Alg{\Gamma(\P)}/A\right)_{\ab},
    \]
    One can readily check that $\ker$ and $A\ltimes -$ are inverse to each other.
\end{proof}
    \begin{defi}
           If $M$ is an $A$-module, the Beck module $A\ltimes M$ defined above is called the \Def{semidirect product} of $A$ and $M$.
    \end{defi}
Let us now turn to the particular case of abelian $\Gamma(\P)$-algebras. 
Abelian group objects in the category of $\Gamma(\P)$-algebras are 
Beck modules over the terminal object. %
Since $\P$ is reduced, this terminal object is always the zero algebra $0$. Then, we obtain the following:
\begin{coro}\label{coro:abelianGammaPalg}
There is an equivalence of categories between the category of abelian group objects in $\Gamma(\P)$-algebras %
and the category of abelian $\Gamma(\P)$-algebras in the sense of Definition~\ref{def:abeliangammapalg}, that is:
\[
(\Alg{\Ga(\P)})_{\ab} \cong (\Alg{\Ga(\P)})_{\mathrm{Ab}}.
\]    
\end{coro}
\begin{proof}
    Applying Theorem~\ref{thm:Amodule} in the case $A=0$, the data of %
an abelian group object in $\Gamma(\P)$-algebras 
is equivalent to the data of a module over the $\Gamma(\P)$-algebra $0$. Such a $0$-module $M$ is an abelian $\Gamma(\P)$-algebra equipped with operations $\beta_{x,(r_1,\dots,r_s)}(0,\dots,0,-) \colon M \to M$. If $s>1$ and if at least one of the $r_i$ for $i<s$ is non-zero, then relation~\ref{relAlambda} implies that  $\beta_{x,(r_1,\dots,r_s)}(0,\dots,0,m)=0$ for all $M$. If $s>1$ and $r_i=0$ for all $i<s$, relation~\ref{relA0} shows that $\beta_{x,(r_1,\dots,r_s)}(0,\dots,0,m)=\beta_{x,(r_s)}(m)$. Finally, when $s=1$, relation~\ref{relAMcompM} shows that $\beta_{x,(r_s)}(m)=\beta_x(m)$ is provided by the structure of abelian $\Gamma(\P)$-algebra on $M$. Thus, a $0$-module is an abelian $\Gamma(\P)$-algebra equipped with additionnal operations which are all trivial. %
\end{proof}

\section{Universal enveloping algebra}\label{sec:Universalalgebra}

In this section, we 
show that the category of modules over a $\Gamma(\P)$-algebra $A$ is equivalent to the category of left modules over an associative %
ring 
$\UU_{\Gamma(\P)}(A)$; see %
Proposition~\ref{prop:univenvalg}. In practice, %
$\UU_{\Gamma(\P)}(A)$ is %
the %
ring 
containing all the operations $\beta_x(-)$ in the definition of an abelian $\Gamma(\P)$-algebra (see Definition~\ref{def:abeliangammapalg}), and all the operations $\beta_{x,\r}(a_1,\dots,a_{s-1},-)$ in the definition of an $A$-module (see Definition~\ref{def:Amodule}). Here, the hyphen ``$-$'' is considered as a placeholder. %
The multiplication $\mu$ of the %
ring 
$\UU_{\Gamma(\P)}(A)$ is %
the composition of %
operations. %

For this section, we fix the base field $\F$ of characteristic $p$.

        \begin{defi}\label{def:UnivEnvAlg}
            Let $\UU_{\Gamma(\P)}(A)$ be the quotient of the $\F$-vector space spanned by the set of symbols $\beta_{x,\r}(a_1,\dots,a_{s-1},-)$, for all $\r=(r_1,\dots,r_s)$, such that $r_1+\dots+r_s=n$, $x\in\P(n)^{\∑_\r}$, and $a_1,\dots,a_{s-1}\in A$, by the vector subspace spanned by the elements:
            \begin{enumerate}
                \item %\resizebox{\linewidth}{!}{
                $\beta_{\lambda x+y,\r}(a_1,\dots,a_{s-1},-)-\left(\lambda\beta_{x,\r}(a_1,\dots,a_{s-1},-)+\beta_{y,\r}(a_1,\dots,a_{s-1},-)\right)$
                %} 
                for all $x,y\in\P(n)^{\r}$, $\lambda\in\F$.
                \item $\beta_{x,\r}(a_1,\dots,a_{s-1},-)-\beta_{\rho^*\cdot x,\r^\rho}(a_{\rho^{-1}(1)},\dots,a_{\rho^{-1}(s-1)},-)$ for all $\rho\in\∑_{s-1}$, where $\rho^*$ denotes the block permutation with blocks of size $(r_i)$ associated to $\rho$.
                \item $\beta_{x,(0,r_1,r_2,\dots,r_s)}(a_0,a_1,\dots,a_{s-1},-)-\beta_{x,(r_1,r_2,\dots,r_s)}(a_1,\dots,a_{s-1},-)$.
                \item $\beta_{x,\r}(\lambda a_1,a_2,\dots,a_{s-1},m)-\lambda^{r_1}\beta_{x,\r}(a_1,\dots,a_{s-1},m) \quad \text{for all } \lambda\in \F$.
                \item For all $\r\in \Comp_{s}(n)$ and $\q\in \Comp_{s'}(s-1)$, the element
%                \begin{multline*}
	\[
                \beta_{x,\r}(\underbrace{a_1,\dots,a_1}_{q_1},\underbrace{a_2,\dots,a_2}_{q_2},\dots,\underbrace{a_{s-1},\dots,a_{s-1}}_{q_{s'}},-) 
                - \beta_{\big(\sum_{\sigma\in \∑_{\q'\rhd\r}/\∑_\r}\sigma\cdot x\big),\ \q'\rhd\r}(a_1,a_2,\dots,a_{s-1},-),
	\]
%                \end{multline*}
                where $\q'=(\q,0)$.
                \item %\resizebox{\linewidth}{!}{
                $\beta_{x,\r}(a_0+a_1,\dots,a_{s-1},-)-\left(\sum_{l+l'=r_1}\beta_{x,\r\circ_1(l,l')}(a_0,a_1,\dots,a_{s-1},-)\right)$.
                %}
                \item \label{item:InvariantSum} $\beta_{x,\r}(a_1,\dots,a_{s-1},-)$ for all $x$ such that there exists $q_1+\dots+q_u=r_s$, $y\in\P(n)^{\∑_{\r\circ_s\q}}$, such that $x=\sum_{\sigma\in \∑_{n-r_s}\times\∑_{r_s}/\∑_{n-r_s}\times \∑_{\q}}\sigma y$, and at least two of the $q_j$'s are non-zero.
                 \item For all $s,s'>0$, $\q\in\Comp_{s'}(n)$, $\r\in \Comp_s(\q)$, $x\in \P(n)^{\∑_{\q\circ_i\r}}$, the sum:
            \begin{equation*}
                \sum_{k=1}^{q_i}\beta_{\sum_{\rho\in \∑_1^{\times (q_1+\cdots+q_{i-1)}}\times \overline E_{\r,k}\times \∑_1^{\times(q_{i+1}+\cdots+q_{s'})}}\rho\cdot \sigma_i^* x,\q\circ_i(q_i-k,k)}(a_1,\dots,a_{s'},-),
            \end{equation*}
            where $\sigma_i^*$ is the block permutation for blocks of size $q_1,q_2,\dots,q_{s'}$ associated to the transposition of $i$ and $s'$ in $\∑_{s'}$, and where:
            \rescale{\begin{align*}
            	                \overline E_{\r,k}=\left\{\rho'\in Sh(\r): \left|\{\rho'(r_1+\cdots+r_j):j\in[s]\}\cap\{q_{i}-k+1,\dots,q_{i}\}\right|\ge 2\right\},
            \end{align*}}
            where $Sh(\r)$ is the set of $(r_1,\dots,r_s)$-shuffles.
                \item For all $r_1+\dots+r_s=n$, $x\in \P(n)^{\∑_{\r}}$, let $q_{1}+\dots+q_{u}=k$, $y\in\P(k)^{\∑_{\q}}$, $(a_i)_{i\in[s+u]}\in A^{\times s+u}$, the element:
                \begin{multline*}
                    \left(\sum_{t+t'=r_sq_{u},t'>0}\beta_{\sum_\tau \tau z,\r\diamond_s(\q)\circ_{s+u}(t,t')}(a_1,\dots,a_{s+u},-)\right)- 
                    \\
                \left(\sum_{\lambda+\lambda'=r_s,\lambda'>0}\sum_{l+l'=q_{u},l'>0}\beta_{\sum_{\sigma'}\sigma' z,R_{\lambda\lambda'}^{ll'}}(a_1,\dots,a_{s+u},-)\right),
                \end{multline*}
                where $z=x\left(1_\P^{\times n-r_s},y^{\times r_s}\right)\in\P(n+r_s(k-1))$, $R_{\lambda\lambda'}^{ll'}$ is defined in \ref{def:Amodule}, where $\tau$ ranges over $\∑_{\r\diamond_s\q}/\left(\prod_{i=1}^{s-1}\∑_{r_i}\right)\times\∑_{r_s}\wr\∑_{\q}$ in the sum, and where $\sigma'$ ranges over $\∑_{R_{\lambda\lambda'}^{ll'}}/\left(\prod_{i=1}^{s-1}\∑_{r_i}\right)\times\∑_{\lambda}\wr\∑_\q\times\∑_{\lambda'}\wr\∑_{\q\circ_u(l,l')}$ in the sum.
            \end{enumerate}
        \end{defi}
        \begin{nota}
            Following Notation~\ref{nota:perm}, we will use the notation: 
			\[            
			\beta_{x,\r}(a_1,\dots,a_{i-1},-,a_{i+1},\dots,a_s) \dfn \beta_{\sigma_i^*,\r^{\sigma_i}}(a_1,\dots,a_{i-1},a_{i+1},\dots,a_s,-).
			\]
        \end{nota}
        
        \begin{prop}\label{prop:EnvAlgpowerp}
            The vector space $\UU_{\Gamma(\P)}(A)$ is %
            spanned by the symbols $\beta_{x,\r}(a_1,\dots,a_{s-1},-)$ where all the integers in $\r$ are powers of $p$, and $\beta_{x,\r}(a_1,\dots,a_{s-1},-)=0$ in $\UU_{\Gamma(\P)}(A)$ if $r_s$ is not a power of $p$.
        \end{prop}
        
        \begin{proof}
            This is equivalent to Proposition~\ref{prop:AMpowersp}.
        \end{proof}
        
        Viewing the $\F$-vector space structure of $\UU_{\Gamma(\P)}(A)$ as a left $\F$-action, 
        we endow $\UU_{\Gamma(\P)}(A)$ with a right $\F$-action by setting, for all $\lambda\in\F$,
        	\[
            	\beta_{x,\r}(a_1,\dots,a_{s-1},-)\cdot \lambda = \lambda^{r_s} \beta_{x,\r}(a_1,\dots,a_{s-1},-).
        	\]
            
        \begin{rema}
        	It is important here to restrict ourselves to the case where $r_s$ is a power of $p$. Otherwise, the preceding definition is not linear in $\lambda\in\F$. This is made possible by Proposition~\ref{prop:EnvAlgpowerp}.
        \end{rema}

        We then equip $\UU_{\Gamma(\P)}(A)$ with a multiplication:
        	\[
        		\mu \colon \UU_{\Gamma(\P)}(A) {}_\F\!\otimes_\F \UU_{\Gamma(\P)}(A) \to \UU_{\Gamma(\P)}(A)
        	\]
        by setting:
            \begin{multline*}
                \mu\left(\beta_{x,\r}(a_1,\dots,a_{s-1},-)\otimes  \beta_{y,\q}(b_1,\dots,b_{u-1},-)\right) = \\
                \beta_{\sum_{\tau}\tau\cdot x(1_{\P}^{n-r_s},y^{\times r_s}),\r\diamond_s\q_s}(a_1,\dots,a_{s-1},b_1,\dots,b_{u-1},-),
            \end{multline*}
            where $\tau$ ranges over $\∑_{\r\diamond_s\q_s}/\left(\prod_{i=1}^{s-1}\∑_{r_i}\right)\times\∑_{r_s}\wr\∑_\q$ in the sum, and extending $\mu$ %
            as an $\F$-bimodule homomorphism, which is well-defined in light of the equations in Definition~\ref{def:UnivEnvAlg}. 
            Note that we used both the right and left $\F$-actions on $\UU_{\Gamma(\P)}(A)$, that is: $(x \cdot \la) \ot y = x \otimes (\la y)$ in the tensor product. 
            
        \begin{lemm}
            The multiplication $\mu$ endows $\UU_{\Gamma(\P)}(A)$ with the structure of a %
            unital ring.
        \end{lemm}
        \begin{proof}
            Let $V$ be the $1$-dimensional $\F$-vector space spanned by $v$. Consider the $\Gamma(\P)$-algebra $B$ obtained as a quotient of the free $\Gamma(\P)$-algebra $\Gamma(\P,A\oplus V)$ by the $\Gamma(\P)$-ideal generated by the two following family of elements:
            \begin{itemize}
                \item $\beta_{x,\r}(a_1,\dots,a_{s-1},v)$ for all $x$ such that there exists $q_1+\dots+q_u=r_s$, 
                %\resizebox{\linewidth}{!}{
                $y\in\P(n)^{\∑_{r_1}\times\dots\times\∑_{r_{s-1}}\times\∑_\q}$ such that $x=\sum_{\sigma\in \∑_{n-r_s}\times\∑_{r_s}/\∑_{n-r_s}\times \∑_q}\sigma y$,
                %}
                \item For all $r_1+\dots+r_s=n$, $x\in \P(n)^{\∑_{\r}}$, let $q_{1}+\dots+q_{u}=k$, $y\in\P(k)^{\∑_{\q}}$, $(a_i)_{i\in[s+u]}\in A^{\times s+u}$, the element:
                \begin{multline*}
                    \left(\sum_{t+t'=r_sq_{u},t'>0}\beta_{\sum_\tau \tau z,\r\diamond_s(\q)\circ_{s+u}(t,t')}(a_1,\dots,a_{s+u},v)\right) - \\
                		\left(\sum_{\lambda+\lambda'=r_s,\lambda'>0}\sum_{l+l'=q_{u},l'>0}\beta_{\sum_{\sigma'}\sigma' z,R_{\lambda\lambda'}^{ll'}}(a_1,\dots,a_{s+u},v)\right),
                \end{multline*}
                where $z=x\left(1_\P^{\times n-r_s},y^{\times r_s}\right)\in\P(n+r_s(k-1))$, $R_{\lambda\lambda'}^{ll'}$ is defined in Definition~\ref{def:Amodule}, where $\tau$ ranges over $\∑_{\r\diamond_s\q}/\left(\prod_{i=1}^{s-1}\∑_{r_i}\right)\times\∑_{r_s}\wr\∑_{\q}$ in the sum, and where $\sigma'$ ranges over $\∑_{R_{\lambda\lambda'}^{ll'}}/\left(\prod_{i=1}^{s-1}\∑_{r_i}\right)\times\∑_{\lambda}\wr\∑_\q\times\∑_{\lambda'}\wr\∑_{\q\circ_u(l,l')}$ in the sum. 
            \end{itemize}
            There is a linear map $f \colon \UU_{\Gamma(\P)}(A)\to B$ such that:
            \[
                f(\beta_{x,\r}(a_1,\dots,a_{s-1},-))=\beta_{x,\r}(a_1,\dots,a_{s-1},v).
            \]
            Indeed, the relations \ref{relperm} to \ref{relcomp} on $B$, and quotienting by the above family of elements ensure that $f$ is well defined. The map $f$ is clearly injective. Note that, for all pair of elements in $\UU_{\Gamma(\P)}(A)$ of the type:
            \[
                \t=\lambda\beta_{x,\r}(a_1,\dots,a_{s-1},-),\quad \mathfrak s=\lambda'\beta_{y,\q}(b_1,\dots,b_{u-1},-),
            \]
            One has:
            \[
                f(\mu(\t\otimes\mathfrak s))=\lambda\beta_{x,\r}(a_1,\dots,a_{s-1},\lambda'\beta_{y,\q}(b_1,\dots,b_{u-1},v)).
            \]
            So, the associativity of $\mu$ is a consequence of the associativity of the composition of operations $\beta$ in $B$, which is guaranteed by the associativity of the monad $\Gamma(\P)$.

            The element $\beta_{1_\P,(1)}(-)$ is clearly a unit for $\mu$.
        \end{proof}
        
\begin{warning}
The ring $\UU_{\Gamma(\P)}(A)$ is an $\F$-vector space but need \emph{not} be an $\F$-algebra. The multiplication $\mu$ satisfies:
\begin{align*}
&\mu \left( \beta_{x,\r}(a_1,\dots,a_{s-1},-) \otimes \left( \la \beta_{y,\q}(b_1,\dots,b_{u-1},-) \right) \right)\\ 
&= \mu \left( \left( \beta_{x,\r}(a_1,\dots,a_{s-1},-) \cdot \la \right) \otimes \beta_{y,\q}(b_1,\dots,b_{u-1},-) \right)\\ 
&= \mu \left( \left( \la^{r_s} \beta_{x,\r}(a_1,\dots,a_{s-1},-) \right) \otimes \beta_{y,\q}(b_1,\dots,b_{u-1},-) \right)\\
&= \la^{r_s} \mu \left( \left( \beta_{x,\r}(a_1,\dots,a_{s-1},-) \right) \otimes \beta_{y,\q}(b_1,\dots,b_{u-1},-) \right).
\end{align*}
A %
concrete example is given by the enveloping algebra $V(A) = \UU_{\Gamma(\Com)}(A)$ studied in Section~\ref{sec:ExCom}. 
In the special case $\F = \F_p$, the equality $\la^{r_s} = \la$ holds (since $r_s$ is a power of $p$), and thus $\UU_{\Gamma(\P)}(A)$ \emph{is} an $\F$-algebra.
\end{warning}        
        
        \begin{nota}
            We will allow the notation:
            \[
                \beta_{x,\r}(a_1,\dots,a_{j-1},-,a_{j+1},\dots,a_s) \dfn \beta_{\sigma_jx,\r^{\sigma_j^*}}(a_1,\dots,a_{j-1},a_{j+1},\dots,a_s,-)
            \]
            in $\UU_{\Gamma(\P)}(A)$.
        \end{nota}

        \begin{defi}
            Denote by $\UU_{\mathrm{Ab}}(\Gamma(\P))\subset\UU_{\Gamma(\P)}(A)$ the vector subspace spanned by the elements $\beta_{x,(n)}(-)$. Note that $\UU_{\mathrm{Ab}}(\Gamma(\P))$ is a unital %
            subring %
            and does not depend on $A$.
        \end{defi}
        Note that any abelian $\Gamma(\P)$-algebra $M$ can be equipped with a left $\UU_{\mathrm{Ab}}(\Gamma(\P))$ action by:
        \[
            \beta_{x,(n)}\otimes m\mapsto \beta_x(m),
        \]
        and conversely, any left $\UU_{\mathrm{Ab}}(\Gamma(\P))$-module $M$ has a structure of abelian $\Gamma(\P)$-algebra given by:
        \[
            \beta_{x}(m) \dfn \beta_{x,(n)}(-)\cdot m.
        \]
        Since those assignments are %
        inverse to one another, we obtain:
        \begin{prop}\label{prop:univenvalg}
            The category of abelian $\Gamma(\P)$-algebra is equivalent to the category of left modules over $\UU_{\mathrm{Ab}}(\Gamma(\P))$. 
        \end{prop}

        \begin{coro}
            Let $V$ be an $\F$-vector space. Then, $\UU_{\mathrm{Ab}}(\Gamma(\P))\otimes V$ is the free abelian $\Gamma(\P)$-algebra generated by $V$.
        \end{coro}
        Similarly, any $A$-module $M$ can be equipped with a left $\UU_{\Gamma(\P)}(A)$ action by:
        \[
            \beta_{x,\r}(a_1,\dots,a_{s-1},-)\otimes m\mapsto \beta_{x,\r}(a_1,\dots,a_{s-1},m),
        \]
        and conversely, any left $\UU_{\Gamma(\P)}(A)$-module $M$ has an $A$-module structure given by:
        \[
            \beta_{x,\r}(a_1,\dots,a_{s-1},m) \dfn \beta_{x,\r}(a_1,\dots,a_{s-1},-) \cdot m.
        \]
        Once again, those assignments are %
        inverse to one another, and we obtain:
        \begin{theo}\label{thm:BeckMod}
            The category of $A$-modules is equivalent to the category of left modules over $\UU_{\Gamma(\P)}(A)$. 
        \end{theo}

        \begin{coro}
            Let $V$ be an $\F$-vector space. Then, $\UU_{\Gamma(\P)}(A) \otimes_{\F} V$ is the free $A$-module generated by $V$.
        \end{coro}
        \begin{coro}
            Let $M$ be an  abelian $\Gamma(\P)$-algebra. Then 
            \[
            		\UU_{\Gamma(\P)}(A) \otimes_{\UU_{\mathrm{Ab}(\Gamma(\P))}} M
            \]
            is the free $A$-module generated by $M$.
        \end{coro}
        The sequence of ring homomorphisms $\F \to \UU_{\mathrm{Ab}(\Gamma(\P))} \to \UU_{\Gamma(\P)}(A)$ then yields a commutative diagram of free/forgetful adjunctions:
        \[
            \diag @C+1.5cm @R+1cm {
            \Vect{\F} \ar@<3pt>[r]^-{\UU_{\mathrm{Ab}(\Gamma(\P))}\otimes -} \ar@<3pt>[rd]^-{\UU_{\Gamma(\P)}(A)\otimes-} &(\Alg{\Gamma(\P)})_{\mathrm{Ab}} \ar@<3pt>[d]^{\UU_{\Gamma(\P)}(A)\otimes_{\UU_{\mathrm{Ab}(\Gamma(\P))}}-} \ar@<3pt>[l]\\
            & \lMod{A}\ar@<3pt>[u]\ar@<3pt>[ul]
            }
        \]
        where the right adjoints are usual restriction of scalars.
        
Recall from \cite{LodayV12}*{\S 12.3.4} the description for the enveloping algebra $\UU_\P(A)$ associated to a $\P$-algebra (without divided powers) $A$. It has a set of generators denoted by $\nu(a_1,\dots,a_k;1)$ in \cite{LodayV12}, for all $\nu\in\P(k+1)$, $a_1,\dots,a_k\in A$. 
The norm map induces a monad morphism $\Tr \colon S(\P) \to \Gamma(\P)$; see \cite{Fresse00}*{\S 1.1.14}. This map has been translated in terms of divided power operations by the third author in \cite{Ikonicoff23}*{Remark~6.3}. In consequence, we get the following: 

\begin{prop}\label{prop:MapEnvAlg}
    For all $\Gamma(\P)$-algebra $A$, the norm map induces a natural $\F$-linear ring homomorphism 
$\te \colon \UU_{\P}(A)\to \UU_{\Gamma(\P)}(A)$, sending $\nu(a_1,\dots,a_k;1)\in \UU_{\P}(A)$ to $\beta_{\nu,(1,\dots,1)}(a_1,\dots,a_k,-)\in \UU_{\Gamma(\P)}(A)$.
\end{prop}

Note that here, in the expression $\UU_{\P}(A)$, we have omitted the forgetful functor $\Alg{\Gamma(\P)}\to \Alg{\P}$ which is induced by the norm map as well. 

\section{Derivations, Module of Kähler differentials}\label{sec:derivations}

One of the key notions for André--Quillen cohomology of~$A$ with coefficients in a Beck module~$M$ is the group of (Beck) derivations $\Der(A,M)$, as in Definition~\ref{def:Derivation}. See \cite{Quillen67}*{\S 2.5}, \cite{Frankland10qui}, or \cite{Beck67} for more background. 
In the case of algebras over an operad $\P$, Beck derivations correspond %
to a natural notion of algebraic derivations; see \cite{GoerssH00}*{Proposition 2.2} and \cite{LodayV12}*{\S 12.3.7}. %
The set of derivations is then represented by a module $\Omega_{\P}(A)$ called the module of %
Kähler differentials 
of~$A$; see \cite{LodayV12}*{\S 12.3.8}. 
In this section, we identify the set of Beck derivations over a $\Gamma(\P)$-algebra~$A$, as well as %
the $A$-module of Kähler differentials $\Omega_{\Gamma(\P)}(A)$ in the divided power setting.

        \begin{defi}
           Let $M$ be an $A$-module, and $\diag{B\ar[r]^{pr}&A}$ be a $\Gamma(\P)$-algebra over $A$. A \Def{Beck derivation}, or simply \Def{derivation} from $B$ to $M$ is a linear map $d \colon B\to M$ such that $pr+d \colon B\to A\ltimes M$ is a morphism of $\Gamma(\P)$-algebras. We denote by $\Der_A(B,M)$ the vector space of derivations from $B$ to $M$. We obtain a bifunctor
           \[
               \Der_A \colon \left(\Alg{\Gamma(\P)}/A\right)^{\mathrm{op}}\times\lMod{A}\to \Vect{\F}.
           \]
        \end{defi}
        \begin{prop}\label{prop:derivation}
        A Beck derivation is a linear map $d \colon B\to M$ such that:
%        \begin{multline*}
	\[
            d(\beta_{x,\r}\left(b_1,\dots,b_s\right)) = 
            \sum_{j=1}^s\sum_{l+l'=r_j,l'>0}\beta_{x,\r\circ_j(l,l')}(pr(b_1),\dots,pr(b_{j}),d(b_j),pr(b_{j+1}),\dots,pr(b_s)).
	\]
%        \end{multline*}
    \end{prop}
    \begin{proof}
        The map $pr+d$ is a $\Gamma(\P)$-algebra if and only if:
        \begin{align*}
            &(pr+d)(\beta_{x,\r}\left(b_1,\dots,b_s\right)) 
            = \beta_{x,\r}\left((pr(b_1),d(b_1)),\dots,(pr(b_s),d(b_s))\right) \\
            &= \Bigg(\beta_{x,\r}\left(pr(b_1),\dots,pr(b_s)\right), 
            \sum_{j=1}^s\sum_{l,l'}\beta_{x,\r\circ_j(l,l')}(pr(b_1),\dots,pr(b_{j}),d(b_j),pr(b_{j+1}),\dots,pr(b_s))\Bigg),
        \end{align*}
        where $l,l'$ runs over the pairs of non-negative integers such that $l+l'=r_j$ and $l'>0$, hence the result.
    \end{proof}
        \begin{defi}\label{def:Kahler}
            For any $\Gamma(\P)$-algebra $A$, denote by $dA$ the underlying vector space of $A$. Elements of $dA$ are denoted by $da$ for $a\in A$.

            The \Def{module of Kähler differentials} of $A$ is the following coequalizer in the category of $A$-modules:
            \[
                 \diag{\UU_{\Gamma(\P)}(A)\otimes \Gamma(\P,A) \ar@<3pt>[r]^{\gamma}\ar@<-3pt>[r]_{\UU_{\Gamma(\P)}(A)\otimes ev_A}&\UU_{\Gamma(\P)}(A)\otimes dA \ar@{->>}[r]&\Omega_{\Gamma(\P)}(A)},
            \] 
            where $\gamma \colon \UU_{\Gamma(\P)}(A)\otimes \Gamma(\P,A)\to \UU_{\Gamma(\P)}(A)\otimes dA$ is given by:
            \rescale{\begin{align*}
                              &\beta_{x,\r}(a_1,\dots,a_{s-1},-)\otimes \lambda \beta_{y,\q}(b_1,\dots,b_{u})\mapsto \\
                              &\lambda^{r_s}\sum_{j=1}^u\sum_{l+l'=q_j,l'>0}\mu\left(\beta_{x,\r}(a_1,\dots,a_{s-1},-)\otimes \beta_{y,\q\circ_j(l,l')}(b_1,\dots,b_{j},-,b_{j+1},\dots,b_{u})\right)\otimes db_j,
                        \end{align*}}
            and where $ev_A \colon \Gamma(\P,A)\to A$ is the structural $\Gamma(\P)$-algebra evaluation map of $A$.
        \end{defi}
        Let us describe $\Omega_{\Gamma(\P)}(A)$ in more detail. For an element of the type 
        \[
			\beta_{x,\r}(a_1,\dots,a_{a-1},-)\otimes da \in \UU_{\Gamma(\P)}(A)\otimes dA,
		\]
		we will denote by $\beta_{x,\r}(a_1,\dots,a_{a-1},da)$ its image in $\Omega_{\Gamma(\P)}(A)$. Following Notation~\ref{nota:perm}, we will use the notation:
		\[
			\beta_{x,\r}(a_1,\dots,a_{i-1},da,a_{i+1},\dots,a_s) \dfn \beta_{\sigma_i^*,\r^{\sigma_i}}(a_1,\dots,a_{i-1},a_{i+1},\dots,a_s,da).
		\]
        Then, by definition, $\Omega_{\Gamma(\P)}(A)$ is an $A$-module generated by the elements $\beta_{x,\r}(a_1,\dots,a_{a-1},da)$ under certain relations, including the following:
        \begin{equation}\label{eq:derivomega}
        		d\left(\beta_{x,\r}(a_1,\dots,a_s)\right) = 
        		\sum_{i=1}^s\sum_{l+l'=r_i,l'>0}\beta_{x,\r\circ_i(l,l')}(a_1,\dots,a_{j-1},a_j,da_j,a_{j+1},\dots,a_s).
        \end{equation}

\begin{defi}
        The \Def{universal derivation} of $A$ is the linear map $d \colon A\to\Omega_{\Gamma(\P)}(A)$ induced by the identity map $A\to dA$.
\end{defi}
As a consequence of the relation~\eqref{eq:derivomega}, we deduce:
        \begin{prop}
            The universal derivation $d \colon A\to\Omega_{\Gamma(\P)}(A)$ is a Beck derivation.
        \end{prop}
        The following result justifies the term ``universal'' derivation, and completes the analogy with the classical module of Kähler differentials \cite{LodayV12}*{12.3.19}:
        \begin{prop}\label{prop:Kahler}
            $\Omega_{\Gamma(\P)}(A)$ represents the functor $\lMod{A}\to \Ab$ sending $M$ to the abelian group of derivations from $A$ to $M$.
        \end{prop}
        \begin{proof}
            We have to show that for all $A$-modules $M$, there is a linear bijection, natural in $M$:
            \[
                \Hom_{\lMod{A}}(\Omega_{\Gamma(\P)}(A),M)\cong \Der_A(A,M).
            \]
            Let $f \colon \Omega_{\Gamma(\P)}(A)\to M$ be an $A$-module morphism. Consider the linear map $D \colon A\to M$ given by $D=f\circ d$, where $d \colon A\to \Omega_{\Gamma(\P)}(A)$ is the universal derivation. Since the universal derivation is a derivation, and since $f$ is an $A$-module morphism, $D$ is a derivation. The assignment $f\mapsto D$ yields a linear map
            \[
            		\phi \colon \Hom_{\lMod{A}}(\Omega_{\Gamma(\P)}(A),M) \to \Der_A(A,M),
            	\]
            	natural in $M$.

            Conversely, let $D \colon A\to M$ be a derivation. We can consider $D$ as a linear map $dA\to M$. This extends uniquely into an $A$-module morphism $\UU_{\Gamma(\P)}(A)\otimes dA\to M$. The fact that $D$ is a derivation ensures that this passes to the coequalizer into an $A$-module morphism $f \colon \Omega_{\Gamma(\P)}(A)\to M$. This assignment $D\mapsto f$ yields a linear map 
            \[
            		\psi \colon \Der_A(A,M) \to \Hom_{\lMod{A}}(\Omega_{\Gamma(\P)}(A),M).
            	\]
            	It is straightforward %
            to check that $\phi$ and $\psi$ are mutually inverse and natural in $M$.
        \end{proof}

	To conclude this section, let us link this new notion of Kähler differentials for a $\Gamma(\P)$-algebra to the usual notion of Kähler differentials for a $\P$-algebra. Recall that $U^{\Ga(\P)}_{\P} \colon \Alg{\Ga(\P)} \to \Alg{\P}$ denotes the forgetful functor from $\Gamma(\P)$-algebras to $\P$-algebras. Then, if $A$ is a $\Gamma(\P)$-algebra, $\Omega_{\P}(A) \dfn \Omega_{\P}(U^{\Ga(\P)}_{\P}(A))$ is the usual $U^{\Ga(\P)}_{\P}(A)$-module of Kähler differentials of the $\P$-algebra $U^{\Ga(\P)}_{\P}(A)$, as in \cite{LodayV12}*{12.3.8}. %
        \begin{prop}
            The map $\te\otimes dA \colon \UU_{\P}(A)\otimes dA\to \UU_{\Gamma(\P)}(A)\otimes dA$, where $\te$ is defined in Proposition~\ref{prop:MapEnvAlg}, induces a map $\te \colon \Omega_{\P}(A)\to \Omega_{\Gamma(\P)}(A)$ %
            given by
            \[
            	\nu(a_1,\dots,da_i,\dots,a_s) \mapsto \beta_{\nu,(1,\dots,1)}(a_1,\dots,da_i,\dots,a_s).
            \]
        \end{prop}
        
        \begin{proof}
        	One readily checks that the given map passes to the coequalizers.
        \end{proof}

\section{Quillen cohomology}\label{sec:Quillencohomology}

We now have all the ingredients to describe Quillen (co)homology of $\Ga(\P)$-algebras, as listed in Section~\ref{sec:Prelim}.

\begin{lemm}\label{lem:Pushforward}
Let $f \colon B \to A$ be a morphism of $\Ga(\P)$-algebras. Via the equivalence from Theorem~\ref{thm:Amodule}, the pushforward along $f$ is given by $f_!(M) = \UU_{\Gamma(\P)}(A) \otimes_{\UU_{\Gamma(\P)}(B)} M$.
\end{lemm}

\begin{proof}
The universal enveloping algebra provides a functor 
\[
\UU_{\Gamma(\P)} \colon \Alg{\Gamma(\P)} \to \Alg{\uAs}
\]
to unital rings. The ring homomorphism $\UU_{\Gamma(\P)}(f) \colon \UU_{\Gamma(\P)}(B) \to \UU_{\Gamma(\P)}(A)$ makes $\UU_{\Gamma(\P)}(A)$ into a right $\UU_{\Gamma(\P)}(B)$-module. 
Via the equivalence from Theorem~\ref{thm:BeckMod}, the pushforward adjunction
\[
	\diag{\lMod{B}\ar@<3pt>[r]^{f_!}&\lMod{A}\ar@<3pt>[l]^{f^*}}
\]
corresponds to the classical restriction/extension of scalars along $\UU_{\Gamma(\P)}(f)$.
\end{proof}

\begin{theo}\label{thm:Abelianization}
Let $A$ be a $\Ga(\P)$-algebra.
\begin{enumerate}
	\item The following two functors form an adjoint pair:
    \[
    	\xymatrix @C+2.6cm {
    	\Alg{\Gamma(\P)}/A \ar@<3pt>[r]^-{\UU_{\Gamma(\P)}(A)\otimes_{\UU_{\Gamma(\P)}(-)}\Omega_{\Gamma(\P)}(-)} & \lMod{A}. \ar@<3pt>[l]^-{A\ltimes -}
    	}
        \diag@=3.5cm{}
    \]
	\item The adjunction simplicially prolongs to a Quillen pair.
	\item Let $C_{\bullet} \ral{\sim} A$ be a simplicial resolution of $A$, i.e., a cofibrant replacement in $s(\Alg{\Gamma(\P)})$. The cotangent complex of $A$ is the simplicial $A$-module
\[
\LL_A = \UU_{\Gamma(\P)}(A) \otimes_{\UU_{\Gamma(\P)}(C_{\bullet})} \Omega_{\Gamma(\P)}(C_{\bullet}).
\]
\end{enumerate}
\end{theo}

In particular, the~$n$\textsuperscript{th} Quillen homology module of $A$ is the $A$-module $\HQ_n(A) = \pi_n(\LL_A)$. The~$n$\textsuperscript{th} Quillen cohomology group of $A$ with coefficients in an $A$-module $M$ is the abelian group 
\[
\HQ^n(A;M) = \pi^n \left( \Hom_{\lMod{A}}(\LL_A,M) \right).
\]

\begin{proof}
(1) Via the equivalences from Theorems~\ref{thm:Amodule} and \ref{thm:BeckMod}, the statement follows from combining Proposition~\ref{prop:Kahler}, Lemma~\ref{lem:Pushforward}, and Lemma~\ref{lem:PushAbel}.

(2) Via the equivalence from Theorem~\ref{thm:BeckMod}, the right adjoint $A\ltimes -$ is the forgetful functor
\[
\left( \Alg{\Gamma(\P)}/A \right)_{\ab} \to \Alg{\Gamma(\P)}/A.
\]
The claim then follows from \cite{Frankland15}*{Proposition~3.40}. 
Part~(3) follows from part~(1).
\end{proof}

\section{Example: The operad \texorpdfstring{$\Com$}{Com}}\label{sec:ExCom}

In this section, we apply our general constructions for Beck modules, universal enveloping algebra, Beck derivations and Kähler differentials to divided power algebras over the operad $\Com$ of associative, commutative (non-unital) algebras. We check that these correspond to the constructions obtained by the first author in \cites{Dokas09,Dokas23}.

We will divide this section into two subsections. In %
Section~\ref{sec:ComClassical}, 
we recall the definition of classical divided power algebras. We refer the reader to \cites{Cartan56exp7,Roby65} and \cite{Berthelot74}*{\S I} for more details. We then review the characterisation of Beck modules, universal enveloping algebra, Beck derivations and Kähler differentials for these objects obtained by the first author \cites{Dokas09,Dokas23}. 
In %
Section~\ref{sec:ComOperadic}, 
we show how these characterisations correspond to those given in this article for a general operad $\P$, once we set $\P=\Com$.

\subsection{Classical definition and state of the art}\label{sec:ComClassical}

\begin{defi}\label{def:PDring}
Let $(A,I)$ be a commutative ring together with an ideal $I\subset A$. \Def{A system of divided powers} on $I$ is a collection of maps  $\gamma_{i} \colon  I\to A$, where $i\geq 0$, such that the following identities hold:
\begin{align}
\gamma_{0}(a) &=1, \label{PDeq0}\\
\gamma_{1}(a) &=a, \;\;\; \gamma_{i}(a)\in I, \; i\geq 1,\label{PDeq1}\\
\gamma_{i}(a+b) &=\sum_{k=0}^{k=i}\gamma_{k}(a)\gamma_{i-k}(b), \; a,b\in I,\; i\geq 0,  \label{PDeq2}\\
\gamma_{i}(ab) &=a^{i}\gamma_{i}(b),\; a\in A, \; i\geq 0,\label{PDeq3}\\
\gamma_{i}(a)\gamma_{j}(a) &=\frac{(i+j)!}{i!j!}\gamma_{i+j}(a), \; a\in I,\; i,j \geq 0,\label{PDeq4}\\
\gamma_{i}(\gamma_{j}(a)) &=\frac{(ij)!}{i!(j!)^{i}}\gamma_{ij}(a), \; a\in I,\; i\geq 0,\; j\geq 1.\label{PDeq5}
\end{align}
We say that $(I,\gamma)$ is a \Def{PD ideal} of $A$. We call \Def{divided power ring} the data of a triple $(A,I,\gamma)$ where $A$ is a ring, and $(I, \gamma)$ is a PD ideal. A morphism of divided power rings 
\[
f \colon  (A,I,\gamma) \to  (B,J,\delta)
\]
is a ring homomorphism $f \colon  A\to A'$ such that $f(I)\subset J$ and such that $f(\gamma_{i}(a))=\delta_{i}(f(a))$ for all $i\geq 0$ and $a\in I$. 
\end{defi}

Note that the identities~$\eqref{PDeq1}$ and $\eqref{PDeq4}$ imply that $a^{n}=n!\gamma_{n}(a)$, where $n\in \mathbf{N}$ and $a\in I$. In particular, in a divided power ring of characteristic $0$, $\gamma_n(a)=\frac{a^n}{n!}$ for all $a\in I$. This justifies the name ``divided powers''. In prime characteristic $p$, one has $a^{p}=0$, for all $a\in I$.

\begin{nota}
Given an augmented $\F$-algebra $A$ with augmentation $\epsilon \colon A \to \F$, denote $A_+ = \ker(\epsilon \colon A \to \F)$. 
 
Given a non-unital $\F$-algebra $B$, denote by $\ag{B} = \F \op B$ the augmented algebra obtained by formally adjoining a unit.

We follow the notation $A_+ = \ker(\epsilon \colon A \to \F)$ %
used in \cite{Soublin87} and \cite{Dokas09}, though some authors use a different notation, notably $\overline{A}$ in \cite{LodayV12}*{\S 1.1}.
\end{nota}

We now %
assume that the base field $\F$ has 
characteristic $p\neq 0$. We restrict to the case of divided power rings $(A,I,\gamma)$ such that $A$ is an augmented $\F$-algebra $A = \F \oplus A_{+}$, and $I=A_+$. 
In this setting, Soublin showed that the divided power structure is entirely determined by a map $\pi$ playing the role of $\gamma_p$.

\begin{prop}[\cite{Soublin87}*{Théorème~1}]\label{prop:Soublin}
    The category of divided power augmented $\F$-algebras $(A,A_+,\gamma)$ is equivalent to the category $\pCom$ with objects the pairs $(A,\pi)$, where $A$ is an augmented algebra with augmentation ideal $A_+$, and where $\pi \colon A_+\to A_+$ is a set map satisfying:
    \begin{align}
a^{p} &= 0,\;a\in A_{+},  \label{DPpeq1}\\
\pi(a+b) &=\pi(a)+\pi(b)+\sum_{k=1}^{k=p-1}\dfrac{(-1)^k}{k}a^{k}b^{p-k},\;a,b\in A_{+}, \label{DPpeq2}\\
\pi(ab)  &=0,\; a,b\in A_{+}, \label{DPpeq3}\\
\pi(\lambda a) &=\lambda^{p}\pi(a),\;a\in A_{+}, \lambda \in \F ,\label{DPpeq4}
\end{align}
and where morphisms $\alpha \colon (A,\pi)\to (A',\pi')$ are homomorphism of augmented algebras $\alpha \colon  A\to A'$ such that $\alpha\circ \pi=\pi'\circ \alpha$.
\end{prop}

\subsubsection*{Beck modules} 
%\paragraph{Beck modules} 
Fix a divided power augmented $\F$-algebra $A$, or equivalently, an object $(A,\pi)$ of the category $\pCom$ defined above. Then, the first author obtained the following characterisation for Beck $A$-modules:

\begin{theo}[\cite{Dokas09}*{Theorem~3.3}]
The category of Beck $A$-modules is
equivalent to the category $\mathcal{M}$ whose objects are pairs
$(M,\pi)$ where $M$ is a $A$-module and $\pi \colon  M \to
M$ is a $p$-semilinear map such that $\pi(am)=0$ for all $a\in A_{+}$ and $m\in M$ and whose morphisms $(M,\pi) \to
(M',\pi')$ are $A$-module homomorphisms $\alpha \colon  M \to M'$ such that $\pi'\circ \alpha = \alpha\circ \pi$.
\end{theo}

\subsubsection*{Enveloping algebra} 
%\paragraph{Enveloping algebra} 
Let us now build the universal enveloping algebra of $A$, which is the representing object for the category of Beck $A$-modules. Let $R_{f}$ be the polynomial ring consisting of the set of polynomials $\sum_{i=0}^{i=m}\lambda_{i}f^{i}$ where $\lambda_{i} \in \F$, $f$ is an indeterminate and $f\lambda=\lambda^{p}f$.
 
We define the ring $V(A)$ as the ring whose underlying $\F$-vector space is the tensor product $V(A) \dfn A\otimes_{\F} R_{f}$ and the multiplication is given by:
\begin{align*}
(a\otimes 1)(a'\otimes 1)&=(aa'\otimes 1),\;\;a,a'\in A,\\
(1\otimes q)(1\otimes q')&=(1\otimes qq'),\;\; q,q'\in R_{f},\\
(a\otimes 1)(1\otimes f) &=(a\otimes f),\; a\in A,\\
(1\otimes f)(a\otimes 1)&=0,\; a\in A_{+},\\
(1\otimes f)(\lambda \otimes 1)&=(\lambda^{p}\otimes f),\;\lambda \in \F.
\end{align*}
Then, the first author obtained the following:
\begin{theo}[\cite{Dokas09}*{Theorem~3.4}]\label{thm:modCom}
The category of Beck $A$-modules is equivalent to the category of left $V(A)$-modules. 
\end{theo}

\subsubsection*{Beck derivations} 
%\paragraph{Beck derivations} 
Let $A'\in p$-${\rm Com}/A$ be a divided powers algebra over $A$ and $(M,\pi)$ a Beck $A$-module. The first author obtained the following:
\begin{prop}[\cite{Dokas09}*{Theorem~3.2}]
    The abelian group of Beck derivations of $A'$ into $M$ is given by \\
    \resizebox{\linewidth}{!}{\begin{minipage}{\linewidth}
	\begin{align*}
		\Der_{p}(A',(M,\pi))=\{d\in \Der(A',M)|\; d(\pi(a))=\pi(d(a))-a^{p-1}d(a),\,a\in A'_{+}\}.
	\end{align*}
\end{minipage}}

\end{prop}

\subsubsection*{Kähler differentials} 
%\paragraph{Kähler differentials} 
We then get the following characterisation for the module of Kähler differentials:
\begin{theo}[\cite{Dokas09}*{Theorem~4.1}]\label{thm:KahlerDiv}
    The module of Kähler differentials for the augmented divided power $\F$-algebra $A$ is the $V(A)$-module $\Omega_{\pCom}(B)$ with the following presentation: the generators are the symbols $da$ for $a \in A$, and the relations are
    \begin{align}
        \label{DPpK1} d(\lambda a + \mu b)&=\lambda da + \mu db,\\
        \label{DPpK2}d(ab)&=adb+bda,\\
        \label{DPpK3}d(\pi(c))&=fdc-c^{p-1}dc,
    \end{align}
    where $a,b\in A$, $\lambda,\mu\in\F$, and $c\in A_+$.
\end{theo}

\subsection{Operadic point of view}\label{sec:ComOperadic}

Let us now show how we recover these notions from Sections\ref{sec:BeckModules}, \ref{sec:Universalalgebra} and \ref{sec:derivations}. Recall that the operad $\Com$ of non-unital commutative algebras is defined as a symmetric sequence by
    \[
        \Com(n)=\begin{cases}
            T_n,&\mbox{if }n>0,\\
            \0,&\mbox{if }n=0,
        \end{cases}
    \] 
    where $T_n$ denotes the trivial representation of dimension 1 generated by an element we denote by $X_n$, and with partial compositions given by $X_n\circ_iX_k=X_{n+k-1}$. We get the following:

    \begin{theo}[\cite{Fresse00}*{Proposition~1.2.3}]
        A $\Gamma(\Com)$-algebra is a (non-unital) associative, commutative algebra $B$ equipped, for all $i>0$, with a set map $\gamma_i \colon B\to B$ satisfying the relations \eqref{PDeq1} to \eqref{PDeq5} from Definition~\ref{def:PDring}.
    \end{theo}
    Note that there is a slight abuse of notation here: relation~\eqref{PDeq2} should be replaced by: 
    \[
        \gamma_{i}(a+b)=\gamma_i(a)+\gamma_i(b)+\sum_{k=1}^{k=i-1}\gamma_{k}(a)\gamma_{i-k}(b).
    \]
    The category of $\Gamma(\Com)$-algebras is then equivalent to the category of divided power augmented $\F$-algebras. Indeed, if $B$ is a $\Gamma(\Com)$-algebra, then $\ag{B} = \F \oplus B$ is equipped with a unique structure of divided power augmented $\F$-algebra $(\ag{B},B,\delta)$ such that $\delta_i(b)=\gamma_i(b)$ for all $i>0$ and $b\in B$. Conversely, if $(A,A_+,\gamma)$ is an augmented divided power $\F$-algebra, then, the collection of maps $\gamma_i$ for $i>0$ equips $A_+$ with a structure of $\Gamma(\Com)$-algebra as above. Following Soublin's theorem, a $\Gamma(\Com)$-algebra is equivalently characterised as a (non-unital) associative, commutative algebra $B$ equipped with a map $\pi \colon B\to B$ satisfying relations \eqref{DPpeq1} to \eqref{DPpeq4} of Proposition~\ref{prop:Soublin}.

    Let us now use the characterisation of $\Gamma(\Com)$-algebra from Theorem~\ref{thm:inv}. Following \cite{Ikonicoff20}, if $B$ is a $\Gamma(\Com)$-algebra, its multiplication is given by $\beta_{X_2,(1,1)}$, and $\pi$, which represents the divided power operation $\gamma_p$, is given by $\beta_{X_p,(p)}$.

\subsubsection*{Abelian algebras, Beck modules} 
%\paragraph{Abelian algebras, Beck modules} 
Following Definition~\ref{def:abeliangammapalg}, an abelian $\Gamma(\Com)$-algebra $M$ is a vector space $M$ endowed with a trivial multiplication $mm'=0$ and a semilinear map $\pi \colon M\to M$ satisfying $\pi(\lambda m)=\lambda^p\pi(m)$. 

    Following Definition~\ref{def:Amodule}, a $B$-module $M$ becomes a module $M$ over the commutative algebra $B$, endowed with a trivial multiplication, and a semilinear map $\pi \colon M\to M$ satisfying $\pi(\lambda m)=\lambda^p\pi(m)$ and such that $\pi(bm)=0$. This last property comes from relation~\ref{relAMcomp}, noticing that $\vert\∑_{(p)\diamond(1,1)}/\∑_p\wr(\∑_{(1,1)})\vert=p!$, and using \ref{relAMlin}:
    \begin{align*}
        \pi(bm)&=\beta_{X_p,(p)}(\beta_{X_2,(1,1)}(b,m)),\\
        &=\beta_{p!X_{2p},(p,p)}(b,m),\\
        &=p!\beta_{X_{2p},(p,p)}(b,m)=0.
    \end{align*}

\subsubsection*{Enveloping algebra} 
%\paragraph{Enveloping algebra} 
From Definition~\ref{def:UnivEnvAlg} the universal enveloping algebra $\UU_{\Gamma(\Com)}(B)$ is spanned by symbols
\[
\beta_{X_n,\r}(b_1,\dots,b_{s-1},-) \quad \text{with} \quad r_1+\dots+r_s=n.
\]
We can reduce this generating family: using the relations of Definition~\ref{def:UnivEnvAlg}, and the same reasoning as in \cite{Ikonicoff20}*{\S 3.3} and \cite{Soublin87}, $\UU_{\Gamma(\Com)}(B)$ is spanned by symbols $\beta_{X_1,(1)}(-)$ (the unit), $\beta_{X_2,(2)}(b,-)$ for $b\in B$, and $\beta_{X_p,(p)}(-)$.

We can then show that there is an isomorphism between $\UU_{\Gamma(\Com)}(B)$ and $V(\ag{B})$, where $V(B)$ was defined in Section \ref{sec:ComClassical}. This isomorphism sends $\beta_{X_1,(1)}(-)$ to $1\otimes 1$, $\beta_{X_2,(1,1)}(b,-)\in \UU_{\Gamma(\Com)}(B)$ to $b\otimes 1$, and $\beta_{X_p,(p)}(-)\in \UU_{\Gamma(\Com)}(B)$ to $1\otimes f$.

Note that $B$ injects into $\UU_{\Gamma(\Com)}(B)$, by identifying each $b\in B$ to $\beta_{X_2,(1,1)}(b,-)$. For a $B$-module (i.e., a $\UU_{\Gamma(\Com)}(B)$-module) $M$, we will denote by $bm:=\beta_{X_2,(1,1)}(b,m)=\beta_{X_2,(1,1)}(b,-)\cdot m$, and $\pi(m):=\beta_{X_p,(p)}(m)$, for $b\in B$, $m\in M$.

\subsubsection*{Beck derivations} 
%\paragraph{Beck derivations} 
By Proposition~\ref{prop:derivation}, a Beck derivation $d \colon B\to M$ is a linear map $d \colon B\to M$ satisfying $d(ab) = a db + b da$ and
\[
    d(\pi(b))=\sum_{i=1}^{p}\beta_{X_p,(p-i,i)}(b,d(b)).
\]
Note that, for all $i$ such that $2\le i\le p-1$,
\[
    X_p=\frac{1}{i!}i!X_p=\frac{1}{i!}\sum_{\tau\in\∑_{p-i}\times \∑_i/\∑_{p-i}\times \∑_1^{\times i}}\tau X_p,
\] 
and so,
\begin{align*}
    \beta_{X_p,(p-i,i)}(b,db)&=\beta_{\frac{1}{i!}\sum_{\tau\in\∑_{p-i}\times \∑_i/\∑_{p-i}\times \∑_1^{\times i}}\tau X_p,(p-i,i)}(b,db)\\
    &=\frac{1}{i!}\beta_{\sum_{\tau\in\∑_{p-i}\times \∑_i/\∑_{p-i}\times \∑_1^{\times i}}\tau X_p,(p-i,i)}(b,db),
\end{align*}
and according to relation~\ref{relMrepet}, $\beta_{\sum_{\tau\in\∑_{p-i}\times \∑_i/\∑_{p-i}\times \∑_1^{\times i}}\tau X_p,(p-i,i)}(b,db)=0$. So,
\begin{align*}
     d(\pi(b))&=\beta_{X_p,(p-1,1)}(b,db)+\beta_{X_p,(0,p)}(b,db),\\
     &=\beta_{X_2,(1,1)}(\beta_{X_{p-1},(p-1)}(b),db)+\beta_{X_p,(p)}(db),\\
     &=\frac{b^{p-1}}{(p-1)!}db+\pi(db).
\end{align*}
%Finally, 
Since in characteristic $p$, $(p-1)!=-1$ (see Wilson's Theorem \cite{DummitF04}*{\S 13.5 Exercise~6}), one has:
    \[
         d(\pi(b))=\pi(d(b))-b^{p-1}d(b).
    \]
We recover the characterisation of Beck derivations given in \cite{Dokas09}.

\subsubsection*{Kähler Differentials} 
%\paragraph{Kähler Differentials} 
    By Definition~\ref{def:Kahler}, the module of Kähler differentials $\Omega_{\Gamma(\Com)}(B)$ of $B$ is the $\UU_{\Gamma(\Com)}(B)$-module generated by elements $db$ for $b\in B$, linear in $b$, under the relations:
    \begin{align*}
        d(ab) &= d(\beta_{X_2,(1,1)}(a,b)) \\
        &= \beta_{X_2,(1,1)}(-,b)\otimes da + \beta_{X_2,(1,1)}(a,-)\otimes db \\
        &= bda+adb,\\
        d(\pi(b)) &= d(\beta_{X_p,(p)}(b)) = \sum_{i=1}^p\beta_{X_p,(p-i,i)}(b,-)\otimes db.
    \end{align*}
    For this last relation, note that the term corresponding to $i=1$ is:
    \begin{align*}
        \beta_{X_p,(p-1,1)}(b,-)\otimes b &= \beta_{X_2,(1,1)}(\beta_{X_{p-1},(p-1)}(b),-)\otimes db \\
        &= \frac{1}{(p-1)!}b^{p-1}db =-b^{p-1}db.
    \end{align*}
    The term corresponding to $i=p$ is 
    \[
         \beta_{X_p,(0,p)}(b,-)\otimes b=\beta_{X_p,(p)}(-)\otimes db=\pi(db).
    \]
    For all $i\in\{2,\dots, p-1\}$, we have again 
    \[
    		X_p = i\frac{1}{i}X_p = \sum_{\sigma\in\∑_{p-i}\times \∑_i/\∑_{p-i}\times \∑_1\times\∑_{i-1}}\frac{1}{i}X_p.
    	\]
    	So, $\beta_{X_p,(p-i,i)}(b,-)=\sum_{\sigma\in\∑_{p-i}\times \∑_i/\∑_{p-i}\times \∑_1\times\∑_{i-1}}\frac{1}{i}X_p$ is one of the elements described in Definition~\ref{def:UnivEnvAlg} in point~\eqref{item:InvariantSum} %(7), and so, is equal to 0 in $\UU_{\Gamma(\Com)}(B)$. Finally, this last relation reads:
    \[
         d(\pi(b))=\pi(db)-b^{p-1}db.
    \]
    We can show that this module of Kähler differentials $\Omega_{\Gamma(\Com)}(B)$ is isomorphic to the module of Kähler differentials $\Omega_{\pCom}(\ag{B})$ of the augmented divided power $\F$-algebra $\ag{B}$. This isomorphism is transparent on elements $db, a db$, and sends $\pi(db)$ to $f\cdot db$.

\section{Example: The operad \texorpdfstring{$\Lie$}{Lie}}\label{sec:ExLie}

In this section, %
we specialize our general constructions to  
the operad $\Lie$ of Lie algebras. We check that %
the resulting constructions 
correspond to %
those 
obtained by the first author in \cite{Dokas04}. 

As in Section~\ref{sec:ExCom}, we divide the section into two subsections. In %
Section~\ref{sec:LieClassical}, 
we recall the definition of restricted Lie algebras. We refer the reader to \cite{Jacobson62} for more details. We then review the characterisation of Beck modules and Beck derivations for these objects obtained by the first author \cite{Dokas04}. 
In %
Section~\ref{sec:LieOperadic}, 
we show how these characterisations correspond to those given in this article for a general operad $\P$, once we set $\P=\Lie$.

\subsection{Classical definition and state of the art}\label{sec:LieClassical}

We assume that the base field $\F$ has characteristic $p\neq 0$.

\begin{defi}[\cite{Jacobson62}*{\S V.7}]\label{def:Jac}
A \Def{restricted Lie algebra} $L=(L,(-)^{[p]})$ over $\F$ is a Lie
algebra over $\F$ together with a map $(-)^{[p]} \colon L \to L$ called the $p$-map such that the following relations
hold

\begin{align}
\label{RLeq1}(\alpha l)^{[p]} &=  \alpha^{p}\;l^{[p]}\\
\label{RLeq2}[l,{l'}^{[p]}]     &= [\cdots[l,\underbrace{{l'}],{l'}],\cdots,{l'}}_{p}] \\{}
\label{RLeq3}(l+{l'})^{[p]} &= l^{[p]}+{l'}^{[p]}+\sum_{i=1}^{p-1}s_{i}(l,{l'})
\end{align}
where $i s_{i}(l,{l'})$ is the coefficient of $\lambda^{i-1}$ in
$ad^{p-1}_{\lambda l+{l'}}(l)$. Here, $ad_{l} \colon  L\to L$ denotes the
adjoint representation given by $ad_{l}({l'}):=[{l'},l]$, $l,{l'} \in L,\;
\alpha \in \F$. A Lie algebra homomorphism $f \colon  L\to L'$ is
called \Def{restricted} if $f(l^{[p]})=f(l)^{[p]}$. We denote by
$\RLie$ the category of restricted Lie algebras over $\F$.
\end{defi}

\begin{ex}
Let $A$ be any associative algebra over a field $\F$. We denote by $A_{\Lie}$ the induced Lie
algebra with the bracket given by $[l,{l'}]:=l{l'}-{l'}l$, for all $l,{l'}\in
A$. Then $(A_{\Lie}, (-)^{[p]})$ is a restricted Lie algebra where $(-)^{[p]}$
is the $p$-th power $l\mapsto l^{p}$. Thus, there is a
functor 
\[
(-)_{\RLie} \colon  \As \to \RLie
\]
from the
category of associative algebras to the category of restricted Lie
algebras.
\end{ex}

\begin{ex}
Let $\rm{G}$ be an algebraic group over $\F$. The associated Lie
algebra $Lie(\rm{G})$ of $\rm{G}$ is endowed with the structure of
restricted Lie algebra \cite{Borel91}*{\S I.3} \cite{Waterhouse79}*{\S 12.1}.
\end{ex}

\subsubsection*{Beck modules} 
%\paragraph{Beck modules} 
Let $L$ be a Lie algebra over $\F$. A \Def{Lie module} over $L$ is a $\F$-vector space $M$ equipped with a $\F$-bilinear map $L\otimes_{\F}M\to M$:  $l \ot m \mapsto lm$ such that
\[
[l,l']m = l(l'm)-l'(lm),\;\text{for all}\;l,l'\in L \; \text{and}\; m\in M.
\]

\begin{defi} 
Let $L$ be a restricted Lie algebra over $\F$. A Lie $L$-module $M$ is
called \Def{restricted} if $l^{[p]} m=(\underbrace{l(\cdots (l(l}_{p} m)\cdots)$. 

Let $M$ be a restricted $L$-module. We denote by $M^{L}$, the following $L$-submodule of $M$:
\[
    M^{L} = \{ m\in M : lm=0 \text{ for all } l\in L \}.
\]
\end{defi}

The first author obtained the following characterisation for Beck $L$-modules.

\begin{theo}[\cite{Dokas04}*{Theorem~1.7}]
The category of Beck $L$-modules is
equivalent to the category $\mathcal{M}$ whose objects are pairs
$(M,f)$ where $M$ is a restricted $L$-module and $f \colon M \to M^{L}$ is a $p$-semilinear map from $M$ into its submodule of
invariants $M^{L}$ and whose morphisms $(M_{1},f_{1}) \to (M_{2},f_{2})$ are $L$ homomorphisms $\alpha  \colon  M_{1} \to M_{2}$ such that $f_{2}\circ \alpha = \alpha\circ f_{1}$.
\end{theo}

When no confusion arises, we will always denote by $f \colon M\to M$ the $p$-semilinear map of a Beck $L$-module $M$.

\subsubsection*{Enveloping algebra} 
%\paragraph{Enveloping algebra} 
Let us now build the universal enveloping algebra of $L$, which is the representing object for the category of Beck $L$-modules. Let $L\in \RLie$ be a restricted Lie algebra and $U(L)$ its usual
enveloping algebra \cite{Jacobson62}*{\S V.1}. We first recall the construction of the restricted enveloping algebra $u(L)$ which is a representing object for restricted $L$-modules.

\begin{defi}
    We denote by $u(L)$ the quotient of the algebra $U(L)$ by the relations $l^{p}-l^{[p]}$ for $l\in L$, which we call the \Def{restricted enveloping algebra} of $L$.
\end{defi} 

This construction provides a functor 
$u \colon \RLie \to \As$. 

\begin{theo}[\cite{Jacobson62}*{\S V.7}]\label{thm:modLie}
    The category of restricted $L$-modules is equivalent to the category of $u(L)$-modules.
\end{theo}

Following N.~Jacobson (see \cite{Jacobson62}*{\S V.2}), the functor $u$ is part of an adjunction
\[
u \colon \RLie \rightleftarrows \As \colon  (-)_{\RLie}
\]
Following \cite{Dokas04}, denote by $w(L)$ the $\F$-vector space $R_{f} \otimes_{\F} u(L)$, where $R_f$ is the polynomial ring on one indeterminate $f$ as in Section~\ref{sec:ExCom}. Then, $w(L)$ can be equipped with a ring structure such that
$R_{f}\to w(L)$ and $u(L)\to w(L)$ are ring homomorphisms, and such that:
\[
(f \otimes 1)(1\otimes l):=f \otimes l \quad \text{and} \quad (1 \otimes l) (f \otimes 1):=0
\]
for all $l\in L$. 
The first author obtained the following characterization of Beck modules, which yields $\UU_{\Ga(\Lie)}(L) = w(L)$.

\begin{theo}[\cite{Dokas04}*{Theorem~1.8}]\label{thm:mod}
The category of Beck $L$-modules is equivalent to the category of left $w(L)$-modules.
\end{theo}

\begin{lemm}\label{lem:MapEnvAlgLie}
The ring homomorphism $\te \colon \UU_{\Lie}(L) \to \UU_{\Ga(\Lie)}(L)$ from Proposition~\ref{prop:MapEnvAlg} is the composite
\[
U(L) \surj u(L) \to w(L)
\]
where $U(L) \surj u(L)$ is the quotient map and $u(L) \to w(L)$ is the ring homomorphism $x \mapsto 1 \otimes x$. %
\end{lemm}

\subsubsection*{Beck derivations} 
%\paragraph{Beck derivations} 
Let $M$ be a Lie $L$-module, a \Def{derivation} of $L$ into $M$ is a $\F$-linear map $D \colon  L\rightarrow M$ such that the Leibniz formula holds
\[
D([l,l']) = lD(l')-l'D(l)
\]
for all $l,l'\in L$. The set of such derivations is denoted by $\Der(L,M)$. Let $L' \in \RLie/L$ be a restricted Lie
algebra over $L$, and $(M,f)$ a Beck $L$-module. %
The first author obtained the following.

\begin{prop}[\cite{Dokas04}*{Lemma~1.4}]\label{prop:Beckdev}
    The abelian group of Beck derivations of $L'$ into $M$ is given by:
\[
\Der_{p}(L',(M,f)) \dfn \{ d\in \Der(L',M): d(l^{[p]}) = \underbrace{l\cdots l}_{p-1}dl+f(d(l)), \; l\in L' \}
\]
\end{prop}

\subsubsection*{Kähler differentials} 
%\paragraph{Kähler differentials} 
%
What we call the module of Kähler differentials of the Lie algebra $L$ is, by analogy with the case of commutative algebras, the $L$-Beck module $\Omega_{\RLie}(L)$ which represents the functor of Beck derivations. It follows from Theorem~\ref{thm:mod}, Proposition~\ref{prop:Beckdev} and \cite{Pareigis68}*{Lemma~2.1} that the module $\Omega_{\RLie}(L)$ of Kähler differentials is %
the $w(L)$-module $C(L)$ considered by Pareigis in \cite{Pareigis68}, cf.\ \cite{Dokas04}*{\S 1.3}.

\begin{theo}[\cite{Pareigis68}] %
    The module of Kähler differentials of the restricted Lie algebra $L$, $\Omega_{\RLie}(L)$, is the $w(L)$-module with the following presentation: the generators are the symbols $dl$ for $l\in L$, and the relations are:
    \begin{align}
        \label{RLpK1} d(\lambda l + \mu l')&=\lambda dl + \mu dl',\\
        \label{RLpK2} d\left([l,l']\right)&=ldl'-l'dl,\\
        \label{RLpK3} d\left(l^{[p]}\right)&=f(dl)+l^{p-1}dl.
    \end{align}
\end{theo}

\subsection{Operadic point of view}\label{sec:LieOperadic}

Let us now show how we recover these notions from Sections~\ref{sec:BeckModules}, \ref{sec:Universalalgebra} and \ref{sec:derivations}. Recall that $\Lie$ is the operad generated by a binary operation $[-,-]\in\Lie(2)$ 
satisfying $(1\,2)\cdot[-,-]=-[-,-]$ and the Jacobi relation:
    \[
        [-,-]\circ_2[-,-]+(1\ 2\ 3)\cdot \left([-,-]\circ_2[-,-]\right)+(1\ 3\ 2)\cdot \left([-,-]\circ_2[-,-]\right)=0.
    \]
As before, we assume that the base field $\F$ has characteristic $p\neq 0$.

\begin{theo}[\cite{Fresse00}*{Theorem~1.2.5}]\label{thm:GammaLieAlg}
    The category $\Alg{\Gamma(\Lie)}$ coincides with the category $\RLie$ of restricted Lie algebras.
\end{theo}

\begin{rema}
In characteristic $p=2$, an algebra over the Lie operad, also called \emph{operadic Lie algebra}, satisfies the antisymmetry equation $[x,y] = -[y,x]$ but might not satisfy the stronger condition $[x,x]=0$ \cite{Fresse00}*{Example~1.1.13}. However, Theorem~\ref{thm:GammaLieAlg} holds also in characteristic $p=2$, since a $\Gamma(\Lie)$-algebra %
satisfies the condition $[x,x]=0$ \cite{Fresse00}*{Remark~1.2.9}.
\end{rema}

We now use the characterisation of $\Gamma(\Lie)$-algebras using Theorem~\ref{thm:inv}. Following \cite{Ikonicoff23}*{Example~6.6.c}, if $L$ is a $\Gamma(\Lie)$-algebra, the Lie bracket on $L$ is given by $\beta_{[-,-],(1,1)}$, and the $p$-map is given by $\beta_{F_p,(p)}$, where $F_p\in \Lie(p)^{\∑_p}$ is the element:
\[
  \sum_{\sigma}\underbrace{[-,-]\circ_1[-,-]\circ_1\dots\circ_1[-,-]}_{p-1}\cdot \sigma,
\]
where the sum runs over the $\sigma\in\∑_p$ such that $\sigma(1)=1$. Note that relation~\eqref{RLeq2} then reads:
\begin{equation}\label{RLeq2'}
    \beta_{[-,-],(1,1)}(\beta_{1_\P,(1)},\beta_{F,(p)})=\beta_{\sum_{\sigma}\sigma [-,-]\circ_1\cdots\circ_1[-,-],(1,p)},
\end{equation}
where $\sigma$ ranges over $\∑_{1}\times \∑_p$ in the sum, which can also be written 
\[
	\beta_{[-,-],(1,1)}(\beta_{1_\P,(1)},\beta_{F,(p)}) = \beta_{F_{p+1},(1,p)}.
\]

\subsubsection*{Abelian algebras, Beck modules} 
%\paragraph{Abelian algebras, Beck modules} 
	Following Definition~\ref{def:abeliangammapalg}, an abelian $\Gamma(\Lie)$-algebra $M$ is a vector space $M$ endowed with a trivial Lie bracket $[m,m']=0$ and a semilinear map $f \colon M\to M$ satisfying $f(\lambda m)=\lambda^pf(m)$.

    By Definition~\ref{def:Amodule}, an $L$-module $M$ becomes a vector space $M$ equipped with a semilinear map $f \colon M \to M$ and an action of $L$ that we denote $[l,m]$ for $l\in L$, $m\in M$ such that $[l,f(m)]=0$. 
    This last relation comes from the following computation:
    \begin{align*}
        [l,m^{[p]}] &= \beta_{[-,-],(1,1)}(l,\beta_{F,(p)}(m))\\
        &= \beta_{\sum_{\sigma}\sigma [-,-]\circ_1\cdots\circ_1[-,-],(1,p)}(l,m).
    \end{align*}
    Here we used the relation~\eqref{RLeq2'}. Now, using relation~\ref{relMrepet} of Definition~\ref{def:Amodule}, this is equal to $0$.

\subsubsection*{Enveloping algebra} 
%\paragraph{Enveloping algebra} 
    Using the same reasoning as in \cite{Fresse00}*{Theorem~1.2.5}, and the relations of Definition~\ref{def:UnivEnvAlg}, we can show that the universal enveloping algebra $\UU_{\Gamma(\Lie)}(L)$ is generated by symbols $\beta_{[-,-],(1,1)}(l,-)$ for $l\in L$, $\beta_{F_p,(p)}(-)$, and a unit $\beta_{1_{\Lie},(1)}(-)$. We can then build an isomorphism between $\UU_{\Gamma(\Lie)}(L)$ and the universal enveloping algebra $w(L)$ defined in Section \ref{sec:LieClassical}. This isomorphism sends $\beta_{[-,-],(1,1)}(l,-)\in \UU_{\Gamma(\Lie)}(L)$ to $1\otimes l$, and $\beta_{F_p,(p)}(-)\in \UU_{\Gamma(\Lie)}(L)$ to $f\otimes 1$.

    Note that $L$ injects into $\UU_{\Gamma(\Lie)}(L)$, by identifying each element $l\in L$ to $\beta_{[-,-],(1,1)}(l,-)$. For an $L$-module (i.e., a $\UU_{\Gamma(\Lie)}(L)$-module) $M$, we will denote by $[l,m]:=\beta_{[-,-],(1,1)}(l,m)=\beta_{[-,-],(1,1)}(l,-)\cdot m$, and $f(m):=\beta_{F_p,(p)}(m)=\beta_{F_p,(p)}(-)\cdot m$ for $l\in L$, $m\in M$.

\subsubsection*{Beck derivations} 
%\paragraph{Beck derivations} 
    By Proposition~\ref{prop:derivation}, a Beck derivation $d \colon L\to M$ is a linear map $d \colon L\to M$ satisfying 
\[
	\begin{cases}
	d\left([l,l']\right) = \beta_{[-,-],(1,1)}(l,dl')+\beta_{[-,-],(1,1)}(dl,l')=[l,dl']-[l',dl] \\
	d\left(l^{[p]}\right) = \sum_{i=1}^p \beta_{F_p,(p-i,i)}(l,dl). \\
	\end{cases}
\]    
%    $d\left([l,l']\right)=\beta_{[-,-],(1,1)}(l,dl')+\beta_{[-,-],(1,1)}(dl,l')=[l,dl']-[l',dl]$, and
%\[
%    d\left(l^{[p]}\right)=\sum_{i=1}^p\beta_{F_p,(p-i,i)}(l,dl).
%\]
%Note that, 
Since $F_p\in Lie(p)^{\∑_p}$, for all $i$ such that $2\le i\le p-1$, 
\[
    F_p=\frac{1}{i!}i!F_p=\frac{1}{i!}\sum_{\tau\in\∑_{p-i}\times \∑_i/\∑_{p-i}\times \∑_1^{\times i}}\tau F_p,
\] 
and thus:
\begin{align*}
    \beta_{F_p,(p-i,i)}(l,dl) &= \beta_{\frac{1}{i!}\sum_{\tau\in\∑_{p-i}\times \∑_i/\∑_{p-i}\times \∑_1^{\times i}}\tau F_p,(p-i,i)}(l,dl)\\
    &= \frac{1}{i!}\beta_{\sum_{\tau\in\∑_{p-i}\times \∑_i/\∑_{p-i}\times \∑_1^{\times i}}\tau F_p,(p-i,i)}(l,dl),
\end{align*}
and %according to 
by relation~\ref{relMrepet}, $\beta_{\sum_{\tau\in\∑_{p-i}\times \∑_i/\∑_{p-i}\times \∑_1^{\times i}}\tau F_p,(p-i,i)}(l,dl)=0$.

\subsubsection*{Kähler differentials} 
%\paragraph{Kähler differentials} 
    By Definition~\ref{def:Kahler}, the module of Kähler differentials $\Omega_{\Gamma(\Lie)}(L)$ of $L$ is the $\UU_{\Gamma(\Lie)}(L)$-module generated by elements $dl$, $l\in L$, under the relations:
    \begin{align*}
        d(\lambda l + \mu l') &= \lambda dl + \mu dl',\\
        d([l,l']) &= d(\beta_{[-,-],(1,1)}(l,l')) \\
        &= \beta_{[-,-],(1,1)}(-,l')\otimes dl + \beta_{[-,-],(1,1)}(l,-)\otimes dl' \\
        &= [l,dl']-[l',dl],\\
        d(f(l)) &= d(\beta_{F_p,(p)}(l)) = \sum_{i=1}^p\beta_{F_p,(p-i,i)}(b,-)\otimes db.
    \end{align*}
    To compute the term corresponding to $i=1$, note that, by \cite{Fresse00}*{Remark~1.2.8}, $\beta_{F_p,(p-1,1)}(l,-)\cdot dl=\beta_{F_p,(p-1,1)}(l,dl)=[\underbrace{l,\dots[l}_{p-1},dl]\dots]$.
    the term corresponding to $i=p$ is 
    \[
         \beta_{F_p,(0,p)}(l,-)\otimes l=\beta_{F_p,(p)}(-)\otimes dl=fdl.
    \]
    For all $i\in\{2,\dots, p-1\}$, again we have
\[    
	F_p = i\frac{1}{i}F_p = \sum_{\sigma\in\∑_{p-i}\times \∑_i/\∑_{p-i}\times \∑_1\times\∑_{i-1}}\frac{1}{i}F_p.
\]
So, $\beta_{F_p,(p-i,i)}(l,-)=\beta_{\sum_{\sigma\in\∑_{p-i}\times \∑_i/\∑_{p-i}\times \∑_1\times\∑_{i-1}}\frac{1}{i}F_p,(p-i,i)}(l,-)$ is one of the elements described in Definition~\ref{def:UnivEnvAlg} in point~\eqref{item:InvariantSum}, %(7), 
and so, is equal to 0 in $\UU_{\Gamma(\Lie)}(L)$. This last relation reads:
    \[
         d(fl)=f(dl)+[\underbrace{l,\dots[l}_{p-1},dl]\dots].
    \]
    Finally, $\Omega_{\Gamma(\Lie)}(L)$ is indeed isomorphic to the module $\Omega_{\RLie}(L)$ defined in Section \ref{sec:LieClassical}.

\section{Comparisons}\label{sec:comp}

In this section, we identify some adjunctions involving $\Ga(\P)$-algebras and check that they induce comparisons on Quillen cohomology. The next sections will focus on examples.

\begin{lemm}\label{lem:ColimMonadic}
Let $\cat{C}$ be a cocomplete closed symmetric monoidal category and $\P$ an operad in $\cat{C}$.
\begin{enumerate}
\item The free $\P$-algebra monad $S(\P) \colon \cat{C} \to \cat{C}$ preserves reflexive coequalizers and filtered colimits.
\item The forgetful functor $U_{\cat{C}}^{\P} \colon \Alg{\P} \to \cat{C}$ creates reflexive coequalizers and filtered colimits.
\end{enumerate}
\end{lemm}

\begin{proof}
The first part is proved in \cite{Rezk96}*{Proposition~2.3.5}. The second part follows from the first part and \cite{Borceux94v2}*{Proposition~4.3.2}.
\end{proof}

\begin{lemm}\label{lem:PreserveRegEpi}
The forgetful functor $U^{\Ga(\P)}_{\P} \colon \Alg{\Ga(\P)} \to \Alg{\P}$ preserves and reflects regular epimorphisms.
\end{lemm}

\begin{proof}
Both categories are monadic over $\Vect{\F}$, as illustrated in the diagram of forgetful functors
\[
\xymatrix{
\Alg{\Ga(\P)} \ar[dr]_-{U^{\Ga(\P)}_{\F}} \ar[r]^-{U^{\Ga(\P)}_{\P}} & \Alg{\P} \ar[d]^{U^{\P}_{\F}} \\
& \Vect{\F}. \\
}
\]
A regular epimorphism $q \colon X \surj Y$ is the coequalizer of its kernel pair $X \x_Y X \rra X$, which is a reflexive pair, with common section the diagonal $X \to X \x_Y X$. %
By Lemma~\ref{lem:ColimMonadic}, the functor $U^{\P}_{\F}$ preserves and reflects %
reflexive coequalizers, hence also regular epimorphisms. 

In $\Vect{\F}$, all regular epimorphisms (namely the surjective maps) split, assuming the axiom of choice. Thus any functor $\Vect{\F} \to \Vect{\F}$ preserves regular epimorphisms, and the functor $U^{\Ga(\P)}_{\F}$ preserves and reflects regular epimorphisms, by \cite{Borceux94v2}*{Theorem~4.3.5}. 
\end{proof}

\begin{rema}
Working over a more general base commutative ring $k$ instead of a field $\F$, the endofunctor $\Ga(\P) \colon \Modd{k} \to \Modd{k}$ need \emph{not} preserve regular epimorphisms.

For example, take $k = \Z$ and the operad $\P$ in $\Modd{\Z} = \Ab$ generated by one binary operation $\mu \in \P(2)$ subject to the relation $\mu \cdot (12) = -\mu$. Then $\P$ is a reduced operad with $\P(2) = \Z_{\si}$, which denotes $\Z$ with $\Sy_2$-action by the sign. Consider the quotient map of abelian groups $q \colon \Z \surj \Z/2$. We compute the $\Sy_2$-fixed points
\begin{align*}
&(\P(2) \ot \Z^{\ot 2})^{\Sy_2} \cong (\Z_{\si} \ot \Z_{\triv})^{\Sy_2} = 0 \\
&(\P(2) \ot (\Z/2)^{\ot 2})^{\Sy_2} \cong ((\Z/2)_{\si})^{\Sy_2} = \Z/2.
\end{align*}
Thus the degree~$2$ summand of the map of abelian groups
\[
\Ga(\P,q) \colon \Ga(\P,\Z) \to \Ga(\P,\Z/2)
\]
is the map $0 \to \Z/2$, which is not surjective.
\end{rema}

Since $\Alg{\Ga(\P)}$ and $\Alg{\P}$ are algebraic categories, Lemma~\ref{lem:PreserveRegEpi} ensures that the adjunction
\begin{equation}\label{eq:FreeGammaP}
F_{\P}^{\Ga(\P)} \colon \Alg{\P} \rla \Alg{\Ga(\P)} \colon U_{\P}^{\Ga(\P)}
\end{equation}
gives rise to the comparison diagrams described in \cite{Frankland15}*{Theorem~4.7}. 
Another source of comparisons will be given by morphisms of operads, as we now describe.

\begin{lemm}
Let $f \colon \P \to \Q$ be a morphism between reduced operads in $\Vect{\F}$. Consider the diagram of four adjunctions
\[
\xymatrix @R+0.5pc @C+0.5pc {
\Alg{\P} \ar@<-0.6ex>[d]_{F_{\P}^{\Ga(\P)}} \ar@<0.6ex>[r]^-{f_!} & \Alg{\Q} \ar@<0.6ex>[l]^-{f^*} \ar@<-0.6ex>[d]_{F_{\Q}^{\Ga(\Q)}} \\
\Alg{\Ga(\P)} \ar@<-0.6ex>[u]_{U_{\P}^{\Ga(\P)}} \ar@<0.6ex>[r]^-{f_!} & \Alg{\Ga(\Q)}. \ar@<0.6ex>[l]^{f^*} \ar@<-0.6ex>[u]_-{U_{\Q}^{\Ga(\Q)}}
}
\]
\begin{enumerate}
\item \label{item:RightAdjoints} The right adjoints commute, and therefore the left adjoints commute (up to natural isomorphism).
\item \label{item:Restriction} Both restriction functors $f^*$ preserves and reflect regular epimorphisms.
\end{enumerate}
\end{lemm}

\begin{proof}
\eqref{item:RightAdjoints} Via the explicit description of $\Ga(\P)$-algebras given in Theorem~\ref{thm:Abelianization}, the restriction functor $f^* \colon \Alg{\Ga(\Q)} \to \Alg{\Ga(P)}$ can be described as follows. Consider a $\Ga(\Q)$-algebra $A$ with operations $\be_{y,\ul{r}} \colon A^{\x s} \to A$, given for all $\ul{r} = (r_1, \dots, r_s)$ and $y \in \Q(n)^{\Sy_{\ul{r}}}$ with $n = r_1 + \cdots + r_s$. Its restriction $f^* A$ has the same underlying $\F$-vector space $A$, with $\Ga(\P)$-algebra structure given by the operations
\[
\be_{x,\ul{r}} = \be_{f(x),\ul{r}}
\]
for all $x \in \P(n)^{\Sy_{\ul{r}}}$ and $\ul{r}$ as above. Note that the map on the arity $n$ parts $f \colon \P(n) \to \Q(n)$ is $\Sy_n$-equivariant, hence restricts to fixed point subspaces $f \colon \P(n)^{\Sy_{\ul{r}}} \to \Q(n)^{\Sy_{\ul{r}}}$.

\eqref{item:Restriction} As observed in the proof of Lemma~\ref{lem:PreserveRegEpi}, in all four categories, regular epimorphisms are preserved and reflected by the forgetful functor to $\Vect{\F}$. More concretely, they are the morphisms whose underlying map of vector spaces is surjective.
\end{proof}

Next, we want to compare the Quillen cohomology of $\Ga(\P)$-algebras and $\P$-algebras. Start with a $\Ga(\P)$-algebra $A$ and consider its underlying $\P$-algebra $U_{\P}^{\Ga(\P)} A$, also denoted $A$ when the context indicates the category. The diagram of adjunctions in \cite{Frankland15}*{\S 4.2.2} specializes to
\[
\xymatrix @R+0.5pc @C+0.5pc {
\Alg{\P}/A \ar@<-0.6ex>[d]_{\ep_{A!} F_{\P}^{\Ga(\P)}} \ar@<0.6ex>[r]^-{\Ab_A} & \left( \Alg{\P}/A \right)_{\ab} \ar@<0.6ex>[l]^-{U_A} \ar@<-0.6ex>[d]_{\ep_{A\#} \tilde{F}_{\P}^{\Ga(\P)}} \\
\Alg{\Ga(\P)}/A \ar@<-0.6ex>[u]_{U_{\P}^{\Ga(\P)}} \ar@<0.6ex>[r]^-{\Ab_A} & \left( \Alg{\Ga(\P)}/A \right)_{\ab}. \ar@<0.6ex>[l]^-{U_A} \ar@<-0.6ex>[u]_{U_{\P}^{\Ga(\P)}}
}
\]
Using the identification of Beck modules and K\"ahler differentials in \cite{LodayV12}*{Propositions~12.3.8 and 12.3.13} for the top row and Theorem~\ref{thm:BeckMod} and Proposition~\ref{prop:Kahler} for the bottom row, the diagram becomes
\begin{equation}\label{eq:AdjGammaP}
\xymatrix @R+1pc @C+6pc {
\Alg{\P}/A \ar@<-0.6ex>[d]_{\ep_{A!} F_{\P}^{\Ga(\P)}} \ar@<0.6ex>[r]^-{\UU_{\P}A \ot_{\UU_{\P}(-)} \Om_{\P}(-)} & \lMod{\UU_{\P}A} \ar@<0.6ex>[l]^-{A \ltimes -} \ar@<-0.6ex>[d]_{\te_!} \\
\Alg{\Ga(\P)}/A \ar@<-0.6ex>[u]_{U_{\P}^{\Ga(\P)}} \ar@<0.6ex>[r]^-{\UU_{\Ga(\P)}A \ot_{\UU_{\Ga(\P)}(-)} \Om_{\Ga(\P)}(-)} & \lMod{\UU_{\Ga(\P)}A}. \ar@<0.6ex>[l]^-{A \ltimes -} \ar@<-0.6ex>[u]_{\te^*}
}
\end{equation}

Here $\te \colon \UU_{\P}A \to \UU_{\Ga(\P)}A$ is the $\F$-linear ring homomorphism 
described in Proposition~\ref{prop:MapEnvAlg}, $\te^*$ denotes restriction of scalars along $\te$, and $\te_!$ denotes extension of scalars $\te_!(M) = \UU_{\Ga(\P)}A \ot_{\UU_{\P}A} M$. Applying \cite{Frankland15}*{Propositions~4.13 and 4.14} yields the following.

\begin{prop}\label{prop:CompHQ}
Let $A$ be a $\Ga(\P)$-algebra. 
\begin{enumerate}
\item There is a natural (up to homotopy) comparison of cotangent complexes
\begin{equation}\label{eq:Cotangent}
\LL^{\P}_{A} \to \te^* \LL^{\Ga(\P)}_A
\end{equation}
in simplicial $\UU_{\P}A$-modules.
\item For each degree $n \geq 0$, there is a natural comparison map of $\UU_{\P}A$-modules %
\[
\HQ^{\P}_n(A) \to \te^* \HQ^{\Ga(\P)}_n(A)
\]
from Quillen homology of $A$ as a $\P$-algebra to Quillen homology of $A$ as a $\Ga(P)$-algebra. 
\item \label{item:CompHQcohom} For $M$ a left $\UU_{\Ga(\P)}A$-module, there is a natural comparison map of abelian groups
\begin{equation}\label{eq:ComparisonHQ}
\HQ_{\Ga(\P)}^n(A;M) \to \HQ_{\P}^n(A; \te^* M)
\end{equation}
from Quillen cohomology of $A$ as a $\Ga(\P)$-algebra to Quillen cohomology of $A$ as a $\P$-algebra.
\item If the map of simplicial $\UU_{\Ga(\P)}A$-modules
\[
\te_! \LL^{\P}_{A} \to \LL^{\Ga(\P)}_A
\]
adjunct to the map \eqref{eq:Cotangent} is a weak equivalence, then the comparison in Quillen cohomology \eqref{eq:ComparisonHQ} is an isomorphism in all degrees $n$.
\end{enumerate}
\end{prop}

We can describe the comparison map of cotangent complexes \eqref{eq:Cotangent} more explicitly.

\begin{prop}\label{prop:compmaps}
\begin{enumerate}
\item For any $\Ga(\P)$-algebra $A$, there is a natural map of $\UU_{\P}A$-modules
\[
\Om_{\P}(A) \to \te^* \Om_{\Ga(\P)}(A)
\]
given by
\[
\mu(a_1, \ldots, da_i, \ldots, a_n) \mapsto \beta_{\mu(1,\ldots,1)}(a_1, \ldots, da_i, \ldots, a_n)
\]
for $\mu \in \P(n)$, $a_j \in A$. 
\item More generally, for any morphism of $\Ga(\P)$-algebras $g \colon B \to A$, there is a natural map of $\UU_{\P}A$-modules
\begin{equation}\label{eq:ComparisonKahler}
\xymatrix @R-0.5cm {
\UU_{\P}A \ot_{\UU_{\P}B} \Om_{\P}(B) \ar@{=}[d] \ar[r] & \te^* \left( \UU_{\Ga(\P)}A \ot_{\UU_{\Ga(\P)}B} \Om_{\Ga(\P)}(B) \right) \ar@{=}[d] \\
\Ab_{UA}(UB \to UA) & \te^* \Ab_{A}(B \to A) \\
}
\end{equation}
given by
\[
\mu(a_1, \ldots, db_i, \ldots, a_n) \mapsto \beta_{\mu(1,\ldots,1)}(a_1, \ldots, db_i, \ldots, a_n)
\]
for $\mu \in \P(n)$, $a_j \in A$ ($j \neq i$), and $b_i \in B$.
\item Given a cofibrant replacement $C_{\bullet} \ral{\sim} A$ in simplicial $\Ga(\P)$-algebras, the comparison maps \eqref{eq:ComparisonKahler} for $C_n \to A$ in simplicial degree $n \geq 0$ yield the comparison map of cotangent complexes $\LL^{\P}_{A} \to \te^* \LL^{\Ga(\P)}_A$ from \eqref{eq:Cotangent}.
\end{enumerate}
\end{prop}

\begin{proof}
First, we show that the underlying vector space of $\UU_{\P}A \ot_{\UU_{\P}B} \Om_{\P}(B)$ is spanned by the elements of the form $\mu(a_1, \ldots, db_i, \ldots, a_n)$. The tensor product $\UU_{\P}A \ot_{\UU_{\P}B} \Om_{\P}(B)$ is spanned by elements of the type:
\[
    \mu(a_1, \ldots,a_{i-1}, -,a_{i+1} \ldots, a_n)\ot_{\UU_{\P}B} \nu(b_1,\ldots,db_j,\ldots,b_k),
\]
With $\mu\in\P(n)$, $\nu\in \P(k)$, $a_u\in A$ and $b_v\in B$.
However, noting that this element is equal to:
\[
    \mu(a_1, \ldots,a_{i-1}, -,a_{i+1} \ldots, a_n)\ot_{\UU_{\P}B} \nu(b_1,\ldots,b_{j-1},-,b_{j+1},\ldots,b_k)\cdot db_j,
\]
tensoring over $\UU_{\P}B$ means that this element is identified with:
\rescale{\begin{align*}
	    \mu \left( a_1, \ldots,a_{i-1}, -,a_{i+1} \ldots, a_n \right) \cdot \nu \left( g(b_1),\ldots,g(b_{j-1}),-,g(b_{j+1}),\ldots,g(b_k) \right)\otimes_{\UU_{\P}B} db_j
\end{align*}}
which is equal to:
\rescale{\begin{align*}
    \mu\circ_i\nu(a_1, \ldots,a_{i-1},g(b_1),\ldots,g(b_{j-1}),-,g(b_{j+1}),\ldots,g(b_k),a_{i+1} \ldots, a_n)\otimes_{\UU_{\P}B} db_j.
\end{align*}}
We now see that $\UU_{\P}A \ot_{\UU_{\P}B} \Om_{\P}(B)$ is spanned by elements of the form:
 \[
     \mu(a_1, \ldots,a_{i-1}, -,a_{i+1} \ldots, a_n)\ot_{\UU_{\P}B} db_i,
 \]

Similarly, one shows that $\UU_{\Ga(\P)}A \ot_{\UU_{\Ga(\P)}B} \Om_{\Ga(\P)}(B)$ is spanned by elements of the type:
\[
    \beta_{\mu,\r}(a_1,\dots,a_{i-1},-,a_{i+1},\dots,a_s)\ot_{\UU_{\Ga(\P)}B}db_i,
\]
which we denote by $\beta_{\mu,\r}(a_1,\dots,db_i,\dots,a_s)$. The expression of the comparison map is then induced by that of the norm map $\theta$ from Proposition~\ref{prop:MapEnvAlg}.
\end{proof}
Here the left adjoint $\te_! \colon \lMod{\UU_{\P}A} \to \lMod{\UU_{\Ga(\P)}A}$ induced on Beck modules is described in terms of the ring homomorphism $\te$, but one might hope to describe it in terms of the original left adjoint $F^{\Ga(\P)}_{\P} \colon \Alg{\P} \to \Alg{\Ga(\P)}$. However, the left adjoint $F^{\Ga(\P)}_{\P}$ does \emph{not always} pass to Beck modules in the sense of \cite{Frankland15}*{Definition~3.29}. For example, %
in the case of divided power algebras, %
$F^{\Ga(\Com)}_{\Com}$ does not pass to Beck modules (see Proposition~\ref{prop:penvelopenotpass}), while in the case of restricted Lie algebras, $F^{\Ga(\Lie)}_{\Lie}$ passes to Beck modules (see Proposition~\ref{prop:Liepenvelopepasses}).

\section{Divided power algebras versus commutative algebras}\label{sec:compdiv}

In this section, we take the operad $\P = \Com$ and analyze the effect of the free-forget adjunction $\Alg{\Com} \rla \Alg{\Ga(\Com)}$. %

We denote by $\Com_{\aug}$ the category of augmented $\F$-algebras and use the equivalence $\Alg{\Com} \cong \Com_{\aug}$, as well as the equivalence $\Alg{\Ga(\Com)} \cong \pCom$ from Proposition~\ref{prop:Soublin}. P.~Berthelot defined the notion of PD envelope of an ideal. In particular, it follows from \cite{Berthelot74}*{\S 2.3} that the forgetful functor 
\[
U^{\Ga(\Com)}_{\Com}  \colon\pCom \to \mathrm{Com_{aug}}
\]
admits a left adjoint functor 
\[
F^{\Ga(\Com)}_{\Com} \colon \mathrm{Com_{\aug}} \to \pCom
\]
given by $F^{\Ga(\Com)}_{\Com}(A)=\hat{A}$, where $\hat{A}$ denotes the PD envelope of the augmentation ideal $A_{+}$ of $A$. We denote by $\eta$ the unit of this adjunction, which is a homomorphism of augmented algebras $\eta_{A} \colon A \to \hat{A}$. 

Let $A$ be a divided power algebra, equivalently, an object of $\pCom$. Since Beck modules over an augmented algebra $B$ are just $B$-modules, $A$ is the object $\UU_{\Com}(A)$ representing Beck modules over the underlying augmented algebra of $A$. Recall from Section~\ref{sec:ExCom} that $V(A) = A \ot_{\F} R_f$ represents Beck modules over $A$. The $\F$-linear ring homomorphism 
$\te \colon A \to V(A)$ is given by $\te(a) = a \otimes 1$. Denote by $\Om^1_A$ the usual Kähler differentials over the underlying augmented algebra, whereas $\Om_{\pCom}(A)$ was described in Theorem~\ref{thm:KahlerDiv}. The diagram of adjunctions \eqref{eq:AdjGammaP} specializes to
\[
\xymatrix @R+1pc @C+5pc {
\Com_{\aug}/A \ar@<-0.6ex>[d]_{\ep_{A!} F^{\Ga(\Com)}_{\Com}} \ar@<0.6ex>[r]^-{A \ot_{(-)} \Om^1_{(-)}} & \lMod{A} \ar@<0.6ex>[l]^-{A \ltimes -} \ar@<-0.6ex>[d]_{\te_!} \\
\pCom/A \ar@<-0.6ex>[u]_{U^{\Ga(\Com)}_{\Com}} \ar@<0.6ex>[r]^-{V(A) \ot_{V(-)} \Om_{\pCom}(-)} & \lMod{V(A)}. \ar@<0.6ex>[l]^-{A \ltimes -} \ar@<-0.6ex>[u]_{\te^*} \\
}
\]
Let %
By Theorem~\ref{thm:modCom}, %
a $V(A)$-module corresponds 
to a pair $(M,\pi)$, where %
$M$ is an 
$A$-module and $\pi \colon M \to M$ is a $p$-semilinear map such that $\pi(am)=0$ holds for all $a \in A_{+}$ and $m\in M$. Equivalently, the $V(A)$-module %
$(M,\pi)$ corresponds 
to an abelian group object
$A \oplus_{p}{} M \to A$ in ($\pCom/A)_{\ab}$, where $A \oplus_{p}M$ is the semidirect product in $\Com$ of $A$ and $M$ together with the map
\[
\pi(a,m) = (\pi(a),\pi(m) - a^{p-1}m), \: a \in A_{+}, \, m \in M.
\]
Via Theorem~\ref{thm:modCom}, restriction of scalars along $\te \colon A \to V(A)$ is the functor sending a pair $(M,\pi)$ to $M$, forgetting the $p$-semilinear map $\pi$. 

\begin{lemm}\label{lem:coefCom}
Let $A \in \Com_{aug}$ be an augmented algebra, and $M$ be a $V(\hat{A})$-module. Then the functor
\[
\eta^{*}_{A} U^{\Ga(\Com)}_{\Com} \colon (\pCom/\hat{A})_{\ab}\to (\Com_{\aug}/A)_{\ab}
\]
is given by
\[
\eta_{A}^{*}U^{\Ga(\Com)}_{\Com} (M)={}_{A}M
\]
where ${}_{A}M$ denotes $M$ with the $A$-module structure induced by restriction of scalars along the morphism $\eta_{A} \colon A \to \hat{A}$.
\end{lemm}

\begin{proof}
In the composite of functors
\[
\xymatrix @R-0.3pc {
(\pCom/\hat{A})_{\ab} \ar[d]^{\cong}_{\blue{\text{Thm. } \ref{thm:modCom}}} \ar[r]^-{U^{\Ga(\Com)}_{\Com}} & (\Com_{\aug}/U\hat{A})_{\ab} \ar[d]^{\cong} \ar[r]^-{\eta_{A}^{*}} & (\Com_{\aug}/A)_{\ab} \ar[d]^{\cong} \\
\lMod{V(\hat{A})} \ar[r]^-{\te^*} & \lMod{\hat{A}} \ar[r]^-{\eta_{A}^{*}} & \lMod{A} \\
}
\]
the first step is restriction of scalars along the ring homomorphism $\te \colon \hat{A} \to V(\hat{A})$. The second step is restriction of scalars along the ring homomorphism $\eta_A \colon A \to \hat{A}$, since this is how pullbacks of Beck modules are computed in commutative algebras.
\end{proof}

\begin{prop}
\begin{enumerate}
\item Let $A \in \Com_{\aug}$ an augmented commutative algebra and $M$ be a $V(\hat{A})$-module. Then there is a comparison map
\[
\HQ_{\pCom}^{*}(\hat{A}; M) \to \HQ_{\Com}^{*}(A; {}_{A}M).
\]
\item Let $B \in \pCom$ be a divided power algebra and $M$ be a $V(B)$-module. Then there is a comparison map
\[
\HQ_{\pCom}^{*}(B; M) \to \HQ_{\Com}^{*}(B; \te^* M).
\]
\end{enumerate}
\end{prop}

\begin{proof}
Part~(1) follows directly from \cite{Frankland15}*{Proposition~4.12} and Lemma~\ref{lem:coefCom}. Part~(2) is a specialization of Proposition~\ref{prop:CompHQ}~\eqref{item:CompHQcohom} to the operad $\P = \Com$.
\end{proof}

We now show that the $p$-envelope functor does not pass to Beck modules:

\begin{prop}\label{prop:penvelopenotpass}
In the case $\F = \F_2$, the functor $F^{\Ga(\Com)}_{\Com}$ that freely adjoins divided power operations does \emph{not} pass to Beck modules.
\end{prop}

\begin{proof}
Take the commutative (unital) $\F_2$-algebra $A=\F_2$ and the $A$-module $M=\F_2$ whose generator (only non-zero element) we denote by $x$. Viewing the $A$-module as a square-zero split extension $pr \colon A \op M \surj A$, apply the functor $F^{\Ga(\Com)}_{\Com}$ to obtain the split epimorphism of divided power algebras
\[
\hat{pr} \colon \widehat{A \op M} \surj \hat{A}.
\]
Then, $\widehat{A\oplus M}$ is the free divided power algebra on one generator $\Gamma(x)$, that is, its underlying vector space is isomorphic to the vector space of polynomials $\F_2[x]$, but the multiplication is induced by $x^n * x^m = \binom{m+n}{n}x^{n+m}$. The divided power algebra $\hat A$ is still equal to $\F_2$. %
The kernel $K = \ker (\hat{pr})$ is equal to the subalgebra of $\Gamma(x)$ of non-constant polynomials. In $K$, we have for example, $x*x^2=3x^3=x^3$. So $K$ is not a square-zero algebra, thus the split epimorphism $\hat{pr}$ does not yield a Beck module.
\end{proof}
To conclude this section, we will specialize the comparison maps of Proposition~\ref{prop:compmaps} to the case of divided power algebras. Let $g \colon B\to A$ be a morphism of non-unital divided power algebras. On the one hand, 
\[
	\UU_{\Com} A\otimes_{\UU_{\Com} B}\Omega_{\Com}(B) = \ag{A }\otimes_{\ag{B}} \Omega^1_{\ag{B}}
\]
is spanned by the elements $a db$ for $a \in A$, $b \in B$, under relations expressing the fact that $d$ is a (linear) $\ag{A}$-derivation, that is:
\[
	\begin{cases}
		d(\lambda b+b') = \lambda db+db' \\
		d(bb') = g(b) db' + g(b') db \\
	\end{cases}
\]
and the action of $\ag{A}$ is given by the multiplication in $A$ ($a\cdot a'db = (aa')db$). On the other hand, 
\[
	\UU_{\Gamma(\Com)} A \otimes_{\UU_{\Gamma(\Com)}B} \Omega_{\Gamma(\Com)}(B) = V(\ag{A}) \otimes_{V(\ag{B})} \Omega_{\pCom}(B)
\]
is spanned by elements $af^{k}db$ for $a\in A$, $b\in B$ and $k\in \N$, under relation expressing the fact that $d$ is a Beck $A$-derivation in $\pCom$, that is, we also have
\[
	d\gamma_p(b) = fdb-g(b)^{p-1}db.
\]	
The action of $V(\ag{A})$ is again given by the multiplication in $V(\ag{A})$, in particular, $(a\otimes f^k)\cdot a'f^{k'}db = aa'^{p^k}f^{k+k'}db$. The comparison map
\[
    \UU_\Com A\otimes_{\UU_\Com B}\Omega_{\Com}(B)\to \UU_{\Gamma(\Com)} A\otimes_{\UU_{\Gamma(\Com)}B}\Omega_{\Gamma(\Com)}(B)
\]
from Proposition~\ref{prop:compmaps} is %
given 
by $adb\mapsto af^0db$.

\section{Restricted Lie algebras versus Lie algebras}\label{sec:complie}

Let $\F$ be a field of %
characteristic $p \neq 0$. Let $(H,(-)^{[p]})$ be a restricted Lie algebra. Then by Theorem~\ref{thm:mod} a $w(H)$-module $M$ is associated to a pair $({}_{u(H)}M,f)$, where ${}_{u(H)}M$ is $M$ viewed as a $u(H)$-module and $f\colon {}_{u(H)}M\to {}_{u(H)}M^{H}$ is a $p$-semilinear map. Equivalently, the $w(H)$-module $M$ is associated to the abelian group object 
\[
H\ltimes_{f} {}_{u(H)}M\to H
\]
in $(\RLie/L)_{\ab}$, where $H\ltimes_{f}{}_{u(H)}M$ denotes the semidirect product in $\RLie$ of $H$ by ${}_{u(H)}M$. In particular, 
\[
H\ltimes_{f}{}_{u(H)}M = \{(h,m),\,h \in H,\, m\in M\}.
\]
The Lie bracket is given by
\[
[(h,m),(h',m')]=([h,h'],hm'-h'm)
\]
and the $p$-map is given by
\[
(h,m)^{[p]}=(h^{[p]},\underbrace{h\cdots h}_{p-1}m+f(m)).
\]
The notion of $p$-envelope of a Lie algebra has been studied in detail in \cite{StradeF88}*{\S 2.5}. 

\begin{defi}
Let $L$ be a Lie algebra over $\F$ and $U(L)$ its enveloping algebra. The \Def{$p$-envelope} $\hat{L}$ of %
$L$ is the Lie subalgebra of $U(L)_{\Lie}$ which contains $L$ and all iterated associative 
$p$-th powers. Thus, $\hat{L}$ is a restricted Lie subalgebra of $U(L)_{\RLie}$. 
\end{defi}

Note that $L$ is an ideal in $\hat{L}$. 
In \cite{Milner75}, A. A. Mil'ner proves that $\hat{L}$ has the following universal property: for all restricted Lie algebras $A\in \RLie$ and all Lie algebra homomorphisms $f \colon  L\rightarrow A$, there is exactly one restricted Lie algebra homomorphism 
$\hat{f} \colon  \hat{L}\rightarrow A$ such that $\hat{f}\circ i=f$; see \cite{StradeF88}*{\S 2.5, Theorem~2.5.2}. We then deduce that $\hat{L}\cong F^{\Ga(\Lie)}_{\Lie}(L)$, where $L$ is a Lie algebra. It also follows from the universal properties of $U(L)$ and $u(\hat L)$ that they are isomorphic.

Denote by $\eta$ the unit of the adjunction $F^{\Ga(\Lie)}_{\Lie} \dashv U^{\Ga(\Lie)}_{\Lie}$ and $\iota$ the unit of the adjunction $u \dashv (-)_{\RLie}$. From the foregoing discussion, for $L$ a Lie algebra, we get a Lie algebra homomorphism $\eta_{L}\colon L\to \hat{L}$, and a restricted Lie homomorphism $\iota_{\hat L} \colon \hat{L} \to u(\hat{L}) \cong U(L)$.

\begin{lemm}\label{lem:CoefLie}
Let  $L$ be a Lie algebra, and $M$ be a $w(\hat{L})$-module. Then the functor
\[
\eta^{*}_{L}U^{\Ga(\Lie)}_{\Lie}\colon (\RLie/\hat{L})_{\ab}\to (\Lie/L)_{\ab}
\]
is given by
\[
\eta_{L}^{*}U^{\Ga(\Lie)}_{\Lie} (M)={}_{U(L)}M.
\]
\end{lemm}

\begin{proof}
Let
\[
\hat{L}\ltimes_{f} {}_{u(\hat{L})}M\to \hat{L}
\]
be an abelian group object in $(\RLie/\hat{L})_{\ab}$. We have the pullback diagram
\[
\begin{tikzcd}
L\times_{\hat{L}} (\hat{L}\ltimes_{f} {}_{u(\hat{L})}M) \arrow[r] \arrow[d]
& \hat{L}\ltimes_{f} {}_{u(\hat{L})}M \arrow[d] \\
L\arrow[r, "\eta_{L}"]
&\hat{L}.
\end{tikzcd}
\]
Since $L\times_{\hat{L}} (\hat{L}\ltimes_{f} {}_{u(\hat{L})}M)$ is spanned by the elements $(l,(l,m))$ with $l\in L$ and $m\in {}_{u(\hat{L})}M$, we get:
\[
\eta_{L}^{*}U^{\Ga(\Lie)}_{\Lie}(L\ltimes_{f} {}_{u(\hat{L})}M\to \hat{L})=L\times_{\hat{L}}(\hat{L}\ltimes_{f} {}_{u(\hat{L})}M)\to L,
\]
so, $\psi \colon L\times_{\hat{L}} (\hat{L}\ltimes_{f} {}_{u(\hat{L})}M) \to L$ is an abelian group object in $(\Lie/L)_{\ab}$. The kernel of $\psi$ is isomorphic to ${}_{U(L)}M$.
\end{proof}

\begin{prop}
\begin{enumerate}
\item Let $L$ be a Lie algebra and $M$ a $w(\hat{L})$-module. Then there is a comparison map
\[
\HQ_{\RLie}^{*}(\hat{L}; M) \to \HQ_{\Lie}^{*}(L; {}_{U(L)}M).
\]
\item Let $H$ be a restricted Lie algebra and $M$ a $w(H)$-module. Then there is a comparison map
\[
\HQ_{\RLie}^{*}(H; M) \to \HQ_{\Lie}^{*}(H; \te^* M)
\]
where $\te \colon U(H) \to w(H)$ is the ring homomorphism described in Lemma~\ref{lem:MapEnvAlgLie}.
\end{enumerate}
\end{prop}

\begin{proof}
Part~(1) follows directly from \cite{Frankland15}*{Proposition~4.12} and Lemma~\ref{lem:CoefLie}. Part~(2) is a specialization of Proposition~\ref{prop:CompHQ}~\eqref{item:CompHQcohom} to the operad $\P = \Lie$.
\end{proof}

Let us now show that the $p$-envelope functor passes to Beck modules. We will need the following observation:

\begin{rema}\label{rem:p-envelope}
Let $H$ be a restricted Lie algebra and $x_{i}\in H$. By Jacobson’s formula on $p$-th powers we have:
$$\left(\sum_{i} x_{i}\right)^{[p]}-\left(\sum_{i}x_{i}^{[p]}\right)\in [H,H].$$
Let $L$ be a Lie algebra with basis $(e_{i})_{i\in I}$ over a field $\F$. From the previous formula it follows by induction that the elements of $\hat{L}$ are of the type $\sum_{i\in I,n_{i}\geq 0} \F e_{i}^{p^{n_{i}}}\in U(L)$.  
\end{rema}

\begin{prop}\label{prop:Liepenvelopepasses}
For any field $\F$ of characteristic $p$, the $p$-envelope functor $F^{\Ga(\Lie)}_{\Lie}$ that freely adjoins a $p$-map passes to Beck modules.
\end{prop}

\begin{proof}
 For a Lie algebra $L$, the restricted Lie algebra $F^{\Gamma(\Lie)}_{\Lie}(L)$ is spanned by elements of the type $l^{[p^k]}$, for $l\in L$ and $k\in \N$, where $(-)^{[p^k]}$ represents the $k$-th iteration of the $p$-map in $F^{\Gamma(\Lie)}_{\Lie}(L)$.
 Let $L\ltimes M\to L$ be a Beck module over $L$, where $M$ is an $L$-module. Consider the induced split epimorphism $F^{\Gamma(\Lie)}_{\Lie}(L\ltimes M)\to F^{\Gamma(\Lie)}_{\Lie}(L)$. 
By \cite{StradeF88}*{\S 2.5, Proposition~5.3}, $F^{\Gamma(\Lie)}_{\Lie}$ preserves both surjections and injections. In particular, $l^{[p^k]}\mapsto (l,0)^{[p^k]}$ and $m^{[p^k]}\mapsto(0,m)^{[p^k]}$ induce injections $F^{\Gamma(\Lie)}_{\Lie}(L) \inj F^{\Gamma(\Lie)}_{\Lie}(L\ltimes M)$ and $F^{\Gamma(\Lie)}_{\Lie}(M)\inj F^{\Gamma(\Lie)}_{\Lie}(L\ltimes M)$. 
Under these injections, the abelian restricted Lie algebra $F^{\Gamma(\Lie)}_{\Lie}(M)$ is a restricted $F^{\Gamma(\Lie)}_{\Lie}(L)$-module. If $f$ is the $p$-map in $F^{\Gamma(\Lie)}_{\Lie}(M)$, then $(F^{\Gamma(\Lie)}_{\Lie}(M), f)$ is a Beck module over $F^{\Gamma(\Lie)}_{\Lie}(L)$. 
We then get an injection 
\begin{align*}
g \colon F^{\Gamma(\Lie)}_{\Lie}(L) \ltimes_f F^{\Gamma(\Lie)}_{\Lie}(M) &\inj F^{\Gamma(\Lie)}_{\Lie}(L\ltimes M) \\
(l^{[p^{k_1}]},m^{[p^{k_2}]}) &\mapsto (l,0)^{[p^{k_1}]}+(0,m)^{[p^{k_2}]}.
\end{align*}
From Remark~\ref{rem:p-envelope}, we deduce that $g$ is in fact an isomorphism, hence $F^{\Gamma(\Lie)}_{\Lie}(L\ltimes M\to L)$ has the structure of a Beck module over $F^{\Gamma(\Lie)}_{\Lie}(L)$.
\end{proof}

To conclude this section, we will specify the comparison maps of Proposition~\ref{prop:compmaps} in the case of restricted Lie algebra. Let $g \colon L\to H$ be a morphism of restricted lie algebra. Then, on the one hand, $\UU_{\Lie}L \otimes_{\UU_{\Lie} H} \Omega_{\Lie}(H)$ is spanned by the elements $ldh$ for $l\in L$, $h\in H$, under relations expressing the fact that $d$ is a (linear) $L$-derivation:
\[
\begin{cases}
d(\lambda h+h') = \lambda dh+dh' \\
d([h,h']) = g(h)dh' - g(h')dh \\
\end{cases}
\]
and the action of $L$ is given by the bracket in $L$ ($l\cdot l'dh=[l,l']$). On the other hand, $\UU_{\Gamma(\Lie)} L\otimes_{\UU_{\Gamma(\Lie)}H}\Omega_{\Gamma(\Lie)}(H)=w(L)\otimes_{w(H)}\Omega_{\RLie}(H)$ is spanned by elements $f^{k}ldh$ for $l\in U(L)$, $h\in H$ and $k\in \N$, under relation expressing the fact that $d$ is a Beck $L$-derivation in $\RLie$, that is, we also have:
\[
d(h^{[p]}) = \underbrace{g(h)\cdots g(h)}_{p-1}dh+fdh
\]
and the action of $w(L)$ is again given by the multiplication in $w(L)$. In particular, for $l\in L$, $(f^kl)\cdot f^{k'}l'dh=0$ as soon as $k'>0$. The comparison map
\[
    \UU_{\Lie} L\otimes_{\UU_{\Lie} H} \Omega_{\Lie}(H)\to \UU_{\Gamma(\Lie)} L\otimes_{\UU_{\Gamma(\Lie)}H} \Omega_{\Gamma(\Lie)}(H)
\]
from Proposition~\ref{prop:compmaps} is %
given 
by $ldh\mapsto f^0ldh$.

\section{Associative algebras versus restricted Lie algebras}\label{sec:compasslie}

Let $\F$ be a field of %
characteristic $p \neq 0$. Denote by $\eta$ the unit of the adjunction $u\vdash (-)_{\RLie}$.  Let $(L,(-)^{[p]})$ be a restricted Lie algebra and let $M$ be an $u(L)$-bimodule. We denote by ${}_{u(L)}M$ the left $u(L)$-module obtained from $M$ by the action $l\cdot m:=lm-ml$, where the dotless notation is for the left and right bimodule actions on $M$. Let $A\in \As$ be an associative algebra. We recall that the category of Beck $A$-modules is equivalent to the category of $A$-bimodules (see \cite{Barr96}*{\S 2.1}).  

\begin{lemm}
Let  $L$ be a restricted Lie algebra, and $M$ be a $u(L)$-bimodule. Then the functor
\[
\eta^{*}_{L}(-)_{\RLie}\colon (\As/u(L))_{\ab} \to (\RLie/L)_{\ab}
\]
is given by
\[
\eta_{L}^{*}(-)_{\RLie} (M)={}_{u(L)}M.
\]
\end{lemm}

\begin{proof}
Consider the associative algebra obtained by semidirect product $u(L)\ltimes M$, whose multiplication is given by $(u,m)(u',m')=(uu',um'+mu')$. The category of Beck $u(L)$-modules is equivalent to the category of $u(L)$-bimodules. Under this equivalence, the $u(L)$-bimodule $M$ is associated to the abelian group object 
$u(L)\ltimes M\to u(L)$ in  $(As/u(L))_{\ab}$. The Lie bracket in $(u(L)\ltimes M)_{\RLie}$ is given by
\[
[(u,m),(u',m')]=(uu'-u'u, um'+mu'-u'm-m'u),
\]
and the $p$-map is given by
\[
(u,m)\mapsto (u,m)^{p}.
\]
Therefore the $p$-map on elements of the form $(0,m)$ is zero. We have a restricted Lie algebra homomorphism $\eta_{L}\colon L\to u(L)_{\RLie}$, and a pullback functor $\eta_{L}^{*} \colon \RLie/u(L)_{\RLie} \to \RLie/L$.
We get a pullback diagram
\[
\begin{tikzcd}
L\times_{u(L)_{\RLie}}(u(L)\ltimes M)_{\RLie} \arrow[r] \arrow[d]
& (u(L)\ltimes M)_{\RLie}  \arrow[d] \\
L\arrow[r, "\eta_{L}"]
& u(L)_{\RLie},
\end{tikzcd}
\]
and $L\times_{u(L)_{\RLie}}(u(L)\ltimes M)_{\RLie}$ is spanned by elements $(l,(l,m))$ for $l\in L$ and $m\in M$. The Lie bracket of this restricted Lie algebra is given by
\[
[(l,(l,m)),(l',(l',m)]=\left([l,l'], [(l,m),(l',m')]\right),
\]
and the  $p$-map is given by
\[
(l,(l,m))\mapsto (l^{[p]}, (l,m)^{p}).
\]
%Therefore, 
Applying the functor $\eta_{L}^{*}(-)_{\RLie}$ to the map $u(L)\ltimes M\to u(L)$ yields a map
%\[
%\eta_{L}^{*}(-)_{\RLie}(u(L)\ltimes M\to u(L))=\eta_{L}^{*}((u(L)\ltimes M)_{\RLie}\to u(L)_{\RLie})=L\times_{u(L)_{\RLie}}(u(L)\ltimes M)_{\RLie}\to L,
%\]
%\[
%\eta_{L}^{*}((u(L)\ltimes M)_{\RLie}\to u(L)_{\RLie}) = L\times_{u(L)_{\RLie}}(u(L)\ltimes M)_{\RLie}\to L,
%\]
\[
\phi \colon L\times_{u(L)_{\RLie}}(u(L)\ltimes M)_{\RLie} \to L
\]
which 
%and $\phi\colon L\times_{u(L)_{\RLie}}(u(L)\ltimes M)_{\RLie}\to L$ 
is an abelian group object in $(\RLie/L)_{\ab}$. The action of $L$ on $M:=\ker \phi$ is given by
\rescale{\begin{align*}
	l\cdot (0,(0,m))=[(l,(l,0)),(0,(0,m))]=\left([l,0], [(l,0),(0,m)]\right)=(0, (0, lm-ml)).
\end{align*}}
With this action, the module $M$ is the restricted Lie module ${}_{u(L)}M$. The $p$-map on $M$ is trivial, $m\mapsto 0$.
\end{proof}
\begin{theo}\label{thm:AssLie}
Let $L$ be a restricted Lie algebra and $M$ a $u(L)$-bimodule. Then there is an isomorphism
\[
\HQ_{\As}^{*}(u(L); M) \cong  \HQ_{\RLie}^{*}(L; {}_{u(L)}M).
\]
\end{theo}

\begin{proof}
The restricted enveloping algebra functor $u$ preserves weak equivalences (see \cite{Priddy70pri}*{2.8}). Thus by \cite{Frankland15}*{Proposition~4.12} the comparison map between Quillen cohomology in both categories is an isomorphism.
\end{proof}

For an associative algebra $A$ over a field, Quillen cohomology agrees with Hochschild cohomology up to a shift:
\[
\HQ_{\As}^{n}(A;M) \cong \HH^{n+1}(A;M) \quad \text{for } n > 0
\]
and a small change in degree $0$ \cite{Quillen70}*{Proposition~3.6}.

\section{Good triples of operads and comparison maps in Quillen (co)homology}\label{sec:Good}

J.-L. Loday defined and studied generalized bialgebras and triples of operads in \cite{Loday08}. In this section we prove a comparison isomorphism theorem for Quillen cohomology in the context of good triples of operads. 
Only in this section 
we suppose that the ground field $\F$ has characteristic zero. In that case, if $\P$ is an operad then the norm %
map $\Tr \colon  S(\P) \rightarrow \Ga(\P)$ is a natural isomorphism \cite{Fresse00}. Therefore a $\Gamma(\P)$-algebra doesn't carry more structure than a $\P$-algebra.

\begin{defi}
Let $\CC$ be an algebraic operad. The \Def{primitive part} of a coalgebra $C$ over the operad $\CC$ is defined by 
\[
\Prim\,C:=\{x\in C\;|\; \delta(x)=0,\,\text{for any generating cooperation}\, \delta \}.
\]
There is a filtration on $C$ given by 
\[
F_{r}(C)=\{x\in C\;|\; \delta(x)=0,\,\text{for any}\; \delta \in \CC(n),\;n>r\}.
\]
We note  that $F_{1}(C)$ is the primitive part of $C$. The above filtration is called the \Def{primitive filtration}. A coalgebra $C$ is \Def{connected} (or \Def{conilpotent}) if $C = \bigcup_{r\geq 1}F_{r}C$.
\end{defi}

Let $\A$ and $\CC$ be two algebraic operads. 
A $\CC^{c}$-$\A$-bialgebra  $H$ is a vector space which is a $\A$-algebra and $\CC$-coalgebra such that the operations of $\A$ and cooperations of $\CC$ acting on $H$ satisfy some compatibility relations.

Let $(\CC,\A,\P)$ be a good triple of operads, as defined in \cite{Loday08}*{2.4.1, 2.5.6}. Then $\P$ is the primitive operad $\Prim_{\CC}\, \A$, which %
is a suboperad of $\A$. The inclusion $\Prim_{\CC}\, \A \subseteq \A$ induces a functor 
\[
G \colon \Alg{\A} \to \Alg{\P}
\]
which is a forgetful functor in a sense that the composition 
$\Alg{\A} \xrightarrow{G} \Alg{\P} \to \Vect{\F}$ 
is the forgetful functor 
$\Alg{\A} \to \Vect{\F}$. 
The functor $G$ admits a left adjoint functor 
\[
U \colon \Alg{\P} \to \Alg{\A}
\]
called the universal enveloping functor. Let $\eta$ be the unit of this adjunction. If $P$ is a $\P$-algebra  then there is a $\P$-algebra  morphism $\eta_{P}\colon P\to GU(P)$. Moreover, the universal enveloping algebra $U(L)$ of a $\P$-algebra $L$ is a connected $\CC^{c}$-$\A$-bialgebra and $\Prim\, U(L)=L$. 

Let $\ZZ=\A/(\bar{\P})$ be the quotient operad of $\A$ by the ideal generated by the (nontrivial) primitive operations. Then J.-L. Loday proved a generalised Poincaré--Birkhoff--Witt theorem. In particular, it is proved in \cite{Loday08}*{Theorem~3.1.4} that for any $\P$-algebra $L$ there is an isomorphism of $\ZZ$-algebras 
\[
\ZZ(L)\to gr\, U(L).
\]

\begin{defi}
A morphism 
of simplicial $\P$-algebras 
$f \colon L \to L'$ is a \Def{weak equivalence} if it 
is a weak equivalence of underlying simplicial $\F$-vector spaces, i.e., the induced map on homotopy groups 
$\pi_{*}(f) \colon \pi_{*}(L) \ral{\cong} \pi_{*}(L')$ 
is an isomorphism.
\end{defi}

The following generalizes an argument used in \cite{Priddy70pri}*{2.8}.

\begin{prop}\label{prop:weak}
Let $(\CC,\A,\P)$ be a good triple of operads. Then the universal enveloping functor $U\colon \Alg{\P} \to \Alg{\A}$ preserves weak equivalences.
\end{prop}

\begin{proof}
If $L$ is a simplicial $\P$-algebra then $U(L)$ is a simplicial $\A$-algebra. Using the Dold--Kan correspondence, we denote by $U(L)$ the associated chain complex. The primitive filtration on the universal enveloping algebra makes $U(L)$ a filtered complex.
By Loday's generalised Poincaré--Birkhoff--Witt  theorem the associated spectral sequence $E_{r}(U(L))$ of $U(L)$ satisfies 
\[
\begin{cases}
E_{0}(U(L)) \cong \ZZ(L)\\
E_{1}(U(L)) \cong H_{*}(\ZZ(L)).
\end{cases}
\]
Let $L,L'$ be two simplicial $\P$-algebras and $f \colon L \to L'$ a weak equivalence. The morphism $U(f) \colon U(L) \to U(L')$ preserves filtrations and it induces a morphism of spectral sequences 
$E_{r}(U(L))\to E_{r}(U(L'))$. 
Since $f$ is a weak equivalence, it follows by a theorem of Dold \cite{Dold58} that 
\[
H_{*}(\ZZ(L)) \cong H_{*}(\ZZ(L')).
\]
Therefore 
\[
E_{1}U(f) \colon E_{1}(U(L)) \to E_{1}(U(L'))
\]
is an isomorphism. The filtrations of $U(L)$ and $U(L')$ 
are complete and bounded below. Hence the spectral sequences converge and the induced map $E_{\infty}U(f)$ is an isomorphism. It follows that $U$ preserves weak equivalences.
\end{proof}

Let $\mathcal{T}$ be a operad and $S$ a $\mathcal{T}$-algebra. The category of $S$-modules over $\mathcal{T}$ is equivalent to the category of abelian group objects of $\Alg{\mathcal{T}}/S$, i.e., the category of Beck $S$-modules in the category of $\mathcal{T}$-algebras (see \cites{LodayV12,GoerssH00}). Hence the category of $U(P)$-modules over the operad $\A$ is equivalent to the category of Beck $U(P)$-modules. Under this equivalence the $U(P)$-module $M$  is associated to the abelian group object $U(P)\ltimes M\to U(P)$, where by $U(P)\ltimes M$ we denote the semidirect product of $U(P)$ by $M$ in the category of $\A$-algebras. We have the following pullback diagram in the category of $\P$-algebras,
\[
\begin{tikzcd}
P\times_{GU(P)}G(U(P)\ltimes M) \arrow[r] \arrow[d]
& G(U(P)\ltimes M) \arrow[d] \\
P\arrow[r, "\eta_{P}"]
& GU(P)
\end{tikzcd}
\]
and 
\[
0\to M\to P\times_{GU(P)}G(U(P)\ltimes M)\to P\to 0
\]
is an abelian extension in the category of $\P$-algebras. %
This induces 
on $M$ a structure of $P$-module over the operad $\P$ which we denote by ${}_{P}M$.

\begin{theo}\label{thm:good}
Let $(\CC,\A,\P)$ be a good triple of operads. Let $P$ be an $\P$-algebra and $M$ a Beck $U(P)$-module. Then we have the following isomorphism 
\[
\HQ_{\Alg{\A}}^{*}(U(P); M)\simeq  \HQ_{\Alg{\P}}^{*}(P; {}_{P}M).
\]
\end{theo}

\begin{proof}
By Proposition~\ref{prop:weak} the enveloping algebra functor $U$ preserves weak equivalences. The claim then follows from \cite{Frankland15}*{Proposition~4.12}.
\end{proof}

\begin{ex}\label{ex:GoodCom}
If we consider the good triple $(\Com, \As, \Lie)$, then by Theorem~\ref{thm:good} we recover in characteristic zero a well known result \cite{CartanE56}*{\S XIII, Theorem~5.1}. In particular, let $L$ be a Lie algebra over $\F$ and $M$ a $U(L)$-bimodule. By Theorem~\ref{thm:good} we have the following isomorphism 
\[
\HQ_{\Alg{\As}}^{*}(U(L);M) \cong  \HQ_{\Alg{\Lie}}^{*}(L;{}_{L}M)
\]
where ${}_{L}M$ is $M$ viewed as a left Lie $L$-module via the action $l\cdot m:=lm-ml$ for all $l\in L$ and $m\in M$.
Moreover, Quillen cohomology for the category of associative algebras is shifted Hochschild cohomology of associative algebras and  
Quillen cohomology of Lie algebras is shifted Chevalley--Eilenberg cohomology of Lie algebras (see \cite{Barr96}). In other words, we have 
\[
\HH_{Hoch}^{*}(U(L); M) \cong  H_{\Lie}^{*}(L; {}_{L}M).
\]
\end{ex}

As another example, let us consider dendriform algebras, defined by J.-L.~Loday in \cite{Loday01}.

\begin{defi}
A \Def{dendriform algebra} $H$ is a vector space over $\F$ equipped with two binary operations $\prec,\succ \colon  H\otimes H \rightarrow H$  such that 
\begin{align*}
(x\prec y)\prec z &=x\prec(y\ast z)\\
(x\succ y)\prec z&= x\succ(y\prec z)\\
(x\ast y)\succ z &=x\succ(y\succ z)
\end{align*}
where $x\ast y = x\succ y + x\prec y$ and $x,y,z \in H$. 
\end{defi}

The identities make the product $\ast$ associative. %
Notice that a dendriform algebra is an associative algebra whose product splits into two binary operations which satisfy the above identities. Dendriform algebras are Koszul dual to diassociative algebras; see \cite{Loday01}*{\S 8} and \cite{LodayV12}*{\S 13.6}. 
Besides, M.~Gerstenhaber and A.~Voronov introduced the notion of brace algebras in \cite{GerstenhaberV95hom}. We denote by $\mathcal{D}$ the operad associated to dendriform algebras and by $\mathcal{B}$ the operad associated to brace algebras. 

M.~Ronco in \cite{Ronco02} proved a Milnor--Moore type theorem for dendriform algebras. In particular, there is a good triple of operads $(\As,\mathcal{D},\mathcal{B})$. The concept of bimodule over a dendriform algebra was defined by M.~Aguiar in \cite{Aguiar04}. Cohomology of dendriform algebras with coefficients in bimodules has been studied by A.~Das in \cite{Das22}. 

Let $D$ be a dendriform algebra. The notion of Beck $D$-module is equivalent to the notion of bimodule over $D$. Applying Theorem~\ref{thm:good} to the good triple $(\As,\mathcal{D},\mathcal{B})$ yields the following.

\begin{prop}
Let $B$ be a $\mathcal{B}$-algebra, and let $M$ be a $U_{dend}(B)$-bimodule, where $U_{dend}(B)$ denotes the enveloping dendriform algebra of the brace algebra $B$. We have the following isomorphism:
\[
\HQ_{\Alg{\mathcal{D}}}^{*}(U_{dend}(B); M) \cong \HQ_{\Alg{\mathcal{B}}}^{*}(B; {}_{B}M).
\]
\end{prop}

\begin{rema}
In %
positive 
characteristic, for classical bialgebras we proved Theorem~\ref{thm:AssLie}. 
We notice that the primitives of a classical bialgebra form a restricted Lie algebra, i.e., a $\Ga(\Lie)$-algebra. If one wants to extend this result to the context of generalized bialgebras, the right framework seems to be the category of $\Gamma(\P)$-algebras.
\end{rema}

\vspace*{6pt}


\begin{bibdiv}
\begin{biblist}*{labels={alphabetic}}

\bib{Aguiar04}{article}{
  author={Aguiar, Marcelo},
  title={Infinitesimal bialgebras, pre-Lie and dendriform algebras},
  conference={ title={Hopf algebras}, },
  book={ series={Lecture Notes in Pure and Appl. Math.}, volume={237}, publisher={Dekker, New York}, },
  isbn={0-8247-5566-9},
  date={2004},
  pages={1--33},
  review={\MR {2051728}},
}

\bib{Andre67}{book}{
  author={Andr{\'e}, Michel},
  title={M\'ethode simpliciale en alg\`ebre homologique et alg\`ebre commutative},
  language={French},
  series={Lecture Notes in Mathematics, Vol. 32},
  publisher={Springer-Verlag, Berlin-New York},
  date={1967},
  pages={iii+122},
  review={\MR {0214644 (35 \#5493)}},
}

\bib{Andre74}{book}{
  author={Andr\'e, Michel},
  title={Homologie des alg\`ebres commutatives},
  series={Die Grundlehren der mathematischen Wissenschaften},
  volume={206},
  publisher={Springer-Verlag},
  date={1974},
}

\bib{Balavoine97}{article}{
  author={Balavoine, David},
  title={Deformations of algebras over a quadratic operad},
  conference={ title={Operads: Proceedings of Renaissance Conferences}, address={Hartford, CT/Luminy}, date={1995}, },
  book={ series={Contemp. Math.}, volume={202}, publisher={Amer. Math. Soc., Providence, RI}, },
  isbn={0-8218-0513-4},
  date={1997},
  pages={207--234},
  review={\MR {1436922}},
  doi={10.1090/conm/202/02581},
}

\bib{Barr96}{article}{
  author={Barr, Michael},
  title={Cartan-Eilenberg cohomology and triples},
  journal={J. Pure Appl. Algebra},
  volume={112},
  date={1996},
  number={3},
  pages={219--238},
  issn={0022-4049},
  review={\MR {1410176}},
  doi={10.1016/0022-4049(95)00138-7},
}

\bib{BarrB69}{article}{
  author={Barr, Michael},
  author={Beck, Jon},
  title={Homology and standard constructions},
  conference={ title={Sem. on Triples and Categorical Homology Theory (ETH, Z\"{u}rich, 1966/67)}, },
  book={ publisher={Springer, Berlin}, },
  date={1969},
  pages={245--335},
  review={\MR {0258917}},
}

\bib{BauesB11}{article}{
  author={Baues, Hans-Joachim},
  author={Blanc, David},
  title={Comparing cohomology obstructions},
  journal={J. Pure Appl. Algebra},
  volume={215},
  date={2011},
  number={6},
  pages={1420--1439},
  issn={0022-4049},
  review={\MR {2769241 (2012f:55022)}},
  doi={10.1016/j.jpaa.2010.09.003},
}

\bib{Beck67}{book}{
  author={Beck, Jonathan Mock},
  title={TRIPLES, ALGEBRAS AND COHOMOLOGY},
  note={Thesis (Ph.D.)--Columbia University},
  publisher={ProQuest LLC, Ann Arbor, MI},
  date={1967},
  pages={115},
  review={\MR {2616383}},
}

\bib{Berthelot74}{book}{
  author={Berthelot, Pierre},
  title={Cohomologie cristalline des sch\'{e}mas de caract\'{e}ristique $p>0$},
  language={French},
  series={Lecture Notes in Mathematics, Vol. 407},
  publisher={Springer-Verlag, Berlin-New York},
  date={1974},
  pages={604},
  review={\MR {0384804}},
}

\bib{BlancDG04}{article}{
  author={Blanc, David},
  author={Dwyer, William G.},
  author={Goerss, Paul G.},
  title={The realization space of a $\Pi $-algebra: a moduli problem in algebraic topology},
  journal={Topology},
  volume={43},
  number={4},
  pages={857 \ndash 892},
  date={2004},
}

\bib{BlancJT12}{article}{
  author={Blanc, David},
  author={Johnson, Mark W.},
  author={Turner, James M.},
  title={Higher homotopy operations and Andr\'{e}-Quillen cohomology},
  journal={Adv. Math.},
  volume={230},
  date={2012},
  number={2},
  pages={777--817},
  issn={0001-8708},
  review={\MR {2914966}},
  doi={10.1016/j.aim.2012.02.009},
}

\bib{Borceux94v2}{book}{
  author={Borceux, Francis},
  title={Handbook of Categorical Algebra 2: Categories and Structures},
  series={Encyclopedia of Mathematics and its Applications},
  volume={51},
  publisher={Cambridge University Press},
  date={1994},
}

\bib{Borel91}{book}{
  author={Borel, Armand},
  title={Linear algebraic groups},
  series={Graduate Texts in Mathematics},
  volume={126},
  edition={2},
  publisher={Springer-Verlag, New York},
  date={1991},
  pages={xii+288},
  isbn={0-387-97370-2},
  review={\MR {1102012}},
  doi={10.1007/978-1-4612-0941-6},
}

\bib{Cartan56exp7}{article}{
  author={Cartan, Henri},
  title={Puissances divis\'ees},
  book={ title={Alg\`ebres d'Eilenberg-MacLane et homotopie}, series={S\'eminaire Henri Cartan, 1954/1955}, volume={7}, number={1}, language={French}, edition={2}, date={1956}, review={\MR {0087935 (19,439a)}}, },
  date={1956},
  part={Expos\'e 7},
  pages={1 \ndash 11},
}

\bib{CartanE56}{book}{
  author={Cartan, Henri},
  author={Eilenberg, Samuel},
  title={Homological algebra},
  publisher={Princeton University Press, Princeton, NJ},
  date={1956},
  pages={xv+390},
  review={\MR {0077480}},
}

\bib{CegarraA86}{article}{
  author={Cegarra, Antonio M.},
  author={Aznar, Enrique R.},
  title={An exact sequence in the first variable for torsor cohomology: the $2$-dimensional theory of obstructions},
  journal={J. Pure Appl. Algebra},
  volume={39},
  date={1986},
  number={3},
  pages={197--250},
  issn={0022-4049},
  review={\MR {0821890}},
  doi={10.1016/0022-4049(86)90145-3},
}

\bib{Das22}{article}{
  author={Das, Apurba},
  title={Cohomology and deformations of dendriform algebras, and ${\rm Dend}_{\infty }$-algebras},
  journal={Comm. Algebra},
  volume={50},
  date={2022},
  number={4},
  pages={1544--1567},
  issn={0092-7872},
  review={\MR {4391505}},
  doi={10.1080/00927872.2021.1985130},
}

\bib{Dokas04}{article}{
  author={Dokas, Ioannis},
  title={Quillen-Barr-Beck (co-) homology for restricted Lie algebras},
  journal={J. Pure Appl. Algebra},
  volume={186},
  date={2004},
  number={1},
  pages={33--42},
  issn={0022-4049},
  review={\MR {2025911}},
  doi={10.1016/S0022-4049(03)00120-8},
}

\bib{Dokas09}{article}{
  author={Dokas, Ioannis},
  title={Triple cohomology and divided powers algebras in prime characteristic},
  journal={J. Aust. Math. Soc.},
  volume={87},
  date={2009},
  number={2},
  pages={161--173},
  issn={1446-7887},
  review={\MR {2551116}},
  doi={10.1017/S1446788708081019},
}

\bib{Dokas15}{article}{
  author={Dokas, Ioannis},
  title={Torsors and the Quillen-Barr-Beck cohomology for restricted Lie algebras},
  journal={Homology Homotopy Appl.},
  volume={17},
  date={2015},
  number={1},
  pages={203--234},
  issn={1532-0073},
  review={\MR {3338548}},
  doi={10.4310/HHA.2015.v17.n1.a10},
}

\bib{Dokas23}{article}{
  author={Dokas, Ioannis},
  title={On K\"{a}hler differentials of divided power algebras},
  journal={J. Homotopy Relat. Struct.},
  volume={18},
  date={2023},
  number={2-3},
  pages={153--176},
  issn={2193-8407},
  review={\MR {4645755}},
  doi={10.1007/s40062-023-00325-2},
}

\bib{Dold58}{article}{
  author={Dold, Albrecht},
  title={Homology of symmetric products and other functors of complexes},
  journal={Ann. of Math. (2)},
  volume={68},
  date={1958},
  pages={54--80},
  issn={0003-486X},
  review={\MR {0097057}},
  doi={10.2307/1970043},
}

\bib{DummitF04}{book}{
  author={Dummit, David S.},
  author={Foote, Richard M.},
  title={Abstract algebra},
  edition={3},
  publisher={John Wiley \& Sons, Inc., Hoboken, NJ},
  date={2004},
  pages={xii+932},
  isbn={0-471-43334-9},
  review={\MR {2286236}},
}

\bib{Duskin75}{article}{
  author={Duskin, J.},
  title={Simplicial methods and the interpretation of ``triple'' cohomology},
  journal={Mem. Amer. Math. Soc.},
  volume={3},
  date={1975},
  pages={v+135},
  issn={0065-9266},
  review={\MR {0393196}},
  doi={10.1090/memo/0163},
}

\bib{Frankland10qui}{book}{
  author={Frankland, Martin},
  title={Quillen Cohomology of Pi-Algebras and Application to their Realization},
  note={Thesis (Ph.D.)--Massachusetts Institute of Technology},
  publisher={ProQuest LLC, Ann Arbor, MI},
  date={2010},
  pages={(no paging)},
  review={\MR {2801771}},
}

\bib{Frankland11}{article}{
  author={Frankland, Martin},
  title={Moduli spaces of 2-stage Postnikov systems},
  journal={Topology Appl.},
  volume={158},
  date={2011},
  number={11},
  pages={1296--1306},
  issn={0166-8641},
  review={\MR {2806362}},
  doi={10.1016/j.topol.2011.05.002},
}

\bib{Frankland15}{article}{
  author={Frankland, Martin},
  title={Behavior of Quillen (co)homology with respect to adjunctions},
  journal={Homology Homotopy Appl.},
  volume={17},
  date={2015},
  number={1},
  pages={67--109},
  issn={1532-0073},
  review={\MR {3338541}},
  doi={10.4310/HHA.2015.v17.n1.a3},
}

\bib{Fresse98}{article}{
  author={Fresse, Benoit},
  title={Homologie de Quillen pour les alg\`ebres de Poisson},
  language={French, with English and French summaries},
  journal={C. R. Acad. Sci. Paris S\'er. I Math.},
  volume={326},
  date={1998},
  number={9},
  pages={1053--1058},
  issn={0764-4442},
  review={\MR {1647186}},
}

\bib{Fresse00}{article}{
  author={Fresse, Benoit},
  title={On the homotopy of simplicial algebras over an operad},
  journal={Trans. Amer. Math. Soc.},
  volume={352},
  date={2000},
  number={9},
  pages={4113 \ndash 4141},
  issn={0002-9947},
  review={\MR {1665330 (2000m:18015)}},
  doi={10.1090/S0002-9947-99-02489-7},
}

\bib{Fresse04}{article}{
  author={Fresse, Benoit},
  title={Koszul duality of operads and homology of partition posets},
  conference={ title={Homotopy theory: relations with algebraic geometry, group cohomology, and algebraic $K$-theory}, },
  book={ series={Contemp. Math.}, volume={346}, publisher={Amer. Math. Soc., Providence, RI}, },
  isbn={0-8218-3285-9},
  date={2004},
  pages={115--215},
  review={\MR {2066499}},
  doi={10.1090/conm/346/06287},
}

\bib{Gerstenhaber64}{article}{
  author={Gerstenhaber, Murray},
  title={On the deformation of rings and algebras},
  journal={Ann. of Math. (2)},
  volume={79},
  date={1964},
  pages={59--103},
  issn={0003-486X},
  review={\MR {0171807}},
  doi={10.2307/1970484},
}

\bib{GerstenhaberV95hom}{article}{
  author={Gerstenhaber, Murray},
  author={Voronov, Alexander A.},
  title={Homotopy $G$-algebras and moduli space operad},
  journal={Internat. Math. Res. Notices},
  date={1995},
  number={3},
  pages={141--153},
  issn={1073-7928},
  review={\MR {1321701}},
  doi={10.1155/S1073792895000110},
}

\bib{GetzlerJ94}{article}{
  author={Getzler, Ezra},
  author={Jones, J.D.S.},
  title={Operads, homotopy algebra and iterated integrals for double loop spaces},
  eprint={arXiv:hep-th/9403055},
  date={1994},
  status={Preprint},
}

\bib{GinzburgK94}{article}{
  author={Ginzburg, Victor},
  author={Kapranov, Mikhail},
  title={Koszul duality for operads},
  journal={Duke Math. J.},
  volume={76},
  date={1994},
  number={1},
  pages={203--272},
  issn={0012-7094},
  review={\MR {1301191}},
  doi={10.1215/S0012-7094-94-07608-4},
}

\bib{Glenn82}{article}{
  author={Glenn, Paul G.},
  title={Realization of cohomology classes in arbitrary exact categories},
  journal={J. Pure Appl. Algebra},
  volume={25},
  date={1982},
  number={1},
  pages={33--105},
  issn={0022-4049},
  review={\MR {0660389}},
  doi={10.1016/0022-4049(82)90094-9},
}

\bib{Goerss90and}{article}{
  author={Goerss, Paul G.},
  title={Andr\'e-Quillen cohomology and the homotopy groups of mapping spaces: understanding the $E_2$-term of the Bousfield-Kan spectral sequence},
  journal={J. Pure Appl. Algebra},
  volume={63},
  date={1990},
  number={2},
  pages={113--153},
  issn={0022-4049},
  review={\MR {1043745}},
}

\bib{GoerssH00}{article}{
  author={Goerss, Paul G.},
  author={Hopkins, Michael J.},
  title={Andr\'e-Quillen (co)-homology for simplicial algebras over simplicial operads},
  conference={ title={Une d\'egustation topologique [Topological morsels]: homotopy theory in the Swiss Alps}, address={Arolla}, date={1999}, },
  book={ series={Contemp. Math.}, volume={265}, publisher={Amer. Math. Soc., Providence, RI}, },
  date={2000},
  pages={41--85},
  review={\MR {1803952}},
}

\bib{GoerssH04spa}{article}{
  author={Goerss, P. G.},
  author={Hopkins, M. J.},
  title={Moduli spaces of commutative ring spectra},
  conference={ title={Structured ring spectra}, },
  book={ series={London Math. Soc. Lecture Note Ser.}, volume={315}, publisher={Cambridge Univ. Press}, place={Cambridge}, },
  date={2004},
  pages={151 \ndash 200},
  review={\MR {2125040 (2006b:55010)}},
  doi={10.1017/CBO9780511529955.009},
}

\bib{GoerssJ09}{book}{
  author={Goerss, Paul G.},
  author={Jardine, John F.},
  title={Simplicial homotopy theory},
  series={Modern Birkh\"auser Classics},
  note={Reprint of the 1999 edition [MR1711612]},
  publisher={Birkh\"auser Verlag, Basel},
  date={2009},
  pages={xvi+510},
  isbn={978-3-0346-0188-7},
  review={\MR {2840650}},
  doi={10.1007/978-3-0346-0189-4},
}

\bib{Grothendieck68}{article}{
  author={Grothendieck, A.},
  title={Crystals and the de Rham cohomology of schemes},
  note={Notes by I. Coates and O. Jussila},
  conference={ title={Dix expos\'{e}s sur la cohomologie des sch\'{e}mas}, },
  book={ series={Adv. Stud. Pure Math.}, volume={3}, publisher={North-Holland, Amsterdam}, },
  date={1968},
  pages={306--358},
  review={\MR {0269663}},
}

\bib{Hinich03}{article}{
  author={Hinich, Vladimir},
  title={Tamarkin's proof of Kontsevich formality theorem},
  journal={Forum Math.},
  volume={15},
  date={2003},
  number={4},
  pages={591--614},
  issn={0933-7741},
  review={\MR {1978336}},
  doi={10.1515/form.2003.032},
}

\bib{Hochschild54}{article}{
  author={Hochschild, G.},
  title={Cohomology of restricted Lie algebras},
  journal={Amer. J. Math.},
  volume={76},
  date={1954},
  pages={555--580},
  issn={0002-9327},
  review={\MR {0063361}},
  doi={10.2307/2372701},
}

\bib{Ikonicoff20}{article}{
  author={Ikonicoff, Sacha},
  title={Divided power algebras over an operad},
  journal={Glasg. Math. J.},
  volume={62},
  date={2020},
  number={3},
  pages={477--517},
  issn={0017-0895},
  review={\MR {4133334}},
  doi={10.1017/s0017089519000223},
}

\bib{Ikonicoff23}{article}{
  author={Ikonicoff, Sacha},
  title={Divided power algebras and distributive laws},
  journal={J. Pure Appl. Algebra},
  volume={227},
  date={2023},
  number={8},
  pages={Paper No. 107356},
  issn={0022-4049},
  review={\MR {4549007}},
  doi={10.1016/j.jpaa.2023.107356},
}

\bib{Jacobson37}{article}{
  author={Jacobson, Nathan},
  title={Abstract derivation and Lie algebras},
  journal={Trans. Amer. Math. Soc.},
  volume={42},
  date={1937},
  number={2},
  pages={206--224},
  issn={0002-9947},
  review={\MR {1501922}},
  doi={10.2307/1989656},
}

\bib{Jacobson62}{book}{
  author={Jacobson, Nathan},
  title={Lie algebras},
  series={Interscience Tracts in Pure and Applied Mathematics, No. 10},
  publisher={Interscience Publishers (a division of John Wiley \& Sons, Inc.), New York-London},
  date={1962},
  pages={ix+331},
  review={\MR {0143793}},
}

\bib{Keller05}{article}{
  author={Keller, B.},
  title={Deformation quantization after Kontsevich and Tamarkin},
  conference={ title={D\'{e}formation, quantification, th\'{e}orie de Lie}, },
  book={ series={Panor. Synth\`eses}, volume={20}, publisher={Soc. Math. France, Paris}, },
  isbn={2-85629-183-X},
  date={2005},
  pages={19--62},
  review={\MR {2274224}},
}

\bib{KontsevichS00}{article}{
  author={Kontsevich, Maxim},
  author={Soibelman, Yan},
  title={Deformations of algebras over operads and the Deligne conjecture},
  conference={ title={Conf\'{e}rence Mosh\'{e} Flato 1999, Vol. I (Dijon)}, },
  book={ series={Math. Phys. Stud.}, volume={21}, publisher={Kluwer Acad. Publ., Dordrecht}, },
  date={2000},
  pages={255--307},
  review={\MR {1805894}},
}

\bib{Livernet98}{book}{
  author={Livernet, Muriel},
  title={Homotopie rationnelle des alg\`ebres sur une op\'{e}rade},
  language={French, with French summary},
  series={Pr\'{e}publication de l'Institut de Recherche Math\'{e}matique Avanc\'{e}e [Prepublication of the Institute of Advanced Mathematical Research]},
  volume={1998/32},
  note={Th\`ese, Universit\'{e} Louis Pasteur (Strasbourg I), Strasbourg, 1998},
  publisher={Universit\'{e} Louis Pasteur, D\'{e}partement de Math\'{e}matique, Institut de Recherche Math\'{e}matique Avanc\'{e}e, Strasbourg},
  date={1998},
  pages={vi+88},
  review={\MR {1695322}},
}

\bib{Loday01}{article}{
  author={Loday, Jean-Louis},
  title={Dialgebras},
  conference={ title={Dialgebras and related operads}, },
  book={ series={Lecture Notes in Math.}, volume={1763}, publisher={Springer, Berlin}, },
  isbn={3-540-42194-7},
  date={2001},
  pages={7--66},
  review={\MR {1860994}},
  doi={10.1007/3-540-45328-8\_2},
}

\bib{Loday08}{article}{
  author={Loday, Jean-Louis},
  title={Generalized bialgebras and triples of operads},
  language={English, with English and French summaries},
  journal={Ast\'{e}risque},
  number={320},
  date={2008},
  pages={x+116},
  issn={0303-1179},
  isbn={978-2-85629-257-0},
  review={\MR {2504663}},
}

\bib{LodayV12}{book}{
  author={Loday, Jean-Louis},
  author={Vallette, Bruno},
  title={Algebraic operads},
  series={Grundlehren der Mathematischen Wissenschaften [Fundamental Principles of Mathematical Sciences]},
  volume={346},
  publisher={Springer, Heidelberg},
  date={2012},
  pages={xxiv+634},
  isbn={978-3-642-30361-6},
  review={\MR {2954392}},
  doi={10.1007/978-3-642-30362-3},
}

\bib{May66res}{article}{
  author={May, J. P.},
  title={The cohomology of restricted Lie algebras and of Hopf algebras},
  journal={J. Algebra},
  volume={3},
  date={1966},
  pages={123--146},
  issn={0021-8693},
  review={\MR {193126}},
  doi={10.1016/0021-8693(66)90009-3},
}

\bib{Miller84}{article}{
  author={Miller, Haynes},
  title={The Sullivan conjecture on maps from classifying spaces},
  journal={Ann. of Math. (2)},
  volume={120},
  date={1984},
  number={1},
  pages={39--87},
  issn={0003-486X},
  review={\MR {750716}},
}

\bib{Milles11}{article}{
  author={Mill\`es, Joan},
  title={Andr\'{e}-Quillen cohomology of algebras over an operad},
  journal={Adv. Math.},
  volume={226},
  date={2011},
  number={6},
  pages={5120--5164},
  issn={0001-8708},
  review={\MR {2775896}},
  doi={10.1016/j.aim.2011.01.002},
}

\bib{Milner75}{article}{
  author={Mil\cprime ner, A. A.},
  title={Irreducible representations of modular Lie algebras},
  language={Russian},
  journal={Izv. Akad. Nauk SSSR Ser. Mat.},
  volume={39},
  date={1975},
  number={6},
  pages={1240--1259, 1437},
  issn={0373-2436},
  review={\MR {0419546}},
}

\bib{Pareigis68}{article}{
  author={Pareigis, Bodo},
  title={Kohomologie von $p$-Lie-Algebren},
  language={German},
  journal={Math. Z.},
  volume={104},
  date={1968},
  pages={281--336},
  issn={0025-5874},
  review={\MR {289591}},
  doi={10.1007/BF01110334},
}

\bib{Priddy70pri}{article}{
  author={Priddy, Stewart B.},
  title={Primary cohomology operations for simplicial Lie algebras},
  journal={Illinois J. Math.},
  volume={14},
  date={1970},
  pages={585--612},
  issn={0019-2082},
  review={\MR {270364}},
}

\bib{Quillen67}{book}{
  author={Quillen, Daniel G.},
  title={Homotopical algebra},
  series={Lecture Notes in Mathematics},
  number={43},
  publisher={Springer-Verlag, Berlin-New York},
  date={1967},
  pages={iv+156 pp. (not consecutively paged)},
  review={\MR {0223432}},
}

\bib{Quillen70}{article}{
  author={Quillen, Daniel},
  title={On the (co-) homology of commutative rings},
  conference={ title={Applications of Categorical Algebra}, address={Proc. Sympos. Pure Math., Vol. XVII, New York}, date={1968}, },
  book={ publisher={Amer. Math. Soc., Providence, R.I.}, },
  date={1970},
  pages={65--87},
  review={\MR {0257068}},
}

\bib{Rezk96}{book}{
  author={Rezk, Charles W.},
  title={Spaces of algebra structures and cohomology of operads},
  note={Thesis (Ph.D.)--Massachusetts Institute of Technology},
  publisher={ProQuest LLC, Ann Arbor, MI},
  date={1996},
  pages={(no paging)},
  review={\MR {2716655}},
}

\bib{Roby65}{article}{
  author={Roby, Norbert},
  title={Les alg\`ebres \`a puissances divis\'{e}es},
  language={French},
  journal={Bull. Sci. Math. (2)},
  volume={89},
  date={1965},
  pages={75--91},
  issn={0007-4497},
  review={\MR {193127}},
}

\bib{Ronco02}{article}{
  author={Ronco, Mar\'{\i }a},
  title={Eulerian idempotents and Milnor-Moore theorem for certain non-cocommutative Hopf algebras},
  journal={J. Algebra},
  volume={254},
  date={2002},
  number={1},
  pages={152--172},
  issn={0021-8693},
  review={\MR {1927436}},
  doi={10.1016/S0021-8693(02)00097-2},
}

\bib{Soublin87}{article}{
  author={Soublin, Jean-Pierre},
  title={Puissances divis\'{e}es en caract\'{e}ristique non nulle},
  language={French},
  journal={J. Algebra},
  volume={110},
  date={1987},
  number={2},
  pages={523--529},
  issn={0021-8693},
  review={\MR {910402}},
  doi={10.1016/0021-8693(87)90064-0},
}

\bib{StradeF88}{book}{
  author={Strade, Helmut},
  author={Farnsteiner, Rolf},
  title={Modular Lie algebras and their representations},
  series={Monographs and Textbooks in Pure and Applied Mathematics},
  volume={116},
  publisher={Marcel Dekker, Inc., New York},
  date={1988},
  pages={x+301},
  isbn={0-8247-7594-5},
  review={\MR {0929682}},
}

\bib{Tamarkin98}{article}{
  author={Tamarkin, Dmitry E.},
  title={Another proof of M. Kontsevich formality theorem},
  date={1998},
  eprint={arXiv:math/9803025},
  status={Preprint},
}

\bib{Waterhouse79}{book}{
  author={Waterhouse, William C.},
  title={Introduction to affine group schemes},
  series={Graduate Texts in Mathematics},
  volume={66},
  publisher={Springer-Verlag, New York-Berlin},
  date={1979},
  pages={xi+164},
  isbn={0-387-90421-2},
  review={\MR {547117}},
}

\end{biblist}
\end{bibdiv}
\end{document}